\documentclass[english]{article}

\usepackage[utf8]{inputenc}
\usepackage{transparent}

\usepackage{amsmath,amsfonts,amssymb,makeidx,geometry}
\usepackage[nice]{nicefrac}

\usepackage{algorithm}
\usepackage[noend]{algpseudocode}
\makeatletter
\def\BState{\State\hskip-\ALG@thistlm}
\makeatother

\usepackage[english]{babel}
\usepackage[babel]{csquotes}

\usepackage{url,amsfonts,amsmath,amssymb,amsthm,amscd,epsfig,braket,array,mathtools, dsfont}
\usepackage{emptypage,fancyhdr,enumitem,textcomp}
\usepackage{bm,booktabs,setspace,varioref,yhmath}
\usepackage[dvipsnames]{xcolor}
\usepackage{tikz-cd}
\usepackage{ stmaryrd }

\usepackage{hyperref}

\theoremstyle{definition}
\newtheorem{Definition}{Definition}[section]

\theoremstyle{plain}
\newtheorem{Theorem}[Definition]{Theorem}
\newtheorem*{theorem*}{Meta-theorem}
\newtheorem{Corollary}[Definition]{Corollary}
\newtheorem{Proposition}[Definition]{Proposition}
\newtheorem{Lemma}[Definition]{Lemma}
\theoremstyle{remark}
\newtheorem{Remark}[Definition]{Remark}

\newcommand{\N}{\mathbb N}
\newcommand{\Z}{\mathbb Z}
\newcommand{\R}{\mathbb R}

\renewcommand{\H}{\mathcal H}

\renewcommand{\phi}{\varphi}

\renewcommand{\O}{\Omega}

\newcommand{\hio}{H_0^1(\O)}

\newcommand{\jj}{l}

\newcommand\restr[2]{{
  \left.\kern-\nulldelimiterspace 
  #1 
  \right|_{#2} 
  }}

\renewcommand{\d}{\partial}
\renewcommand{\a}{\alpha}

\DeclareMathOperator*{\supp}{supp}

\newcommand{\rst}[1]{\ensuremath{{\mathbin\upharpoonright}%
\raise-.5ex\hbox{$#1$}}}

\DeclarePairedDelimiter{\norm}{\lVert}{\rVert}

{\left\lbrace\begin{array}{@{}l@{}}}%
{\end{array}\right.}

\begin{document}

\allowdisplaybreaks[4]
\setlength{\parskip}{0pt}

\newcommand{\email}[1]{\protect\href{mailto:#1}{#1}}

\title{Compressed sensing photoacoustic tomography reduces to compressed sensing for undersampled Fourier measurements}

\author{Giovanni S. Alberti\thanks{MaLGa Center, Department of Mathematics, University of Genoa, Via Dodecaneso 35, 16146 Genova, Italy
(\email{giovanni.alberti@unige.it}, \email{matteo.santacesaria@unige.it}).}
\and
Paolo Campodonico\thanks{Department of Applied Mathematics and Theoretical Physics, University of Cambridge, Wilberforce Road, Cambridge CB3 0WA, United Kingdom (\email{pc628@cam.ac.uk}).}
\and
Matteo Santacesaria\footnotemark[1]}

\maketitle

\begin{abstract}
Photoacoustic tomography (PAT) is an emerging imaging modality that aims at measuring the high-contrast optical properties of tissues by means of high-resolution ultrasonic measurements. The interaction between these two types of waves is based on the thermoacoustic effect. In recent years, many works have investigated the applicability of compressed sensing to PAT, in order to reduce measuring times while maintaining a high reconstruction quality. However, in most cases, theoretical guarantees are missing. In this work, we show that in many measurement  setups of practical interest, compressed sensing PAT reduces to compressed sensing for undersampled Fourier measurements. This is achieved by applying known reconstruction formulae in the case of the free-space model for wave propagation, and by applying the theories of Riesz bases and nonuniform Fourier series in the case of the bounded domain model. Extensive numerical simulations illustrate and validate the approach.
\end{abstract}
\providecommand{\keywords}[1]{\textbf{Keywords:} #1}
\keywords{Photoacoustic tomography, compressed sensing, Riesz bases, Riesz sequences, frames, structured sampling, nonuniform Fourier series.}

\section{Introduction}\label{chap:intro}

\subsection{Physical Principles of Photoacoustic Tomography}

Photoacoustic tomography (PAT) is a novel medical imaging modality, arguably the most advanced among the so-called \emph{hybrid} modalities \cite{PAT-2006,wang-anastasio-2015,Kuchment2015}. As its name suggests, photoacoustic tomography combines two different types of waves: electromagnetic (EM) and acoustic, i.e.\ light and ultrasound. The aim of this imaging modality is to infer the optical properties of biological tissues, which are relevant, for example, in the detection of tumours. The advantage of photoacoustic tomography with respect to other imaging techniques is twofold: first,  it is non-invasive, because the wavelengths of the involved electromagnetic and acoustic waves are not dangerous for biological tissues. Secondly, by using two types of waves, photoacoustic tomography combines their advantages, allowing to reconstruct images that feature both a high contrast and a high resolution. In particular, the high contrast is produced due to a much higher absorption of EM energy by cancerous cells, while ultrasound alone would not produce good contrast, since, from an acoustic point of view, both healthy and cancerous tissues are mostly an aqueous substance. On the other hand, the good (submillimeter) resolution is achieved thanks to the ultrasound measurements, whereas the EM transmit poorly in biological tissues. The high quality of PAT images comes precisely from the combination of both light and sound, and would not be possible with either acoustic imaging or electromagnetic imaging modalities alone.

PAT is based on the \textit{thermoacoustic effect}: when tissue is irradiated with a short pulse of electromagnetic radiation, it heats up and consequently expands. This expansion generates a pressure wave that propagates through the object and can be measured outside of latter by wide-band ultrasonic transducers. The absorbed EM energy and the initial pressure it creates are much higher in cancerous cells than in healthy tissues, since cancerous cells are much more absorbent. Thus, if one could reconstruct the initial pressure, the resulting PAT tomogram would contain highly useful diagnostic information.

\subsection{Mathematical Models}
From a mathematical point of view, the image reconstruction problem in photoacoustic tomography can be interpreted as an inverse initial value problem for the wave equation.  The unknown to be determined is the initial pressure $f$ that is generated by the thermoelastic expansion. The propagating pressure wave $p$ satisfies the wave equation and initial conditions: 
\begin{equation}\notag
\d_{tt}p -c^2\Delta p = 0, \quad \restr{p}{t=0} =f, \quad \restr{\d_t p}{t=0} = 0.
\end{equation}
The wave pressure $p$ is measured on a portion $\Gamma \subseteq \d\O$ of an acquisition surface $\d\O$, which is the boundary of a bounded open set $\O\subseteq\R^d$. The inverse problem of PAT consists of the reconstruction of $f$. We  assume that the support of $f$ is contained within $\O$ and that the speed of sound $c$ in the medium is homogeneous and equal to $1$. 
A more sophysticated model for PAT considers the velocity $c$ as non-constant \cite{agranovsky-kuchment-2007,bekhachmi-etal-2016} or as another unknown of the problem \cite{stefanov-uhlmann-2013,liu-uhlmann-2015}. 

We consider here two possible models for photoacoustic tomography, depending on \textit{where} the wave equation is satisfied: the \textit{free-space model} and the \textit{bounded domain model}. The two settings are different both from a mathematical and a practical point of view, as will be illustrated in the following sections.
\subsubsection{Free space model}
This is the most popular model studied in the literature on PAT. According to the free-space model, the pressure wavefront $p$ that propagates across the body satisfies:
\begin{equation}\label{eq:free_space}
\begin{cases}
\d_{tt}p-\Delta p = 0 & \text{in } \R^d \times (0,+\infty),\\
p(\cdot,0) =f & \text{in } \R^d,\\
\d_t p(\cdot,0) = 0 & \text{in } \R^d.
\end{cases}
\end{equation}
In this scenario, the wave keeps propagating on the whole space $\R^d$, passing through the observation surface $\Gamma \subseteq \d\O$ which is therefore \textit{not} considered as a physical barrier for the wave. This acquisition surface is merely the set of points were the ultrasound transducers are located, but they are not assumed to have any impact on the wave. From a practical point of view, this means that the transducers must be suitably designed to be small enough and not invasive, in order not to interfere considerably with the pressure wavefront $p$.\par
In the free-space model, the measurements of the transducers usually are of the form
\begin{equation}\label{eq_g_fs}
g(x,t):=p(x,t), \qquad x \in \Gamma, \ t \in [0,T],
\end{equation}
where $T>0$ is the measurement time. In this case, thanks to the link with the spherical means Radon transform, the inverse problem of the reconstruction of $f$ is now fairly well understood, and uniqueness and stability hold, provided that $\Gamma$ and $T$ are large enough. There exist several reconstruction methods, including time-reversal, closed formulae and eigenfunction expansions. The reader is referred to the review paper \cite{Kuchment2015} and to the references therein.

\subsubsection{Bounded domain model}\label{subsub:bounded}

Let us now turn to the bounded domain model \cite{Ammari2010,Kunyansky2013,2015-acosta-montaldo,2015-holman-kunyansky,2016-chervova-oksanen,Alberti2018}, which is considered in order to take into account the effect of the ultrasonic transducers on the acoustic propagation.
In this case, given $f \in \hio$, the pressure $p$ satisfies the wave equation only inside a Lipschitz bounded domain $\O \subseteq \R^d$ with Dirichlet boundary values and specified initial conditions:
\begin{equation}\label{eq:bounded_domain}
\begin{cases}
\d_{tt}p-\Delta p = 0 & \text{in } \O\times (0,+\infty),\\
p = 0 & \text{on } \d\O \times (0, +\infty), \\
p(\cdot, 0) =f & \text{in } \O, \\
\d_t p(\cdot, 0) = 0 & \text{in } \O.
\end{cases}
\end{equation}
 In this case, the wave does not propagate outside the domain $\O$, but is \textit{reflected} on its boundary $\d\O$. The Dirichlet condition $\restr{p}{\d\O \times(0, +\infty)} = 0$  models the reflection of the wave on the \textit{hard} surface $\d\O$. In this setting, one measures the normal derivative on $\Gamma \subseteq \d\O$:
\begin{equation}\label{eq_g_bd}
g(x,t):= \d_\nu p(x,t), \qquad  x \in \Gamma, t \in [0, T],
\end{equation}
where $\nu$ is the outward normal unit vector of $\d\O$. It can be shown that $g \in L^2(\Gamma\times[0,T])$ (\cite{Alberti2018}). 
It would be possible to consider Neumann boundary conditions instead of Dirichlet boundary conditions, that is, to impose $\restr{\d_\nu p}{ \d\O \times(0, +\infty) }=0$ (corresponding to a \textit{soft} surface $\partial\Omega$ and to a \textit{resonating} cavity $\Omega$). In this case, one would measure $p(x,t)$ for $(x,t) \in \Gamma \times [0,T]$, and the mathematics would be similar. We opted for the Dirichlet condition to illustrate an even more different situation than the free-space scenario, but everything extends, \textit{mutatis mutandis}, to the problem with Neumann boundary conditions.

In this context, the inverse problem with full measurements $g\mapsto f$ is nothing but the classical observability problem for the wave equation, which is a well-known problem in control theory \cite{RUSSELL-1978,LIONS-81,LIONS-1988-1,KOMORNIK-1994,LASIECKA-TRIGGIANI-2000,ERVEDOZA-ZUAZUA-2012}. Roughly speaking, uniqueness always holds provided that $T$ is large enough, and stability depends on $\Gamma$: if $\Gamma$ is large enough, then Lipschitz stability holds \cite{LIONS-1988-1,LIONS-SIAM-88} (e.g., if $\O$ is a ball, a sufficient portion is more than half of its boundary; if $\O$ is a 2D square, a sufficient portion is given by two adjacent sides, while one side does not suffice \cite{1989-grisvard}); if $\Gamma$ is simply relatively open and nonempty, then logarithmic stability holds \cite{laurent-leautaud-2019}.

\subsection{ Compressed sensing in photoacoustic tomography}
Compressed sensing (CS) is a developing area of signal processing that is less than fifteen years old and aims to improve efficiency of signal recovery \cite{Candes2006,Donoho2006,foucartmathematical}. Indeed, the goal of compressed sensing is to exploit some prior knowledge on the unknown  signal in order to reconstruct it from fewer measurements than the standard theory would require. In recent years, much research has been done on applying CS to PAT with the aim of speeding up the reconstruction by taking fewer measurements, see \cite{2009-provost-lesage,2010-guo-etal,2014-arridge-etal,sandbichler2015novel,2016-burgholzer-etal,haltmeier2016compressed,Arridge2016,betcke-etal-2017,Alberti2017,haltmeier2018sparsification,antholzer2019}. Common to most of these works is the following general model for CS PAT.

In the previous sections, the measurements of the acoustic wave were assumed to be collected by \textit{point-like} transducers. We therefore knew the values of $g(x,t) = p(x,t)$ or $g(x,t) = \d_\nu p(x,t)$ for all times $t \in [0,T]$ and for \textit{all} points $x \in \Gamma \subseteq \d\O$. 
The typical setting in CS PAT   consists of replacing these pointwise measurements with spatial averages of the form
\begin{equation}\label{eq:measurements}
g_\jj(t) =(g(\cdot, t), \Phi_\jj)_{L^2(\Gamma)}=
\begin{cases}
(p(\cdot, t), \Phi_\jj)_{L^2(\Gamma)} & \text{for the free-space model}, \\
(\d_\nu p(\cdot, t), \Phi_\jj)_{L^2(\Gamma)} & \text{for the bounded domain model},
\end{cases}
\end{equation}
where $\Phi_\jj \in L^2(\Gamma) \subseteq L^2(\d\O)$, $\jj \in L\subseteq \N$, represent the masks yielding the averages. Throughout the paper we will refer to $\Phi_\jj$ also as sensor, detector or transducer. Because of the fast sound propagation, these are usually supposed to be independent of time.
In the case when $L=\N$ and $\{\Phi_\jj\}_{\jj\in \N}$ were an orthonormal basis of $L^2(\Gamma)$, the measurements $\{g_\jj\}_{\jj\in\N}$ would correspond to the full, pointwise, measurement $g$. However, in practice, it is useful to do subsampling, and to consider fewer measurements, under the assumption that the unknown $f$ is sparse with respect to a suitable dictionary.

This setting leads to the following two questions:
\begin{enumerate}
\item How should the masks $\{\Phi_\jj\}_\jj$ be chosen? How many measurements are needed?
\item How can the reconstruction of $f$ be performed from these undersampled measurements? \label{item:question}
\end{enumerate}
 These questions are only partially addressed in the papers cited above.  
 Overall, the works that use CS for PAT can be divided into two broad categories depending on their answer to question \ref{item:question}: on the one hand, they can exploit compressed sensing to directly reconstruct the unknown initial pressure $f$ from the incomplete data $\{g_\jj\}_{\jj\in L}$. This is the so-called \textit{one-step approach} and is implemented, for instance, in \cite{2009-provost-lesage,2010-guo-etal,2014-arridge-etal,Arridge2016,Alberti2017}. In this scenario, $f$ is assumed to be sparse in a suitable basis (usually wavelets or curvelets) and the transducers are either point-like (as in the early works \cite{2009-provost-lesage,2010-guo-etal}) or binary patterns (for example, after being discretised, the $\Phi_l$'s are represented by Bernoulli or scrambled Hadamard matrices in \cite{2014-arridge-etal,Arridge2016}).
 
 On the other hand, a \textit{two-step approach} consists in first using compressed sensing to complete the partial measurements $\{g_\jj\}_{\jj\in L}$ to obtain the full data $g$, and then reconstructing the unknown $f$ from the completed data $g$ via a standard technique (usually time reversal or back-projection). This is the case, for instance, in   \cite{sandbichler2015novel,2016-burgholzer-etal,haltmeier2016compressed,betcke-etal-2017,haltmeier2018sparsification,antholzer2019}. Here, in order to be able to apply CS techniques, sparsity is enforced by designing suitable temporal transformations that sparsify the incomplete measurements. The detectors $\Phi_l$'s, after being discretised, are usually represented by binary matrices.
 
 Recently, the use of deep learning (DL) for PAT has been investigated for various tasks, see the review paper \cite{hauptmann2020deep} and the references therein. Among these, DL has been used to learn efficient regularisers that can be used as a penalty term in variational formulations of the PAT problem \cite{antholzer2019deep}, or to remove the artifacts and improve the quality of cheap and fast reconstructions \cite{antholzer2019}.

Arguably, the main difficulty of applying CS to PAT is the role of sparsity. In fact, due to the complexity of the forward wave operator, it is unclear how the sparsity assumptions on the initial data $f$ will affect the measurement $g$. For this reason, research on CS PAT has been mostly experimental in nature, and theoretical guarantees are still missing and difficult to derive. As an example, \cite{2009-provost-lesage} compares different sparsity bases for the unknown $f$ via quantitative experiments. On the other hand, two-step approaches circumvent the difficult physics of the wave equation by using compressed sensing with a different aim, namely to complete the partial data. However, these approaches need to assume sparsity of the  measurements, or enforce it by applying empirically crafted sparsifying transforms. Overall, both one- and two-step methods require a more solid theoretical understanding of the role of sparsity.

 The aim of this work is to lay a theoretical foundation to derive reconstruction algorithms that allow the use of compressed sensing in PAT with rigorous guarantees. More precisely, we give partial answers to the above questions, as we briefly summarise below.

\subsection{ Main contribution of this work}\label{sub:main}

In this work, we consider several bounded domains $\Omega$: the 2D disk and the 3D ball with $\Gamma=\partial\Omega$, both in the free-space model and in the bounded domain model, and the 2D square in the bounded domain model with measurements on either one or two sides of the boundary. In all these cases, we are able to construct sensors $\Phi_\jj$ so that the knowledge of $g_\jj$ allows for the reconstruction of a subset of the generalised Fourier coefficients of $f$ with respect to the eigenfunctions $\{\varphi_{n,k}\}_{n,k\in\N}$ of the Dirichlet Laplacian on $\Omega$ (which are naturally indexed by a double index). More formally, in all the cases listed above, we  prove a result of this type.

\begin{theorem*}
There exist sensors $\{\Phi_\jj\}_{\jj\in\N}\subseteq L^2(\Gamma)$ and reconstruction maps $\{R_\jj\}_{\jj\in\N}$ such that
\[
R_\jj\colon g_\jj\longmapsto \bigl((f,\varphi_{n,\jj})_{L^2(\Omega)}\bigr)_{n\in\N},\qquad \jj\in\N.
\]
Both the sensors and the reconstruction maps are  explicitly constructed. In fact, $\{\Phi_\jj\}_{\jj\in\N}$ is an orthonormal basis of $L^2(\Gamma)$.
\end{theorem*}

This result shows that the physical undersampled measurements $\{g_\jj\}_{\jj\in L}$ with a finite $L\subseteq\N$ yield, by using the map $R_\jj$, the quantities
\begin{equation}\label{eq:NJ}
(f,\varphi_{n,\jj})_{L^2(\Omega)},\qquad n\in\N,\jj\in L,
\end{equation}
which are undersampled generalised Fourier measurements. In other words, by suitably choosing the sensors $\Phi_\jj$, \textbf{the problem of CS PAT may be reduced to a CS problem for undersampled (generalised) Fourier measurements}, which is arguably the most studied setup in CS. More precisely, the problem of the recovery of a signal from subsampled Fourier measurements is now well understood, even in this infinite-dimensional setting \cite{Adcock2016,Adcock2017,Alberti2017}, provided that the unknown is sparse in, e.g., a wavelet basis.
However, there are two fundamental differences with respect to classical CS.
\begin{enumerate}
\item In classical CS, the measurements consist of Fourier coefficients with respect to standard complex exponentials. This is (substantially) the case with the 2D square, where the eigenfunctions of the Laplacian are trigonometric functions. However, in circular domains like the 2D disk and the 3D ball the situation is different, since the eigenfunctions are tensorised in the radial and angular variables.
\item Here, the subsampling pattern of the Fourier coefficients $(f,\varphi_{n,\jj})_{L^2(\Omega)}$ has a particular structure. More precisely, it is not possible to subsample a finite random subset of $\{(n,\jj)\in\N\times\N\}$, but the sampling is completely determined by the choice of $L\subseteq\N$, which yields the structured measurements \eqref{eq:NJ}. In other words, the subsampling pattern has the tensorised form $\N\times L$. Few CS results with structured sampling patterns appeared  recently in \cite{boyer2019compressed,ciuciu2019,adcockbrugia2020}.
\end{enumerate}

\subsection{Structure of the paper}
This work is structured as follows. In section~\ref{sec:free} the free-space model is discussed, while Section~\ref{sec:bounded} is devoted to the bounded domain model for PAT. In Section~\ref{sec:CS} we discuss an example of structured subsampling pattern arising in compressed sensing PAT. Some numerical simulations are shown in Section~\ref{sec:numerics}.
Appendix~\ref{sec:3Dball|} contains the derivation for the 3D ball in the free-space model, and in Appendix~\ref{sec:riesz} we discuss some basic facts about Riesz bases of exponentials that are used in Section~\ref{sec:bounded}.

\section{Free-space model for PAT}\label{sec:free}
\subsection{A general formula}
We recall the free-space model for PAT:
\begin{equation}\label{eq:general}
\begin{cases}
\d_{tt}p-\Delta p = 0 & \text{in } \R^d \times (0,+\infty),\\
p(\cdot,0) =f & \text{in } \R^d,\\
\d_t p(\cdot,0) = 0 & \text{in } \R^d.
\end{cases}
\end{equation}
The solution $p$ is naturally extended to an even function in the time variable, setting $p(x,t) :=p(x,-t)$ for $t<0$ and  $x \in \R^d$.

In the following, we  apply the Fourier transform both on the space variable $x$ and in the time variable $t$. With an abuse of notation, they are denoted by the same symbol $\hat{\cdot}$, and the relevant variable will be clear from the context.

Our first result is a general formula that allows for the  recovery of the scalar products of $f$ with a family of functions $\psi_\rho$ from PAT measurements made with an arbitrary sensor $\Phi$.

\begin{Theorem}\label{thm:moment}
Let $\O \subseteq \R^d$ be a bounded Lipschitz domain and $f \in \hio$ be such that $\hat f \in L^1(\R^d)$. Let $p$ be the solution to problem \eqref{eq:general} and $\Phi \in L^2(\Gamma) \subseteq L^2(\d\O)$. Define 
\begin{equation}\notag
g(t) = (p(\cdot,t),\Phi)_{L^2(\Gamma)}=\int_{\Gamma} p(x,t) \overline{\Phi(x)} \, d\sigma(x), \qquad t \in \R.
\end{equation}
Then
\begin{equation}\label{eq:general_moment}
(f, {\psi}_{\rho})_{L^2(\Omega)} =  \hat g(\rho),\qquad \rho \in \R,
\end{equation}
where
\begin{equation}\label{eq:psi}
\begin{aligned}
{\psi}_{\rho}(y) =  \pi \rho^{\frac{d}{2}}\int_{\Gamma} \Phi(x) |x-y|^{-(\frac{d}{2}-1)}  J_{\frac{d}{2}-1}(2\pi \rho |x-y|)\,d\sigma(x),
\end{aligned}
\end{equation}
and $J_{\frac{d}{2}-1}$ is the Bessel function of the first kind of order $\frac{d}{2}-1$.
\end{Theorem}

\begin{proof}
By applying the spatial Fourier transform  on the functions $p$ and $f$, we obtain:
\begin{equation}\notag
\begin{cases}
\hat{p}''(\xi,t) + 4 \pi^2 |\xi|^2 \hat{p}(\xi,t) = 0, \\
\hat{p}(\xi,0) = \hat{f}(\xi), \quad \hat{p}'(\xi,0) = 0.
\end{cases}
\end{equation}
For each fixed $\xi \in \R^d$, the previous system is a second-order ordinary Cauchy problem in the time variable, whose solution is
\begin{equation}\notag
\hat{p}(\xi,t) = \hat{f}(\xi)\cos(2\pi|\xi|t).
\end{equation}

The measured data $g$ is at each time $t \in \R$: 
\begin{equation}\notag
\begin{aligned}
g(t) & = \int_{\Gamma} p(x,t) \overline{\Phi(x)} \, d\sigma(x) \\
& = \int_{\Gamma}\Bigg(\int_{\mathbb{R}^d}\hat{p}(\xi,t)e^{2\pi i \xi \cdot x} d \xi\Bigg)\  \overline{\Phi(x)}\, d\sigma(x) \\
& = \int_{\mathbb{R}^d} \hat{p}(\xi,t) \underbrace{\Bigg( \int_{\Gamma}  \overline{\Phi(x)} e^{2\pi i \xi \cdot x} d\sigma(x) \Bigg)}_{c(\xi)} d\xi  \\
& = \int_{\mathbb{R}^d} \hat{f}(\xi)\cos(2\pi |\xi| t) \, c(\xi) \, d\xi .
\end{aligned}
\end{equation}
{The use of the inversion formula is justified by the hypothesis $\hat{f} \in L^1(\R^d)$, and the integrals were interchanged thanks to Fubini's theorem since $\O$ is bounded and therefore $\Phi \in L^2(\d\O) \subseteq L^1(\d\O)$.}\par
Let us switch to spherical coordinates: writing $\xi = \rho \omega$ with $\rho \in (0,+\infty)$ and $\omega \in \mathbb{S}^{d-1}$, one has
\begin{equation*}
g(t)  = \int_0^{+\infty} \cos(2 \pi \rho t) \underbrace{\Bigg(\int_{\mathbb{S}^{d-1}} \hat{f}(\rho \omega) c(\rho \omega) d\sigma(\omega) \Bigg) \rho^{d-1}}_{2h(\rho)}   \, d\rho   = 2 \int_0^{+\infty} \cos(2 \pi \rho t) \, h(\rho) \, d\rho.
\end{equation*}
Let us evenly reflect the function $h$: setting $h(\rho):=h(-\rho)$ for $\rho<0$, one has
\begin{equation*}
g(t) = \int_{\R} h(\rho)\cos(2 \pi \rho t) \, d\rho
 = \int_{\R} h(\rho)e^{2 \pi i\rho t} \, d\rho,
\end{equation*}
so that $h = \hat g$. Thus, from the definition of $h$ and using that $\supp f\subseteq \O$, we readily derive:
\begin{equation}\notag
\begin{aligned}
\hat g(\rho)& = \frac{\rho^{d-1}}{2} \int_{\mathbb{S}^{d-1}} \hat{f}(\rho \omega) c(\rho \omega) d\sigma(\omega) \\
& = \frac{\rho^{d-1}}{2} \int_{\mathbb{S}^{d-1}}\Bigg( \int_{\mathbb{R}^d}f(y)e^{-2\pi i \rho \omega \cdot y}dy \Bigg) \Bigg(\int_{\Gamma}  \overline{\Phi(x)} e^{2\pi i \rho \omega \cdot x} d\sigma(x) \Bigg) d\sigma(\omega)  \\
& =(f, {\psi}_{\rho})_{L^2(\Omega)},
\end{aligned}
\end{equation}
where
\begin{equation}\label{eq:psi1}
\psi_{\rho}(y)= \frac{\rho^{d-1}}{2} \int_{\Gamma} \Phi(x) \Bigg( \int_{\mathbb{S}^{d-1}} e^{2\pi i \rho (y-x) \cdot \omega} d\sigma(\omega) \Bigg) d\sigma(x).
\end{equation}
Let us recall the following formula reported in \cite[Appendix~B.4]{Grafakos2008}:
\begin{equation*}
\int_{\mathbb{S}^{d-1}} e^{-2\pi i \zeta \cdot \omega} d\sigma(\omega)  = \widehat{d\sigma_{\mathbb{S}^{d-1}}}(\zeta) = 2\pi |\zeta|^{-(\frac{d}{2}-1)} J_{\frac{d}{2}-1}(2\pi |\zeta|), \qquad \zeta \in \R^d.
\end{equation*}
Then, plugging in $\zeta = \rho (x-y)$, \eqref{eq:psi} follows immediately from \eqref{eq:psi1}.
\end{proof}
{We can apply the previous theorem to different masks $\Phi_\jj$ and measurements $g_\jj$, $\jj \in \N$. We rewrite formula \eqref{eq:general_moment} to make  the dependence on $\jj$ explicit:
\begin{equation}\label{eq:moment}
(f, {\psi}_{\jj, \rho})_{L^2(\Omega)} = \hat g_\jj(\rho),
\end{equation}
 where $\psi_{\jj,\rho}$ is defined as in the statement of the theorem.} The previous formula gives an explicit relation between the detector $\Phi_\jj$ and the probing function $\psi_{\jj,\rho}$ (which  depends  also on the frequency $\rho \in \R$). It is valid in any dimension $d$, for any bounded open set $\O$ and for any portion $\Gamma$ of the acquisition surface on which detectors are located.

The previous formula provides an answer to our initial question:  taking the Fourier transform of the time-dependent data $g_\jj$ corresponds to measuring the moment of the unknown pressure $f$ against the function $\psi_{\jj,\rho}$, which depends on the chosen mask $\Phi_\jj$. 
The main question is whether, upon reasonable choices of the mask $\Phi_\jj$, the corresponding probing functions $\psi_{\jj,\rho}$ will help to reconstruct the unknown function $f$.

While this question remains open for general domains, a positive answer  may be given in the particular case where the masks $\Phi_\jj$ are chosen as the normal derivatives of the Dirichlet eigenfunctions of the Laplacian in $\Omega$.
\begin{Theorem}[{\cite[Theorem~3]{agranovsky-kuchment-2007}}]\label{thm:sine}
Let $\O \subseteq \R^d$ be a bounded Lipschitz domain and $f \in \hio$ be such that $\hat f \in L^1(\R^d)$. Let $p$ be the solution to problem \eqref{eq:general} and $\Phi =\partial_\nu\psi\in L^2(\Gamma) \subseteq L^2(\d\O)$ be the normal derivative of a Dirichlet eigenfunction $\psi\in H^1_0(\Omega)$ satisfying
\begin{equation*}
\left\{
\begin{array}{ll}
-\Delta\psi=\lambda^2\psi & \text{in $\Omega$,}\\
\psi=0&\text{on $\partial\Omega$,}
\end{array}
\right.
\end{equation*}
where $\lambda^2>0$ is the corresponding eigenvalue. Define 
\begin{equation*}
g(t) = (p(\cdot,t),\Phi)_{L^2(\Gamma)}=\int_{\Gamma} p(x,t) \overline{\partial_\nu\psi(x)} \, d\sigma(x), \qquad t \ge 0.
\end{equation*}
Then $g\in L^1([0,+\infty))$ (in fact, compactly supported when $d$ is odd) and 
\begin{equation}\label{eq:sin}
(f, \psi)_{L^2(\Omega)} =  -\frac{1}{\lambda}\int_0^{+\infty}g(t)\sin(\lambda t)\,dt.
\end{equation}
\end{Theorem}
In the following subsections, we will apply these results to the 2D disk and 3D ball.

\subsection{2D Disk}\label{sub:2ddiskfree}
We  first consider the case of a two-dimensional disk, namely choose $d =2$, $\O=B_{\R^2}(0,1)$, $\Gamma=\d\O=\mathbb{S}^1$, where we identify $\mathbb{S}^1$ as the subset  of complex numbers of modulus one. 

We consider detectors of the form:
\begin{equation}\notag
\Phi_\jj(e^{i\theta}):= e^{i\jj\theta}, \qquad \theta \in [0,2\pi], \ \jj \in \Z.
\end{equation}
{From a practical point of view, this choice corresponds to using two real detectors, modeled by the trigonometric functions $\cos(\jj\theta)$ and $\sin(\jj\theta)$.}\par

\begin{Theorem}\label{thm:2Ddisk_free}
Let $\O=B_{\R^2}(0,1)$, $\Gamma=\d\O = \mathbb{S}^1$, $f \in \hio$ be such that $\hat f \in L^1(\R^2)$, $\jj\in\Z$  and $p$ be the solution to problem \eqref{eq:general} and define
\begin{equation}\notag
g_\jj(t) = (p(\cdot,t),\Phi_\jj)_{L^2(\mathbb{S}^1)}=\int_0^{2\pi} p(e^{i\theta},t) e^{-i\jj\theta} \, d\theta, \qquad t \in [0, +\infty).
\end{equation}
Then
\begin{equation}\label{eq:twoorone}
(f, \tilde{\psi}_{\jj,n})_{L^2(\Omega)} = \frac{1}{2\pi^2} \frac{ (\hat g_\jj)' (\rho_{\jj,n})}{j_{\jj,n} \,J'_\jj(j_{\jj,n})}=-J'_\jj(j_{\jj,n})\int_0^{+\infty}g_\jj(t)\sin(j_{\jj,n} t)\,dt,\qquad n \in \N_+,
\end{equation}
where $\rho_{\jj,n}=\frac{j_{\jj,n}}{2\pi}$ and  $j_{\jj,n}$ is the $n^\text{th}$ positive zero of the Bessel function of the first kind $J_\jj$ and
\begin{equation}\notag
\tilde{\psi}_{\jj,n}(r,\theta)= e^{i\jj\theta} J_\jj(j_{\jj,n}r),\qquad r\in [0,1],\;\theta\in [0,2\pi),
\end{equation}
is an eigenfunction of the Dirichlet Laplacian in $B_{\R^2}(0,1)$ in polar coordinates (see, e.g., \cite{Henrot2006}).
\end{Theorem}

As anticipated in Section~\ref{sub:main}, the coefficients  $\big((f, \tilde{\psi}_{\jj, n})_{L^2(\Omega)}\big)_{\jj \in \Z, n \in \N_+}$ can be interpreted as \textit{generalised Fourier coefficients}, since $\tilde{\psi}_{\jj, n}$ are eigenfunctions of the Dirichlet Laplacian and, up to normalisation, form an orthonormal basis of $L^2(\Omega)$. Classical compressed sensing has focused on subsampled Fourier measurements, and from this
 point of view, we are interested in taking few measurements, which corresponds to choosing few masks $\Phi_\jj$ for the detectors. We therefore have data $g_\jj$ at our disposal for few $\jj$'s. For each of these $\jj$, by using one of the two reconstruction formulas given in \eqref{eq:twoorone}, we  reconstruct the coefficients $(f, \tilde{\psi}_{\jj, n})_{L^2(\Omega)}$ for \textit{every} $n \in \N_+$, since it simply corresponds to selecting appropriate values for the free parameter $n$. This gives rise to a peculiar compressed sensing problem with a precisely structured subsampling pattern. 
 
Let us mention two issues that need to be addressed: first, from a computational point of view, computing the derivative of $\hat g_\jj$ can be problematic when the data is imprecise or corrupted by noise; this is similar to the problem that occurs with the Hankel transform method described in \cite[\S19.4.1.1]{Kuchment2015}, and could be addressed in a similar way.
Secondly, since in dimension $d=2$ the acoustic wave never leaves the open set $\O = B_{\R^2}(0,1)$ but only decays \cite{Evans}, the computation  of \eqref{eq:twoorone} is inaccurate since $g_\jj$ is  known only in the interval $[0,T]$ and not in $[0, +\infty)$, as it should be necessary.

The proof of this result makes use of the following identity.
\begin{Lemma}\label{lem:2d}
Let $\zeta= |\zeta|e_\alpha \in \mathbb{R}^2$ with $e_\alpha =(\cos\alpha,\sin\alpha)\in \mathbb{S}^1$. The following identity holds:
\begin{equation}\label{eq:lem2d}
\int_0^{2\pi} e^{-2\pi i \zeta \cdot e_\theta}e^{i\jj\theta}d\theta =  2\pi (-i)^\jj e^{i\jj\alpha} J_\jj(2\pi|\zeta|),\qquad \jj \in \Z.
\end{equation}
\end{Lemma}
\begin{proof}
We readily derive
\begin{equation*}
\int_0^{2\pi} e^{-2\pi i \zeta \cdot e_\theta}e^{i\jj\theta}d\theta  =  \int_0^{2\pi} e^{-2\pi i |\zeta| e_\alpha \cdot e_\theta}e^{i\jj\theta}d\theta   = \int_0^{2\pi} e^{-2\pi i |\zeta| \sin(\frac{\pi}{2}-\alpha + \theta)}e^{i\jj\theta}d\theta .
\end{equation*}
By making the substitution $\beta= \frac{\pi}{2}-\alpha+\theta$, so that $d\beta = d\theta$, and exploiting the $2\pi$-periodicity of the integrand function, we obtain
\begin{equation*}
\int_0^{2\pi} e^{-2\pi i \zeta \cdot e_\theta}e^{i\jj\theta}d\theta   = \int_{\frac{\pi}{2}-\alpha}^{\frac{5}{2}\pi-\alpha} e^{-2\pi i |\zeta| \sin\beta}\,e^{i\jj(\beta + \alpha - \frac{\pi}{2})}d\beta 
= e^{i\jj(\alpha-\frac{\pi}{2}  )} \int_0^{2\pi} e^{-2\pi i |\zeta| \sin\beta}\,e^{i\jj\beta}d\beta .
\end{equation*}
Finally, identity \eqref{eq:lem2d} follows immediately from the well-known  integral representation of Bessel functions \cite[eq.~(9.19)]{temme-1996}.
\end{proof}

We are now ready to prove Theorem~\ref{thm:2Ddisk_free}.

\begin{proof}[Proof of Theorem~\ref{thm:2Ddisk_free}]
The second part of \eqref{eq:twoorone} is an immediate consequence of Theorem~\ref{thm:sine}, since $\tilde{\psi}_{\jj,n}$ is a Dirichlet eigenfunction of the Laplacian in the disk with eigenvalue $j_{\jj,n}^2$ and
\[
\partial_\nu \tilde{\psi}_{\jj,n}=\partial_r \tilde{\psi}_{\jj,n}|_{r=1}=j_{\jj,n} J'_\jj(j_{\jj,n}) \Phi_\jj.
\]
It remains to prove the first identity of \eqref{eq:twoorone}.

By Theorem~\ref{thm:moment} we have
\begin{equation}\label{eq:general_moment2}
(f, {\psi}_{\jj, \rho})_{L^2(\Omega)} =  \hat g_\jj(\rho),
\end{equation}
where, by \eqref{eq:psi1}, we can write
\begin{equation}\label{eq:2d_disk2}
\begin{aligned}
\psi_{\jj, \rho}(y) & = \frac{\rho}{2}   \int_{\Gamma} \Phi_\jj(x) \int_{\mathbb{S}^1}e^{2\pi i\rho(y-x)\cdot\omega} \, d\sigma(\omega) d\sigma(x)  \\
& =\frac{\rho}{2} \int_{\mathbb{S}^1} e^{2\pi i \rho y \cdot \omega} \int_{\Gamma} \Phi_\jj(x) e^{-2\pi i \rho x \cdot \omega} d\sigma(x) \, d\sigma(\omega), \qquad y \in \O.
\end{aligned}
\end{equation}

Let us obtain an explicit formula for $\psi_{\jj,\rho}$. Parametrising $\mathbb{S}^1$ in the usual way, writing $x=e^{i\theta}=e_\theta$ and $\omega = e^{i\alpha}=e_\alpha$ ($\a,\theta \in [0,2\pi]$), equation \eqref{eq:2d_disk2} becomes
\begin{equation}\label{eq:psi2}
\begin{aligned}
\psi_{\jj, \rho}(y) & = \frac{\rho}{2} \int_0^{2\pi} e^{2\pi i \rho y \cdot e_\alpha} \int_0^{2\pi} e^{-2\pi i \rho e_\theta \cdot e_\alpha}e^{i\jj\theta} d\theta d\alpha, \qquad  y \in B_{\R^2}(0,1).
\end{aligned}
\end{equation}
Using  identity \eqref{eq:lem2d} twice, writing $y = |y|e_\nu$ in \eqref{eq:psi2}, so that $- y = |y|e_{\nu+\pi}$, we have
\begin{equation}\notag
\begin{aligned}
\psi_{\jj, \rho}(y) 
& =\frac{\rho}{2}  \int_0^{2\pi} e^{2\pi i \rho y \cdot e_\alpha} \big(2\pi(-i)^\jj  e^{i\jj\alpha} J_\jj(2\pi\rho) \big) d\alpha \\
& = \pi (-i)^\jj \rho J_\jj(2\pi\rho) \int_0^{2\pi} e^{-2\pi i \rho |y| e_{\nu+\pi} \cdot e_\alpha}e^{i\jj\alpha} \, d\alpha \\
& =  (-1)^\jj 2\pi^2 \rho J_\jj(2\pi\rho) \, e^{i\jj(\nu+\pi)} J_\jj(2\pi\rho|y|) \\
& = 2\pi^2 \rho J_\jj(2\pi\rho) \, e^{i\jj\nu} J_\jj(2\pi\rho|y|) .
\end{aligned}
\end{equation}
As a consequence, by \eqref{eq:general_moment2} and L'H\^opital's rule we have
\[
\frac{1}{2\pi^2} \frac{ (\hat g_\jj)' (\rho_{\jj,n})}{j_{\jj,n} \,J'_\jj(j_{\jj,n})}=\lim_{\rho \to \rho_{\jj,n} }\frac{1}{2\pi^2}\frac{\hat  g_\jj(\rho)}{\rho \, J_\jj(2\pi\rho)}=
\lim_{\rho \to \rho_{\jj,n} } (f, y\mapsto e^{i\jj\nu} J_\jj(2\pi\rho|y|))_{L^2(\Omega)}
=(f,\tilde{\psi}_{\jj,n})_{L^2(\Omega)},
\]
since $2\pi\rho_{\jj,n}=j_{\jj,n}$.
\end{proof}

\subsection{3D Ball}

A similar result holds when the detectors are located on the surface of a sphere in $3$ dimensions.
Let $d=3$, $\O=B_{\R^3}(0,1)$, $\Gamma= \d\O = \mathbb{S}^2$. We use spherical coordinates, and write $x \in \mathbb{S}^2$ as $x=(\sin\theta\cos\phi, \sin\theta\sin\phi, \cos\theta)$ with $\theta \in [0,\pi]$ and  $\phi \in [0,2\pi]$.
We consider detectors of the form
\begin{equation}\notag
\Phi_{l,m}(x) = Y_{l}^m(x) = P_l^m(\cos\theta)e^{im \varphi},\qquad l \in \N, m \in \{-l,...,l\},
\end{equation}
where  $Y_{l}^m$ is the spherical harmonic function of degree $l$ and order $m$ and  $P_l^m$ is the associated Legendre polynomial.

\begin{Theorem}\label{thm:3Ddisk_free}
Let $\O=B_{\R^3}(0,1)$, $\Gamma=\d\O = \mathbb{S}^2$, $f \in \hio$ be such that $\hat f \in L^1(\R^3)$, $l \in \N$, $m \in \{-l,...,l\}$ and $p$ be the solution to problem \eqref{eq:general} and define
\begin{equation}\notag
g_{l,m}(t) = (p(\cdot,t),\Phi_{l,m})_{L^2(\mathbb{S}^2)}=\int_{\mathbb S^2} p(x,t) \overline{\Phi_{l,m}(x)} \, d\sigma(x), \qquad t \in [0, +\infty).
\end{equation}
Then
\begin{equation}\label{eq:twoorone3}
(f, \tilde{\psi}_{l,m, n})_{L^2(\Omega)} =\frac{(\hat g_{l,m})'(\rho_{l,m})}{4\pi j_{l+\frac{1}{2},n}^2 \, j_l'(j_{l+\frac{1}{2},n})}=-j_l'(j_{l+\frac{1}{2},n})\int_0^{+\infty}g(t)\sin(j_{l+\frac{1}{2},n} t)\,dt,
\qquad n\in\N_+,
\end{equation}
where $\rho_{l,n}=\frac{j_{l+\frac{1}{2},n}}{2\pi}$, $j_{l+\frac{1}{2},n}$ is the $n^\text{th}$ zero of the spherical Bessel function $j_l(x) = \sqrt{\frac{\pi}{2x}} J_{l+\frac{1}{2}}(x)$ (or, equivalently, the $n^\text{th}$ zero of the Bessel function $J_{l+\frac{1}{2}}$), and
\begin{equation}\notag
\tilde{\psi}_{l,m,n}(y)= j_l(j_{l+\frac{1}{2},n}|y|)\,  Y^m_l(y/|y|),\qquad y \in B_{\R^3}(0,1),
\end{equation}
is an eigenfunction of the Dirichlet Laplacian in $B_{\R^3}(0,1)$ (see, e.g., \cite{Henrot2006}).

\end{Theorem}

In 3D, thanks to Huygens' principle it is possible to precisely compute the Fourier transform of the data $g_{l,m}$ because it will vanish for sufficiently large times \cite{Evans}. However, the computational problem of taking the derivative needs to be addressed. 
As in the case of the 2D disk, also in this context of a 3D ball the coefficients $\big((f, \tilde{\psi}_{l,m, n})_{L^2(\Omega)}\big)_{l,m,n}$ can be interpreted as generalised Fourier coefficients.

The proof of this result is very similar to that of Theorem~\ref{thm:2Ddisk_free}. For the sake of completeness, it is presented in Appendix~\ref{sec:3Dball|}.

\section{Bounded domain model for PAT}\label{sec:bounded}
Let us now consider the bounded domain model for photoacoustic tomography. In this case, the pressure wave satisfies the following Cauchy problem for the wave equation with Dirichlet boundary conditions
\begin{equation}\label{eq:BDD}
\begin{cases}
\partial_{tt}p-\Delta p = 0 & \text{in }\Omega \times (0,T),\\
p = 0 & \text{on } \partial \Omega\times [0,T], \\
p(\cdot, 0) = f & \text{in } \O, \\
\partial_tp(\cdot, 0)=0 & \text{in } \Omega.
\end{cases}
\end{equation}
The solution may be written explicitly by using the orthonormal basis $\{\varphi_n\}_{n \in \N_+}$ of $L^2(\Omega)$ of Dirichlet eigenfunctions of $-\Delta$:
\begin{equation}\notag
p(x,t) = \sum_{n \in \mathbb{N_+}}(f,\varphi_n)_{L^2(\Omega)} \cos(\lambda_n t) \varphi_n(x),
\end{equation}
where $\{\lambda^2_n\}_n$ are the corresponding eigenvalues.

Recall that, in the bounded-domain model, we measure the normal derivative of the pressure wave on a portion $\Gamma \subseteq \d\O$ of the boundary of the open set. As in the previous section, we do not measure the normal derivative pointwise, but rather its moments against transducers modeled by a family of functions $\{\Phi_\jj\}_{\jj}$. We have, for $t \in [0,T]$,
\begin{equation}\label{eq:objective}
\begin{split}
g_\jj(t) &=  (\partial_\nu p(t), \Phi_\jj)_{L^2(\Gamma)}  \\
& = \sum_{n \in \mathbb{N_+}} (f,\varphi_n)_{L^2(\Omega)} (\partial_\nu \varphi_n, \Phi_\jj)_{L^2(\Gamma)} \cos(\lambda_n t).
\end{split}
\end{equation}
We shall now discuss how to choose the masks $\Phi_\jj$ and how to reconstruct (some of) the generalised Fourier coefficients $(f,\varphi_n)_{L^2(\Omega)}$, depending on the domain $\Omega$.

\subsection{2D Disk}

Let $\Omega = B_{\R^2}(0,1)$ be the unit disk in $\R^2$. We  start from equation \eqref{eq:objective}, which expresses the measurements $g_\jj$ as a series involving the eigenfunctions of the Dirichlet Laplacian. In polar coordinates $x = (\rho\cos\theta,\rho\sin\theta) \in B_{\R^2}(0,1)$, we can write eigenfunctions and eigenvalues as
\begin{equation}\notag
\begin{aligned}
&\varphi_{n,k}(\rho, \theta) = \frac{ \sqrt{2}}{|J_n'(j_{n,k})|} \, J_n(j_{n,k}\rho) \, e^{in\theta},\quad & n \in \Z, k \in \N_+,\\
&\lambda_{n,k} = j_{n,k}, & n \in \Z, k \in \N_+,
\end{aligned}
\end{equation}
where $J_n$ is the Bessel function and $j_{n,k}$ is its $k^\text{th}$ positive zero, as in $\S$\ref{sub:2ddiskfree}. We consider the measure $\frac{d\theta}{2\pi}$ on $\Gamma=\partial\Omega$.
The normal derivatives of these eigenfunctions on the boundary $\d\O$ of the disk are, since $\partial_\nu = \partial_\rho|_{\rho =1}$:
\begin{equation}\notag
\partial_\nu \varphi_{n,k}(1,\theta) =  \frac{ \sqrt{2}}{|J_n'(j_{n,k})|} \, j_{n,k} \, J_n'(j_{n,k})\,e^{in\theta} =: c_{n,k} e^{in\theta}, \qquad \theta \in [0,2\pi].
\end{equation}
If we choose, similarly to the free-space model described in $\S$\ref{sub:2ddiskfree}, the transducers $\Phi_\jj$  as $\Phi_\jj(\theta):=e^{i\jj\theta}$, then
\begin{equation}\notag
(\partial_\nu \varphi_{n,k}, \Phi_\jj)_{L^2(\partial B(0,1))} = c_{n,k} \delta_{n,\jj}, \qquad n,\jj \in \Z, k \in \N_+.
\end{equation}
So, by \eqref{eq:objective}, the measurements $g_\jj$ are, for fixed $\jj \in \N$: 
\begin{equation}\label{eq:gj2ddisk}
g_\jj(t) = \sum_{k \in \mathbb{N}_+}  c_{\jj,k} (f, \varphi_{\jj,k})_{L^2(\Omega)} \cos(j_{\jj,k} t), \qquad t\in [0,T].
\end{equation}
As in the free-space model, we now show that, for fixed $\jj$, all generalised Fourier coefficients $\{(f,\varphi_{\jj,k})_{L^2(\Omega)}\}_{k \in \N_+}$ may be reconstructed provided that $T$ is large enough. The reader is referred to Appendix~\ref{sec:riesz} for a brief review on Riesz sequences, which are the main tool used in this section.
\begin{Proposition}\label{prop:2ddiskfree}
Take $T\ge 1.01$. For every $\jj\in\Z$, the family $\{\cos(j_{\jj,k} \cdot)\}_{k\in\N_+}$ is a Riesz sequence in $L^2([0,T])$ with lower bound independent of $\jj$. More precisely, we have
\begin{equation}\label{eq:rieszdisk}
\sum_{k\in \N_+} |a_k|^2 \le \frac{2\pi\, T^2}{T^2-1.018} \, \Bigl\|\sum_{k\in \N_+} a_k \cos(j_{\jj,k} \cdot)\Bigr\|_{L^2([0,T])}^2,\qquad a\in \ell^2.
\end{equation}
Further, the sequence $(a_k)_k$ may be reconstructed from $g_\jj (t)=\sum_{k\in \N_+} a_k \cos(j_{\jj,k} t)$ as
\begin{equation}\label{eq:rec}
a_k=\bigl(g_\jj,S_\jj^{-1} (\cos(j_{\jj,k} \cdot))\bigr)_{L^2([0,T])},\qquad k\in\N_+,
\end{equation}
where $S_\jj\colon L^2([0,T])\to L^2([0,T])$ is the frame operator defined as
\[
(S_\jj g)(t)=\sum_{k\in \N_+} \bigl(g, \cos(j_{\jj,k} \cdot)\bigr)_{L^2([0,T])} \cos(j_{\jj,k} t).
\]
\end{Proposition}
We will prove this result below. Let us now see how to apply it to our problem. By applying the reconstruction formula \eqref{eq:rec} to \eqref{eq:gj2ddisk} we obtain
\begin{equation*}
(f,\varphi_{\jj,k})_{L^2(\Omega)} = \frac{\bigl(g_\jj,S_\jj^{-1} (\cos(j_{\jj,k} \cdot))\bigr)_{L^2([0,T])}}{c_{\jj,k}},\qquad k\in\N_+.
\end{equation*}
Furthermore, by \eqref{eq:rieszdisk}, the reconstruction is stable, uniformly in $\jj$, in the sense that
\[
\sum_{k\in\N_+} |(f,\varphi_{\jj,k})_{L^2(\Omega)}|^2\le \frac{0.55\, T^2}{T^2-1.018}\, \|g_\jj\|_{L^2([0,T])}^2,
\]
since $|c_{\jj,k}|^2=|\sqrt{2} \, j_{\jj,k}|^2\ge 2j_{0,1}^2/\ge \frac{2\pi}{0.55}$.

The proof of Proposition~\ref{prop:2ddiskfree} is based on the behaviour of the zeros of Bessel functions.

\begin{Lemma}[\cite{Elbert1986}]\label{lem:zeros}
Let $\nu \in \R$. The sequence $\{j_{\nu,k+1}-j_{\nu,k}\}_{k \in \N_+}$ is strictly decreasing for $|\nu|> \frac{1}{2}$, constant for $|\nu|=\frac{1}{2}$ and increasing for $|\nu| < \frac{1}{2}$. In all the three cases, the sequence converges to $\pi$ as $k \to +\infty$.
\end{Lemma}

We are now ready to prove Proposition~\ref{prop:2ddiskfree}.

\begin{proof}[Proof of Proposition~\ref{prop:2ddiskfree}]
Take $\jj\in\Z$. We want to apply Corollary~\ref{cor:cos}, part 2, with $\lambda_k = j_{\jj,k}$ for $k\ge 1$. We estimate the quantity
\begin{equation}\notag
\gamma:=\min\bigl(\,\inf_{k \in \mathbb{N_+}}(\lambda_{k+1}-\lambda_k),2\lambda_1\bigr) ,
\end{equation}
by using Lemma~\ref{lem:zeros} with $\nu=\jj$:
 if $\jj = 0$, then 
\begin{equation}\notag
\inf_{k \in \mathbb{N}_+}(\lambda_{k+1}-\lambda_{k}) = j_{0,2}-j_{0,1} \geq 3.115,\qquad 2\lambda_1=2j_{0,1}\ge 4.8;
\end{equation}
if $\jj \neq 0$, then
\begin{equation}\notag
\inf_{k \in \mathbb{N}_+}(\lambda_{k+1}-\lambda_{k}) = \pi, \qquad
2\lambda_1=2j_{\jj,1}\ge 2j_{0,1}\ge 4.8.
\end{equation}
In both cases, we have $\gamma\ge 3.115$. Thus, since $T\ge 1.01$, we have $\gamma>\pi/T$, and we can apply Corollary~\ref{cor:cos}, part 2.
We have that $\{\cos(\lambda_k \cdot)\}_{k\in\N_+}$ is a Riesz sequence in $L^2([0, T])$ with lower bound 
\begin{equation}\notag
A = \frac{1}{2\pi}\Big(1-\Big(\frac{\pi}{T\gamma}\Big)^2\Big) \ge \frac{1}{2\pi}\Big(1-\frac{1.018}{T^2}\Big).
\end{equation}
This proves \eqref{eq:rieszdisk}. The remaining parts of the result immediately follow from Proposition~\ref{prop:rec}.
\end{proof}

\subsection{3D Ball}

The case of the 3D ball is very similar to the case of the 2D disk. Let $\Omega = B_{\R^3}(0,1)$ be the unit ball in $\R^3$ and $\Gamma=\partial\Omega=\mathbb{S}^2$. In spherical coordinates $x =(\rho \sin\theta\cos\phi,\rho \sin\theta\sin\phi, \rho\cos\theta)\in B_{\R^3}(0,1)$ ($\rho \in [0,1], \theta \in [0,\pi], \phi \in [0,2\pi]$), the  eigenfunctions and eigenvalues of the Dirichlet Laplacian are
\begin{equation}\notag
\begin{aligned}
& \phi_{n,k,l}(\rho, \theta, \phi) = \frac{\sqrt{2}}{|j'(j_{n+\frac{1}{2},k})|} j_n(j_{n+\frac{1}{2},k} \, \rho) \, Y_{n}^l(\theta, \varphi), & & \quad n \in \N, k \in \N_+, l \in \{-n,...,n\},\\
& \lambda_{n,k}^2 = j_{n+\frac{1}{2},k}^2, & & \quad n \in \N, k \in \N_+,
\end{aligned}
\end{equation}
where $j_n(x) = \sqrt{\frac{\pi}{2x}} J_{n+\frac{1}{2}}(x)$ is the $n^\text{th}$ spherical Bessel function and $Y_{n}^l$ is the spherical harmonic of degree $n$ and order $l$.

The normal derivatives of these eigenfunctions on the boundary $\d\O$ are, since $\partial_\nu = \partial_\rho|_{\rho =1}$:
\begin{equation}\notag
\partial_\nu \varphi_{n,k,l}(1,\theta,\varphi) = \sqrt{2}j_{n+\frac{1}{2},k} \frac{j_n'(j_{n+\frac{1}{2},k})}{|j'(j_{n+\frac{1}{2},k})|}  \, Y_{n}^l(\theta, \varphi) =: c_{n,k} Y_{n}^l(\theta, \varphi) .
\end{equation}
In this case, we choose the transducers to be modeled by spherical harmonics, thus parametrized by two indices. Explicitly, we define $\Phi_{m,p}:= Y_{m}^p \in L^2(\partial B(0,1)) = L^2(\mathbb{S}^2)$ for fixed $m \in \N, p \in \{-m,...,m\}$. Thus
\begin{equation}\notag
(\partial_\nu \varphi_{n,k,l}, \Phi_{m,p})_{L^2(\partial B(0,1))} = c_{n,k}\, \delta_{n,m} \,\delta_{l,p}, \qquad n\in\N, k \in \N_+, l \in \{-n,...,n\}.
\end{equation}
In view of \eqref{eq:objective}, for $m \in \N, p \in \{-m,...,m\}$ the measurements are
\begin{equation*}
g_{m,p}(t) = \sum_{k \in \mathbb{N}_+}  c_{m,k} (f, \varphi_{m,k,p})_{L^2(\Omega)} \cos(j_{m+\frac12,k}t),\qquad t\in [0,T],
\end{equation*}
which is completely analogous to \eqref{eq:gj2ddisk}.

We can apply again Lemma~\ref{lem:zeros} and Corollary~\ref{cor:cos}, part 2, with $\lambda_k = j_{m+\frac12,k}$ for $k\ge 1$ and obtain that $\{\cos(j_{m+\frac12,k}\cdot)\}_{k \in \N_+}$ is a Riesz sequence in $L^2([0,T])$ for every $m \in \N$ provided that $T>1$. We also obtain the bound
\[
\sum_{k\in\N_+} |(f, \varphi_{m,k,p})_{L^2(\Omega)}|^2
\le\frac{1}{ 2\pi^2} \sum_{k\in\N_+} |c_{m,k}(f, \varphi_{m,k,p})_{L^2(\Omega)}|^2
\le\frac{1}{\pi}\frac{T^2}{T^2-1} \|g_{m,p}\|^2_{L^2([0,T])},
\]
where we have used also that $|c_{m,k}| = \sqrt{2} j_{m+\frac{1}{2},k} \geq  \sqrt{2} j_{\frac{1}{2},1} =  \sqrt{2} \pi$. This shows that the measurement $g_{m,p}$ allows for the  recovery of the generalised Fourier coefficients $\bigl((f, \varphi_{m,k,p})_{L^2(\Omega)}\bigr)_{k\in\N_+}$. Further, the reconstruction is stable, uniformly in $m$ and $p$.

\subsection{2D Square: reconstruction from one side}\label{sub:2Doneside}
The case where $\Omega$ is a square was examined in detail in \cite{Kunyansky2013}. There, the authors presented a reconstruction algorithm that \emph{approximates} the generalised Fourier coefficients of $f$ by considering a windowed Fourier transform of the measurements $g_{\jj}$ with respect to some chosen smooth cutoff functions. Our approach is different, since it relies on the theory of Riesz bases and does not require any arbitrary choice of auxiliar functions. Moreover, it provides an \emph{exact} reconstruction scheme of the generalised Fourier coefficients of $f$. 

Let $d=2$, $\Omega = [0,1]^2 \subseteq \mathbb{R}^2$ be the unit square in $\R^2$. The eigenfunctions of the Dirichlet Laplacian on $\O$ are
\begin{equation}\notag
\begin{aligned}
&\varphi_{n,k}(x_1, x_2) = 2 \sin(n\pi x_1) \sin(k\pi x_2), & \qquad n,k \in \N_+,\\
&\lambda_{n,k}= \pi \sqrt{n^2+k^2}, & n,k \in \N_+.
\end{aligned}
\end{equation}
Suppose we make measurements on only one side of the square. Without loss of generality, assume we measure on the vertical side $\Gamma = \{1\}\times [0,1]$. On $\Gamma$, the normal derivative of the eigenvectors is
\begin{equation*}
\partial_\nu \varphi_{n,k}(1,x_2) = \partial_{x_1} \varphi_{n,k}(1,x_2) = (-1)^n 2n\pi \sin(k\pi x_2),\qquad x_2\in[0,1].
\end{equation*}
We choose detectors $\Phi_\jj\in L^2(\Gamma)$, $\jj\in\N_+$, of the form 
\begin{equation}\label{eq:phi_square}
\Phi_{\jj}(x_2):= \sqrt{2}\sin(\jj\pi x_2) ,\qquad x_2\in [0,1],
\end{equation}
so that
\begin{equation*}
 (\partial_\nu \varphi_{n,k}, \Phi_{\jj})_{L^2(\Gamma)} =
 (-1)^n \sqrt{2}n\pi \, \delta_{k,\jj} .
\end{equation*}
Therefore, for $\jj \in \N_+$ the measurements given in \eqref{eq:objective} are of the form:
\begin{equation}\label{eq:meas-square}
\begin{split}
g_\jj(t) & =\sum_{n,k \in \mathbb{N_+}} (f,\varphi_{n,k})_{L^2(\Omega)} (\partial_\nu \varphi_{n,k}, \Phi_\jj)_{L^2(\Gamma)} \cos(\lambda_{n,k} t)\\
& =\sum_{n \in \mathbb{N_+}} (-1)^n \sqrt{2}\pi n (f,\varphi_{n,\jj})_{L^2(\Omega)}  \cos(\lambda_{n,\jj} t).
\end{split}
\end{equation}
For each $\jj\in\N_+$, the reconstruction of $\bigl((f,\varphi_{n,\jj})_{L^2(\Omega)}\bigr)_n$ from the knowledge of $g_\jj$ on $[0,T]$ may be performed as follows.

\begin{Proposition}
Take $T\ge 1$ and $\jj\in\N_+$. Then $\{\cos(\lambda_{n,\jj} \cdot)\}_{n \in \N_+}$ is a Riesz sequence in $L^2([0,T])$. In particular, there exists $C_\jj^T > 0$ such that
\begin{equation}\label{eq:rieszsquare}
\sum_{n\in \N_+} |a_n|^2 \le C_\jj^T  \Bigl\|\sum_{n\in \N_+} a_n \cos(\lambda_{n,\jj} \cdot)\Bigr\|_{L^2([0,T])}^2,\qquad a\in \ell^2.
\end{equation}
Further, the sequence $(a_n)_n$ may be reconstructed from $g_\jj (t)=\sum_{n\in \N_+} a_n \cos(\lambda_{n,\jj} t)$ as
\begin{equation}\label{eq:recsquare}
a_n=\bigl(g_\jj,S_\jj^{-1} (\cos(\lambda_{n,\jj} \cdot))\bigr)_{L^2([0,T])},\qquad n\in\N_+,
\end{equation}
where $S_\jj\colon L^2([0,T])\to L^2([0,T])$ is the frame operator.
\end{Proposition}
\begin{proof}
We apply Theorem \ref{lem:cos} to the sequences $\{\pi n\}_{n \in \N}$ and $\{\lambda_{n,\jj}\}_{n \in \N}$, setting $\lambda_{0,\jj}:=0$. Note that $\{\cos(\pi n \cdot)\}_{n \in \N}$ is a Riesz basis in $L^2([0,1])$. Furthermore,
\begin{equation*}
\sum_{n \in \N} (\lambda_{n,\jj} - \pi n)^2  = \pi^2 \sum_{n \in \N} (\sqrt{n^2+\jj^2} - n)^2  = \pi ^2 \sum_{n \in \N} \Big(\frac{\jj^2}{\sqrt{n^2+\jj^2} + n}\Big)^2 < +\infty,
\end{equation*}
Thus, the family $\{\cos(\lambda_{n,\jj} \cdot)\}_{n \in \N}$ is a Riesz basis in $L^2([0,1])$. In particular, $\{\cos(\lambda_{n,\jj} \cdot)\}_{n \in \N_+}$ is a Riesz sequence in $L^2([0,1])$, and so also in $L^2([0,T])$ since $T\ge 1$. Hence,  \eqref{eq:rieszsquare} follows. The reconstruction formula \eqref{eq:recsquare} is a consequence of Proposition~\ref{prop:rec}.
\end{proof}

As in the other cases, we can apply this result to the measurements $g_\jj$ given in \eqref{eq:meas-square} provided that $T\ge 1$. More precisely, from $g_\jj$ we first reconstruct the sequence $\bigl((-1)^n\sqrt{2}\pi n (f,\varphi_{n,\jj})_{L^2(\Omega)}\bigr)_n$ and then the sequence of Fourier coefficients $\bigl( (f,\varphi_{n,\jj})_{L^2(\Omega)}\bigr)_n$. Further, for $\jj$ fixed, the reconstruction is stable:
\[
\sum_{n\in\N_+} |(f,\varphi_{n,\jj})_{L^2(\Omega)}|^2\le 
\frac{1}{2\pi^2}\sum_{n\in\N_+} |\sqrt{2}\pi n (f,\varphi_{n,\jj})_{L^2(\Omega)}|^2\le 
\frac{C_\jj^T}{2\pi^2} \|g_\jj\|_{L^2([0,T])}^2,
\]

The main drawback of this approach is the absence of quantitative and uniform estimates on $C_\jj^T$. In fact, the stability deteriorates as $l$ grows, as shown in the following lemma.
\begin{Lemma}
Take $T\ge1$. Then
\[
\lim_{\jj\to +\infty} C^T_\jj=+\infty.
\]
\end{Lemma}
\begin{proof}
Intuitively, this follows from the fact that $\lambda_{2,\jj}-\lambda_{1,\jj}\to 0$ as $\jj\to+\infty$, and so it becomes harder and harder to distinguish the coefficients of the corresponding factors in the sum \eqref{eq:meas-square} as $l$ grows. More formally, if the constants $C^T_\jj$ were uniformly bounded in $\jj$ by a constant $C$, by using that $\{\varphi_{n,k}\}_{n,k\in\N_+}$ is an orthonormal basis of $L^2(\Omega)$, we would obtain
\[
\|f\|_{L^2(\Omega)}^2=\sum_{n,\jj\in\N_+} |(f,\varphi_{n,\jj})_{L^2(\Omega)}|^2 \le
\frac{C}{2\pi^2}\sum_{\jj\in\N_+}  \|g_\jj\|_{L^2([0,T])}^2
=\frac{C}{2\pi^2}\sum_{\jj\in\N_+}  \|(\partial_\nu p, \Phi_\jj)_{L^2(\Gamma)}\|_{L^2([0,T])}^2,
\]
so that, by using that  $\{\Phi_\jj\}_{\jj\in\N_+}$ is an orthonormal basis of $L^2(\Gamma)$,
\[
\frac{2\pi^2}C \|f\|_{L^2(\Omega)}^2\le 
\int_0^T\sum_{\jj\in\N_+}  
 |(\partial_\nu p(\cdot,t) , \Phi_\jj)_{L^2(\Gamma)}|^2\,dt
 =\int_0^T \|\partial_\nu p(\cdot,t) \|_{L^2(\Gamma)}^2\,dt =\|\partial_\nu p\|^2_{L^2(\Gamma\times[0,T])}.
\]
However, $\Gamma$ does not satisfy the geometric control condition \cite{1992-bardos-etal}, and the inverse problem is not Lipschitz stable in this case (cfr.\ $\S$\ref{subsub:bounded}).
\end{proof}

In other words, given a measuring time $T\ge 1$ and a threshold $\bar \jj\in\N_+$, the reconstruction of $\bigl( (f,\varphi_{n,\jj})_{L^2(\Omega)}\bigr)_{n\in\N_+,1\le \jj\le \bar \jj}$ is stable, but stability deteriorates as $\bar \jj\to+\infty$. More quantitatively, the following result holds.

\begin{Proposition}
Take $\bar \jj\in \N_+$ and
\begin{equation}\label{eq:TJ}
T > \frac{\sqrt{\bar \jj^2+4}+\sqrt{\bar \jj^2+1}}{3}.
\end{equation}
Then, for every $1\le \jj\le \bar \jj$ we have
\begin{equation*}
\sum_{n\in \N_+} |a_n|^2 \le 2\pi\Big(1-\Big(\frac{\sqrt{\bar \jj^2+4}+\sqrt{\bar \jj^2+1}}{3T}\Big)^2\Big)^{-1}  \Bigl\|\sum_{n\in \N_+} a_n \cos(\lambda_{n,\jj} \cdot)\Bigr\|_{L^2([0,T])}^2,\qquad a\in \ell^2.
\end{equation*}
\end{Proposition}
\begin{proof}
The distance between the square roots of two consecutive eigenvalues is
\begin{equation}\label{eq:difference}
\lambda_{n+1,\jj}-\lambda_{n,\jj} = \pi \big( \sqrt{\jj^2+(n+1)^2}-\sqrt{\jj^2+n^2} \big) =  \frac{\pi(2n+1)}{\sqrt{\jj^2+(n+1)^2}+\sqrt{\jj^2+n^2}} ,
\end{equation}
which is increasing in $n$, so that its minimum is attained at $n=1$. Thus
\begin{equation*}
\inf_{n \in \mathbb{N_+}}(\lambda_{n+1,\jj}-\lambda_{n,\jj}) = \frac{3\pi}{\sqrt{\jj^2+4}+\sqrt{\jj^2+1}} 
\le 2\pi\sqrt{\jj^2+1} 
= 2\lambda_{1,\jj}
.
\end{equation*}
The result is now an immediate consequence  of Corollary~\ref{cor:cos}, part 2.
\end{proof}

As above, we can apply this result to the measurements $g_\jj$ given in \eqref{eq:meas-square}, provided that $T$ satisfies \eqref{eq:TJ}. For $1\le \jj\le \bar \jj$, the reconstruction is uniformly stable:
\[
\sum_{n\in\N_+} |(f,\varphi_{n,\jj})_{L^2(\Omega)}|^2\le 
\frac{1}{\pi} \Big(1-\Big(\frac{\sqrt{\bar \jj^2+4}+\sqrt{\bar \jj^2+1}}{3T}\Big)^2\Big)^{-1}\|g_\jj\|_{L^2([0,T])}^2,
\]
and reconstruction may be performed using \eqref{eq:recsquare}. As a consequence, in view of CS, the Fourier coefficients recovered would be $(f,\varphi_{n,\jj})_{L^2(\Omega)}$ for $n\in\N_+$ and $\jj\in L$, where $L\subseteq\{1,\dots,\bar \jj\}$ is the subsampling pattern of the measurements. This is illustrated in Figure~\ref{fig:2dsquare1}, in which $L=\{1,2,5,8,10,14,17\}$.

 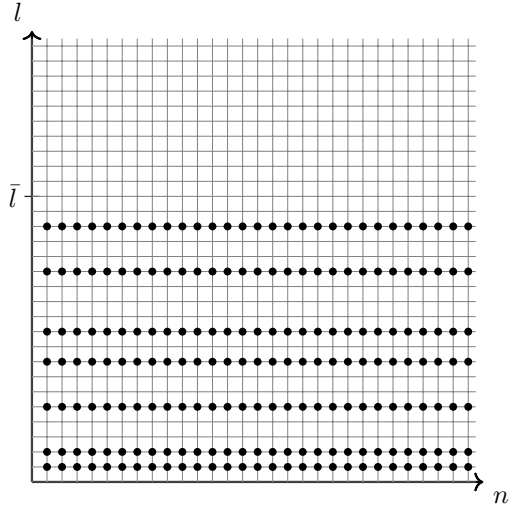
\begin{figure}
 \centering
 \begin{tikzpicture}[dotr/.style={fill=red,circle,inner sep=0pt,minimum size=4pt},dotb/.style={fill=blue,circle,inner sep=0pt,minimum size=4pt},dot/.style={fill,circle,inner sep=0pt,minimum size=3pt}]
 
 \draw[thick,->] (0,0) -- (6,0);
\draw[thick,->] (0,0) -- (0,6);
\draw[thick,->] (0,0) -- (6,0) node[anchor=north west] {$n$};
\draw[thick,->] (0,0) -- (0,6) node[anchor=south east] {$\jj$};

\node at (-.25,3.8) {$\bar l$};
 \draw[-] (-.1,3.8) -- (.1,3.8);

\draw[step=0.2cm,gray,very thin] (0,0) grid (5.9,5.9);
    
    \foreach \y in {1,2,5,8,10,14,17}
    {\foreach \x in {1,...,29} 
    {\node [dot] at (.2*\x,.2*\y) {};} 
    }

\end{tikzpicture}
\caption{Example of a subsampling pattern for the 2D square with measurements on one vertical side. The coefficients $(f,\varphi_{n,\jj})_{L^2(\Omega)}$ corresponding to the dots in the horizontal strips are recovered, with $L=\{1,2,5,8,10,14,17\}$. }
\label{fig:2dsquare1}
 \end{figure}

\subsection{2D square: reconstruction from two sides }\label{sub:2Dtwosides}

In the previous subsection, we showed that if we activate only one side, namely a vertical side, we can stably reconstruct the moments $(f, \varphi_{n,\jj})_{L^2([0,1]^2)}$ for which $\jj \leq \bar \jj$ and for any $n \in \N_+$, where $\bar \jj$ is a frequency that guarantees uniform stability for a given measuring time $T\geq1$.

As a consequence, if we activate \emph{two} adjacent sides, we can stably reconstruct the moments for which \emph{at least one} of the frequency indices ($n$ or $k$) is smaller than $\bar \jj$. What about the high frequencies, for which $n$ and $k$ are both high?
 We now present a reconstruction method that is able to recover all the coefficients $(f, \varphi_{n,k})_{L^2([0,1]^2)}$ for any pair of frequencies $n,k \in \N_+$. The stability of this reconstruction is still an open question, even though we believe it is uniform and of Lipschitz type.
 
 Recall that the instability for high frequencies was due to the fact that, by \eqref{eq:difference}, we have
 \[
 \lim_{\jj\to+\infty} \lambda_{n+1,\jj}-\lambda_{n,\jj} = 0,\qquad n\in\N_+.
 \]
The key observation here is that, if we restrict to the tail $n\ge \jj$, since the distance $ \lambda_{n+1,\jj}-\lambda_{n,\jj}$ is increasing in $n$, by \eqref{eq:difference} we have
\begin{equation*}
\inf_{n \ge \jj}(\lambda_{n+1,\jj}-\lambda_{n,\jj}) = 
\lambda_{\jj+1,\jj}-\lambda_{\jj,\jj}
=\frac{\pi(2\jj+1)}{\sqrt{\jj^2+(\jj+1)^2}+\sqrt{2}\,\jj}\ge \frac{\pi}{\sqrt{2}},\qquad \jj\in\N_+.
\end{equation*}
In other words, these distances are bounded from below uniformly in $\jj\in\N_+$. This suggests a bound of the type
\begin{equation}\label{eq:conjecture}
\sum_{n=\jj}^{+\infty} |(f,\varphi_{n,\jj})_{L^2(\Omega)}|^2\le 
C \|g_\jj\|_{L^2([0,T])}^2,\qquad \jj\in\N_+,
\end{equation}
 for some absolute constant $C>0$ (independent of $\jj$), provided that $T$ is sufficiently large (independently of $\jj$). We have not been able to rigorously prove this estimate, but our numerical experiments provided in Section~\ref{sec:numerics} strongly support this bound. Choosing a few measurements $\jj\in L_{\mathrm{ver}}$ for some finite $L_{\mathrm{ver}}\subseteq\N_+$ allows for the stable reconstruction of the Fourier coefficients
$
 (f,\varphi_{n,\jj})_{L^2(\Omega)}
$ for $\jj\in L_{\mathrm{ver}}$ and $n\ge \jj$.  By repeating the same argument for the measurements on the horizontal side $ [0,1]\times \{1\}$, we can stably recover the Fourier coefficients $
 (f,\varphi_{\jj,k})_{L^2(\Omega)}
$ for $\jj\in L_{\mathrm{hor}}$ and $k\ge \jj$. An example is shown in Figure~\ref{fig:2dsquare}, in which $L_{\mathrm{ver}}=\{1,2,5,8,10,14,17,19,25\}$ and $L_{\mathrm{hor}}=\{1,3,6,9,11,15,18,22,27\}$. In the hypothetical full-measurement setting, this approach yields a stable reconstruction method of all Fourier coefficients of $f$.

\begin{Remark}
It is worth observing that the different behaviour of the stability constants  with one or two sides is perfectly consistent with the stability of the inverse problem with full measurements discussed in $\S$\ref{subsub:bounded}:  the inverse problem is Lipschitz stable with measurements taken on two adjacent sides, but not with only one side.
\end{Remark}

 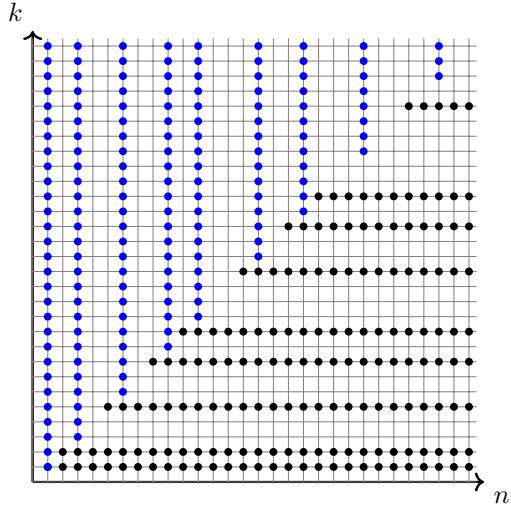
\begin{figure}
 \centering
 \begin{tikzpicture}[dotr/.style={fill=red,circle,inner sep=0pt,minimum size=4pt},dotb/.style={fill=blue,circle,inner sep=0pt,minimum size=3pt},dot/.style={fill,circle,inner sep=0pt,minimum size=3pt}]
 
 \draw[thick,->] (0,0) -- (6,0);
\draw[thick,->] (0,0) -- (0,6);
\draw[thick,->] (0,0) -- (6,0) node[anchor=north west] {$n$};
\draw[thick,->] (0,0) -- (0,6) node[anchor=south east] {$k$};

\draw[step=0.2cm,gray,very thin] (0,0) grid (5.9,5.9);
    
    \foreach \y in {1,2,5,8,10,14,17,19,25}
    {\foreach \x in {\y,...,29} 
    {\node [dot] at (.2*\x,.2*\y) {};} 
    }
    
    \foreach \y in {1,3,6,9,11,15,18,22,27}
    {\foreach \x in {\y,...,29} 
    {\node [dotb] at (.2*\y,.2*\x) {};} 
    }
\end{tikzpicture}
\caption{Example of subsampling pattern for the 2D square with measurements on two sides. The coefficients $(f,\varphi_{n,k})_{L^2(\Omega)}$ corresponding to the black dots in the horizontal strips are recovered from the measurements on the vertical side $\{1\}\times [0,1]$, while those  corresponding to the blue dots in the vertical strips from the measurements on the horizontal side $ [0,1]\times \{1\}$. }
\label{fig:2dsquare}
 \end{figure}

\section{An example of structured subsampling pattern}\label{sec:CS}

In this section, we discuss the subsampling pattern arising in the case of the 2D square with measurements on one side (see Figure~\ref{fig:2dsquare1}). This provides an example of the challenges of a compressed sensing problem with a structured subsampling pattern.

According to $\S$\ref{sub:2Doneside}, it is possible to stably reconstruct the generalised Fourier coefficients $(f, \phi_{n,\jj})_{L^2(\O)}$ for $\jj$ smaller than a chosen $\bar \jj \in \N_+$ and for all $n \in \N_+$. Our goal is to subsample the frequencies indexed by $\jj$, without losing information on the structured unknown signal $f$. We will address the problem from a more abstract point of view.

Let the sampling basis $\{\psi_{l_1,l_2}\}_{l_1,l_2 \in \N}$ be an orthonormal basis of $L^2([0,1]^2)$ made of tensorised functions:
\begin{equation}\label{eq:tensor}
\psi_{l_1,l_2}(x) = \psi_{l_1}(x_1) \psi_{l_2}(x_2), \qquad  l_1,l_2 \in \N,
\end{equation}
where $\{\psi_{l}\}_{l \in \N}$ is an orthonormal basis of $L^2([0,1])$. In our previous discussion, this was the Fourier basis, with $\psi_{l}(s)=\sqrt{2} \sin (\pi l s)$ for  $l \in \N_+$. Similarly, let $\{\varphi_{j_1,j_2} = \varphi_{j_1} \otimes \varphi_{j_2} \}_{j_1,j_2 \in \N}$ be the sparsity basis, for example a wavelet basis. The tensor symbol $\otimes$ simply means that the sparsity basis admits a representation of the form \eqref{eq:tensor}, like the sampling basis. \par
Suppose that the unknown $f \in L^2([0,1]^2)$ is sparse with respect to the sparsity basis $\{\phi_{j_1,j_2}\}_{j_1,j_2 \in \N}$. The problem is to reconstruct $f$ from measurements of the form
\begin{equation}\notag
(f,\psi_{l_1,l_2}), \qquad l_1 \in  \N, \ l_2 \in L \subseteq\N,
\end{equation}
where $L \subseteq \N$ is finite and possibly small.\par
Let us expand $f$ on the sparsity basis
\begin{equation}\notag
f = \sum_{j_1,j_2 \in \N} c_{j_1,j_2} \phi_{j_1, j_2}.
\end{equation}
Then, for every $l_1\in  \N$ and $l_2 \in L \subseteq\N$
\begin{equation}\notag
\begin{aligned}
(f,\psi_{l_1,l_2}) & = \sum_{j_1,j_2 \in \N} c_{j_1,j_2} (\phi_{j_1, j_2}, \psi_{l_1,l_2}) \\
& = \sum_{j_1,j_2 \in \N} c_{j_1,j_2} (\phi_{j_1}, \psi_{l_1})(\phi_{j_2}, \psi_{l_2}) \\
& = \Bigg(\sum_{j_1,j_2 \in \N} c_{j_1,j_2}  (\phi_{j_2}, \psi_{l_2}) \phi_{j_1}, \psi_{l_1}\Bigg).
\end{aligned}
\end{equation}
Since $\{\psi_{l_1}\}_{l_1 \in \N}$ is an orthonormal basis,  we have
\begin{equation}\notag
\sum_{j_1,j_2 \in \N} c_{j_1,j_2}  (\phi_{j_2}, \psi_{l_2}) \phi_{j_1} = \sum_{l_1\in \N}(f,\psi_{l_1,l_2})\psi_{l_1}, \qquad  l_2 \in L.
\end{equation}
And since $\{\varphi_{j_1}\}_{j_1 \in \N}$ is also an orthonormal basis, we can  reconstruct
\begin{equation}\notag
\Bigg(\sum_{j_2 \in \N} c_{j_1,j_2} \phi_{j_2}, \psi_{l_2}\Bigg)
= \sum_{l_1\in \N}(f,\psi_{l_1,l_2})(\psi_{l_1},\varphi_{j_1}),\qquad j_1 \in \N,l_2 \in L.
\end{equation}

We can reinterpret the previous argument in the following way: for each fixed $j_1 \in \N$ we measure
\begin{equation}\notag
\Bigg(\sum_{j_2 \in \N} c_{j_1,j_2} \phi_{j_2}, \psi_{l_2}\Bigg), \qquad l_2 \in L,
\end{equation}
which is a one-dimensional CS problem consisting in the reconstruction of $\sum_{j_2 \in \N} c_{j_1,j_2} \phi_{j_2}$ from partial measurements with the subsampling pattern $\{\psi_{l_2}\}_{l_2 \in L}$. Therefore, if $(c_{j_1,j_2})_{j_2 \in \N}$ is sparse, then it is possible to do the reconstruction from measurements only in $L$ by using the classical CS theory. It is worth observing that the subsampling pattern $L$ has to be chosen a priori independently of $j_1\in\N$, and in particular independently of the vector do be reconstructed. As a consequence, CS results of uniform type are needed (see, e.g., \cite{2019-chen-adcock}).

\section{Numerical simulations}\label{sec:numerics}

In this section we present numerical experiments for the bounded domain model, in the particular case of the 2D unit square. We chose to focus on the bounded domain model and not on the free-space model because the latter has been investigated more, and in particular the reconstruction method given in Theorem~\ref{thm:sine} was already known, albeit not applied in a CS setting. On the other hand, the reconstruction in the bounded domain model using the Riesz sequences is new, to the best of our knowledge. Furthermore, considering the square instead of the disk allows us to illustrate the stability issue and the dependence on the measurement surface $\Gamma$.

Two cases are considered:
\begin{enumerate}
\item full measurements are taken on one or two sides;
\item the measurements are taken on two sides and are randomly subsampled,  and a sparsity promoting algorithm is used (TV regularisation).
\end{enumerate}

\subsection{Setup and simulation of PAT data}

In all simulations, the initial value $f$ to be recovered is the Shepp-Logan phantom (see Figure~\ref{fig:shepp}). To generate the measurements $g_l(t) = (\partial_\nu p(\cdot, t), \Phi_\jj)_{L^2(\Gamma)}$ for $t \in [0,T]$ and $l \in \N_+$, we solve the wave equation with Dirichlet boundary conditions \eqref{eq:BDD} on the unit square $\Omega = [0,1]\times [0,1]\subseteq \R^2$ to obtain the solution $p$. This is done using a finite difference approximation  \cite{Linge2017}, with a square mesh of side length $2\cdot 10^{-4}$. For most simulations, the final time is $T=3$ and the interval $[0,T]$ is divided into $N$ sub-intervals $[t_{i-1}, t_i]$ of length  $t_i - t_{i-1} = \sqrt 2 \cdot 10^{-4}$ for $i = 1,\dots,N$ where $N = 21,214$. The time step is chosen in order to guarantee stability \cite{Linge2017}. The choice of $T$ ensures observability of the wave equation when we measure on two sides \cite{1989-grisvard,Alberti2018}. In general, observability is guaranteed provided that $T \geq c(\Omega)$ for some constant $c(\Omega)$ depending on the domain $\Omega$ and that $\Gamma$ is large enough; for the unit square, $c([0,1]^2) = 2\sqrt{2}$ and while two sides are enough, one side is not.

\begin{figure}
	\begin{center}
		\includegraphics[width=6cm]{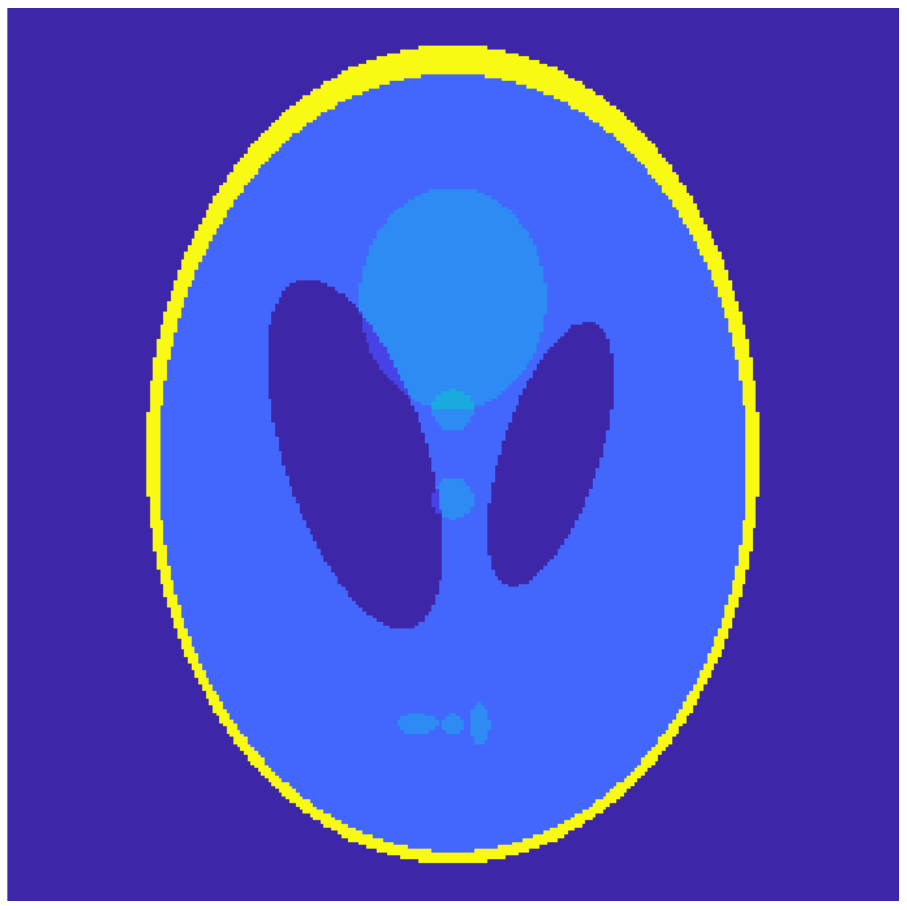}
	\end{center}
	\caption{\label{fig:shepp}Shepp-Logan phantom.}
\end{figure}

Given the solution $p$ of the wave equation, we again use finite differences to approximate its normal derivative on the boundary $\partial_\nu p(\cdot,t_i)|_{\Gamma}$ for $i = 1,\dots,N$.

We finally compute the scalar products on the boundary. We obtain two matrix approximations of the data $\{g_l(t)\}_{l\in \N_+,t\in[0,T]}$, one for the horizontal side  $\Gamma = [0,1]\times \{1\}$  and one for the vertical side $\Gamma = \{1\}\times[0,1]$, which we denote by $G^h$ and $G^v$, respectively. These are defined as:
\[
(G^h)_{l,i} = (\partial_\nu p (t_i),\Phi^h_l)_{L^2([0,1]\times \{1\})}, \quad (G^v)_{l,i} = (\partial_\nu p (t_i),\Phi^v_l)_{L^2(\{1\} \times [0,1])}
\]
for $l = 1,\dots,L_{\text{max}}$, $i= 1,\dots, N$, where we denote
\[
\Phi^h_{l}(x_1):= \sqrt{2}\sin(l\pi x_1) ,\quad x_1\in [0,1], \quad \Phi^v_{l}(x_2):= \sqrt{2}\sin(l\pi x_2) ,\quad x_2\in [0,1],
\]
and $L_{\text{max}}$ has to be much smaller than $N$ to avoid Gibbs phenomena. In our experiments we choose $L_{\text{max}} = 256$, which provides sufficiently high frequency information and it is suited to compute the fast Fourier transform in the compressed sensing reconstructions.

We add zero-mean random Gaussian noise to the matrices $G^h$ and $G^v$ to obtain noisy matrices $ \tilde G^h$ and $\tilde G^v$ with relative noise defined as
\begin{equation}\label{def:noise}
\delta_h = \frac{\max_{l,i}| (\tilde G^h)_{l,i} - (G^h)_{l,i}| }{\max_{l,i}|(G^h)_{l,i}|}, \quad \delta_v = \frac{\max_{l,i}|(\tilde  G)^v_{l,i} - (G^v)_{l,i}|}{\max_{l,i}|(G^v)_{l,i}|}.
\end{equation}
which is a discretisation of a (relative) $L^\infty([0,T])$ norm in the time variable and an  $\ell^\infty$ norm in the frequency $l$.

\subsection{Reconstruction from one and two sides without subsampling}\label{sub:rec1}

Our goal is to reconstruct the generalised Fourier coefficients $\{(f,\varphi_{n,k})_{L^2(\Omega)}\}_{n,k}$ of the phantom $f$ with respect to the eigenfunctions $\{\varphi_{n,k}\}_{n,k}$ of the Dirichlet Laplacian in the unit square. These are given by
\begin{align*}
\varphi_{n,k}(x_1, x_2) = 2 \sin(n\pi x_1) \sin(k\pi x_2),  \quad \lambda_{n,k}= \pi \sqrt{n^2+k^2}. 
\end{align*}
Here, $n,k = 1, \dots J$, where we choose $J=L_{\text{max}}=256$. 

We now illustrate how to perform such reconstruction given full measurements either on one side or on two sides, i.e., whether we are given only one of the matrices $G^h$ and $G^v$ or both of them.

\subsubsection*{One side}
Let us first consider the case where measurements are taken on one side only.
Assume we are given the measurements $G^h$ on the horizontal top side. For fixed $l \in \{1,\dots, L_{\text{max}}\}$, we solve the following system, which arises by truncating and discretising \eqref{eq:meas-square}: 
\begin{equation}\label{eq:discsys}
A_\jj f_l=(G^h)_{\jj} ,
\end{equation}
where $(G^h)_{\jj}$ is the $l$-th row of the matrix $G^h$, $f_l =\big( (f,\varphi_{n,\jj})_{L^2(\Omega)}\big)_{n=1,\dots,J} \in \R^J$ and $A_l \in \R^{N \times J}$ is defined as
\[
(A_l)_{i,n} = (-1)^n \sqrt{2}\pi n \cos(\lambda_{n,l} t_i).
\]
Solving the linear system \eqref{eq:discsys} yields the coefficients $(f,\varphi_{n,l})_{L^2(\Omega)}$ for $n=1,\dots,J$. Letting $l$ vary, we obtain all the generalised Fourier coefficients $(f,\varphi_{n,l})_{L^2(\Omega)}$ for $n,l=1,\dots,J$. However, since the observability condition is not satisfied, this reconstruction cannot be stable. Indeed, as noticed in Section~\ref{sub:2Doneside}, the reconstruction is unstable for large values of $l$ (relative to the final time $T$) and therefore part of the coefficients must be discarded. Thus, we stably reconstruct all coefficients  $(f,\varphi_{n,\jj})_{L^2(\Omega)}$ for $n=1,\dots,J$ and $\jj=1,\dots,\bar \jj$, where $\bar l \ll J$ and depends on the final time $T \geq c(\Omega)$ and on the noise level. Furthermore, we notice that - independently of $T$ - the coefficients with $l \leq n$ are stably reconstructed, while for $l \geq n$ the quality of the reconstruction depends on $T$ and quickly deteriorates as $l$ grows. Therefore, when measuring on a horizontal side, we can stably reconstruct the coefficients $\{(f,\varphi_{n,l})_{L^2(\Omega)}\}_{l \leq n}$. This observation corroborates our conjecture explained in Section \ref{sub:2Dtwosides}. 

In the very same way, if we take measurements on the vertical side (i.e.\ if we are given the matrix $G^v$), then the same considerations on stability apply by swapping the two indices and we can stably reconstruct the coefficients $(f,\varphi_{l,k})_{L^2(\Omega)}$ for $l \leq k$.

\subsubsection*{Two sides}

If we are given measurements on two sides, i.e.\ both matrices $G^h$ and $G^v$ are known, then we can combine the two previous cases and stably reconstruct \emph{all} the generalised Fourier coefficients. First, for each side we solve the family of the associated linear systems and collect the two sets of coefficients $(f,\varphi_{n,k})_{L^2(\Omega)}$. Then, from the set corresponding to the horizontal measurements, we extract only the coefficients $(f,\varphi_{n,k})_{L^2(\Omega)}$ with $k < n$, while from the vertical measurements we extract those with $k > n$. On the diagonal $k=n$, we average the two coefficients. Proceeding this way, we stably obtain all coefficients $(f,\varphi_{n,k})_{L^2(\Omega)}$ for $n,k = 1,\dots,J$.\smallskip

\subsubsection*{Reconstruction formula}

In all the three cases, after stably recovering (part of) the generalised Fourier coefficients, we can find an approximation of $f$ using the partial sum 
	\begin{equation}\label{eq:psum}
	f(x_1,x_2) = \sum_{n,k} (f,\varphi_{n,k})_{L^2(\Omega)}\varphi_{n,k}(x_1,x_2), \qquad (x_1,x_2) \in [0,1]^2,
	\end{equation}
where the indices range over the set where the coefficients are stably recovered, which depends on the side (or sides) where measurements are taken,  as explained above.

\subsection{Compressed sensing PAT}\label{sub:CS}

We next consider the problem of reconstructing \emph{all} the generalised Fourier coefficients $\{(f,\varphi_{n,k})_{L^2(\Omega)}\}_{n,k = 1,\dots,J}$ from a subset of them which has a specific structure: we consider a compressed sensing problem with generalised Fourier data where the frequencies are undersampled with the structured pattern presented in Section~\ref{sub:2Dtwosides}. For the numerical experiments we acquire PAT measurements on two sides and we consider a two-level sampling scheme as in \cite[Section 4.1]{Adcock2017}: first, we fully sample the  frequencies for which one of the indices is low, i.e.\ the frequencies $(n,k) \in (\{1,\dots,J_0\}\times \{1,\dots,J\}) \cup (\{1,\dots,J\}\times\{1,\dots,J_0\})$ for some $J_0 \ll J$; then we randomly sample single indices $n$ and $k$ from the interval $\{J_0  + 1, \dots,J\}$, following a non-uniform distribution that is more concentrated in the low frequencies (log-sampling, see \cite{Alberti2017}) and sample the Fourier coefficients indexed by the half lines $\{n\}\times\{n,\dots,J\}$ and $\{k,\dots,J\}\times\{k\}$. We call $\Theta \subset \{1,\dots,J\}^2$ the set of the indices of the sampled coefficients. See Figure~\ref{fig:samp} for an example: the white pixels denote the sampled coefficients and the axes represent the $n$ and $k$ indices.

\begin{figure}
\begin{center}
\includegraphics[width=6cm]{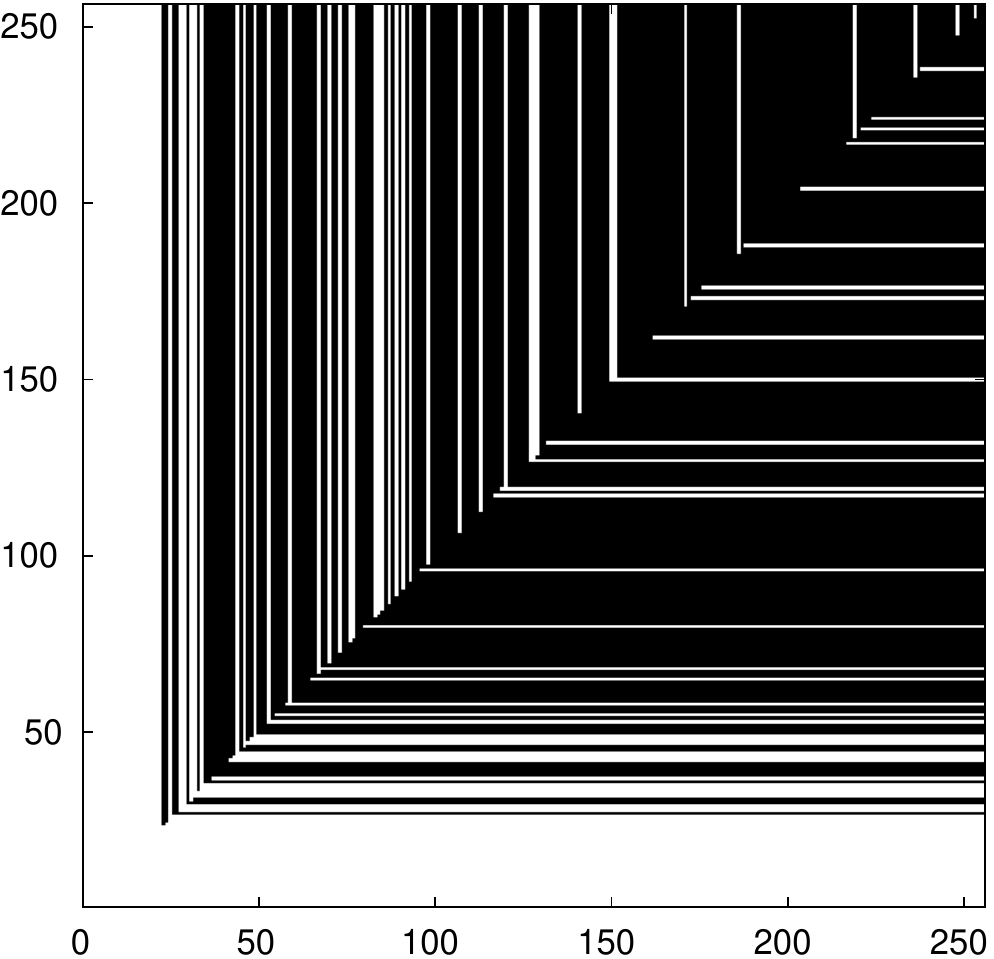}
\end{center}
\caption{\label{fig:samp} Example of our two-level subsampling scheme. The horizontal axis represents the $n$ index and the vertical axis represents the $k$ index. The white part represents the sampled coefficients. Here we fully sample both indices until $J_0 = 22$, and then $30+30$ vertical and horizontal half-lines are non-uniformly randomly sampled between $23$ and $256$.}
\end{figure}

Here, we chose not to impose sparsity constraints with respect to an orthonormal basis of tensorised functions (e.g.\ wavelets) as in Section~\ref{sec:CS}, but we assume that the unknown has a sparse gradient, which is known to give better results \cite{2013-needell-ward}. As in classical compressed sensing works \cite{Candes2006}, we use Total Variation (TV) regularisation to recover $f$. Theoretical guarantees are provided in \cite{2013-needell-ward,2015-poon}. Let $F \in \R^{J \times J}$ be the discretisation of an image in the unit square and $\| F\|_{TV}$ its discrete TV norm defined as
\[
\| F\|_{TV} = \sum_{n,k} \sqrt{|(D_1 F)_{n,k}|^2 + |(D_2 F)_{n,k}|^2 },
\]
where $D_1$ and $D_2$ are the finite differences $(D_1 F)_{n,k} := F_{n,k}-F_{n-1,k}$ and $(D_2 F)_{n,k} := F_{n,k}-F_{n,k-1}$. To recover $f$ from incomplete generalised Fourier samples we find a solution to the optimisation problem
\begin{equation}\label{eq:TVmin}
\min \|F \|_{TV} \; \text{ subject to } \hat F_{n,k} =  (f,\varphi_{n,k})_{L^2(\Omega)}, \quad \text{for } (n,k) \in \Theta,
\end{equation}
where we denoted by $\hat F_{n,k}$ the coefficients of the 2D discrete sine transform (DST):
\[
\hat F_{n,k} =\frac{2}{J+1} \sum_{n',k' = 1}^J F_{n',k'}\sin\left(\pi n \frac{n'}{J+1}\right)\sin\left(\pi k \frac{k'}{J+1}\right),
\]
which is the discretised version of  the generalised Fourier coefficients of $F$.

In order to numerically solve the convex optimisation problem \eqref{eq:TVmin}, we have adapted the Matlab code $\ell_1$-\textsc{MAGIC} \cite{candes2005l1} to our setting. In particular, we implemented the DST instead of the discrete Fourier transform: up to a zero-padding, this is done by computing the fast Fourier transform and taking its imaginary part.

\subsection{Results}

\subsubsection{Reconstruction from one and two sides without subsampling}

In Figure~\ref{fig:recfull} we present the reconstructions obtained from PAT noisy measurements taken either on one or two sides, using different noise levels with $\delta_v \approx \delta_h$.

\begin{figure}
\begin{picture}(300,420)
\put(20,300){\includegraphics[width=4.5cm]{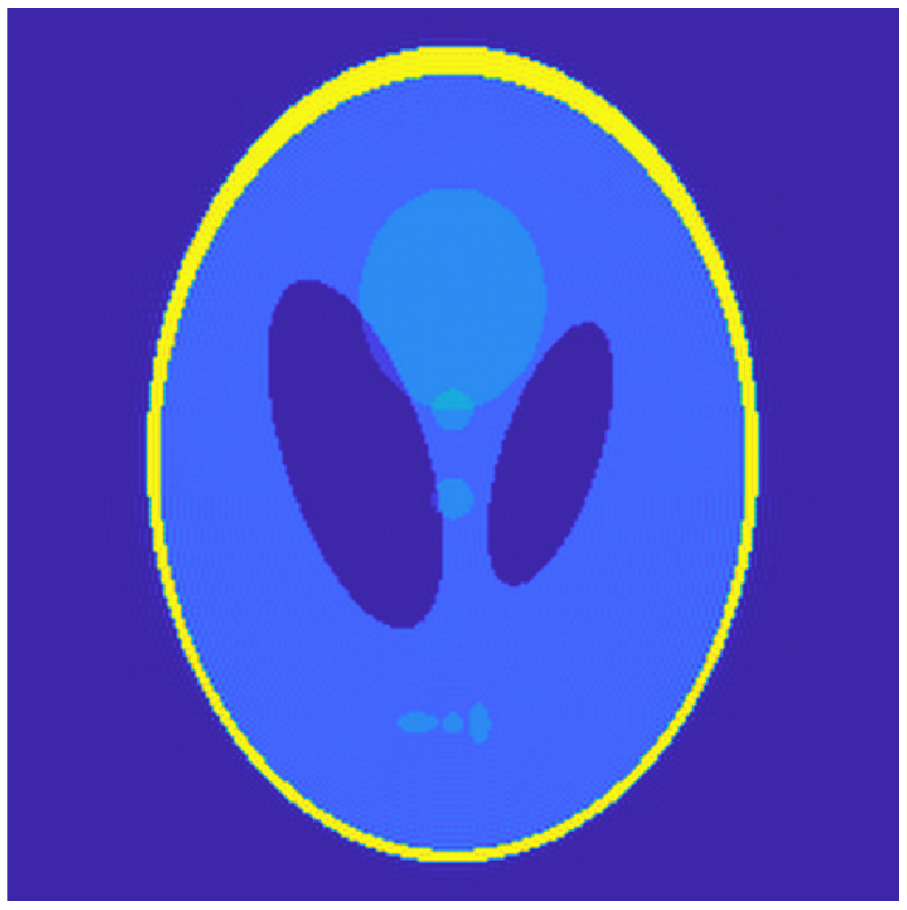}}
\put(160,300){\includegraphics[width=4.5cm]{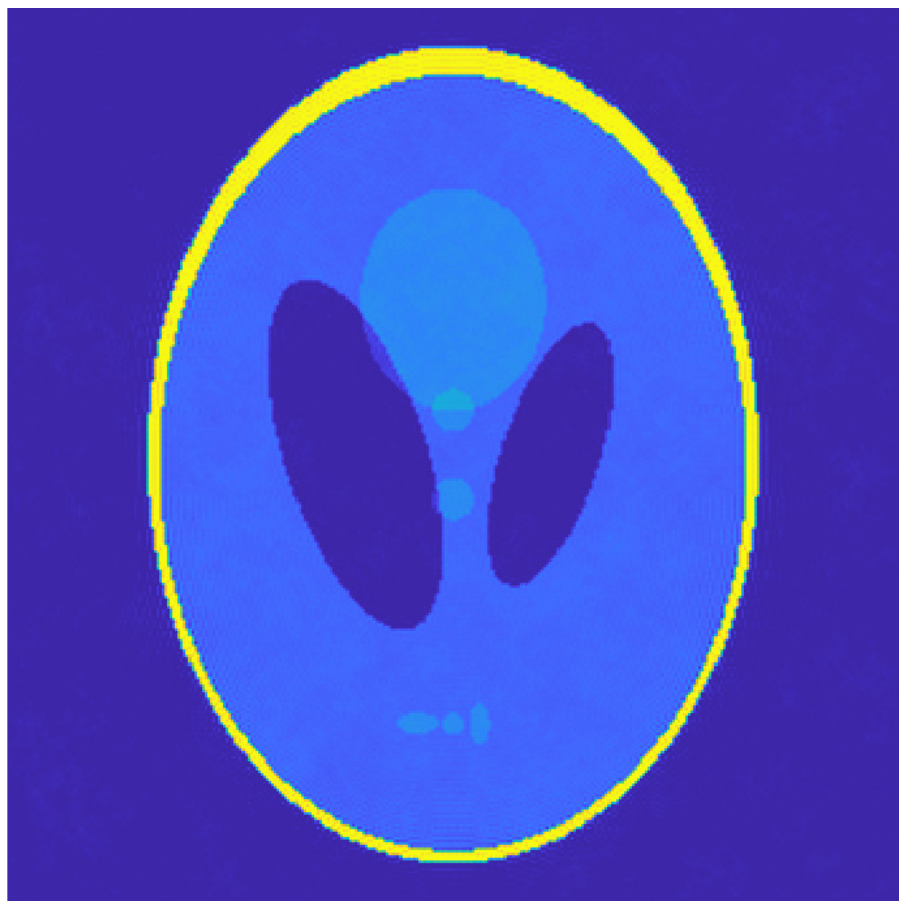}}
\put(300,300){\includegraphics[width=4.5cm]{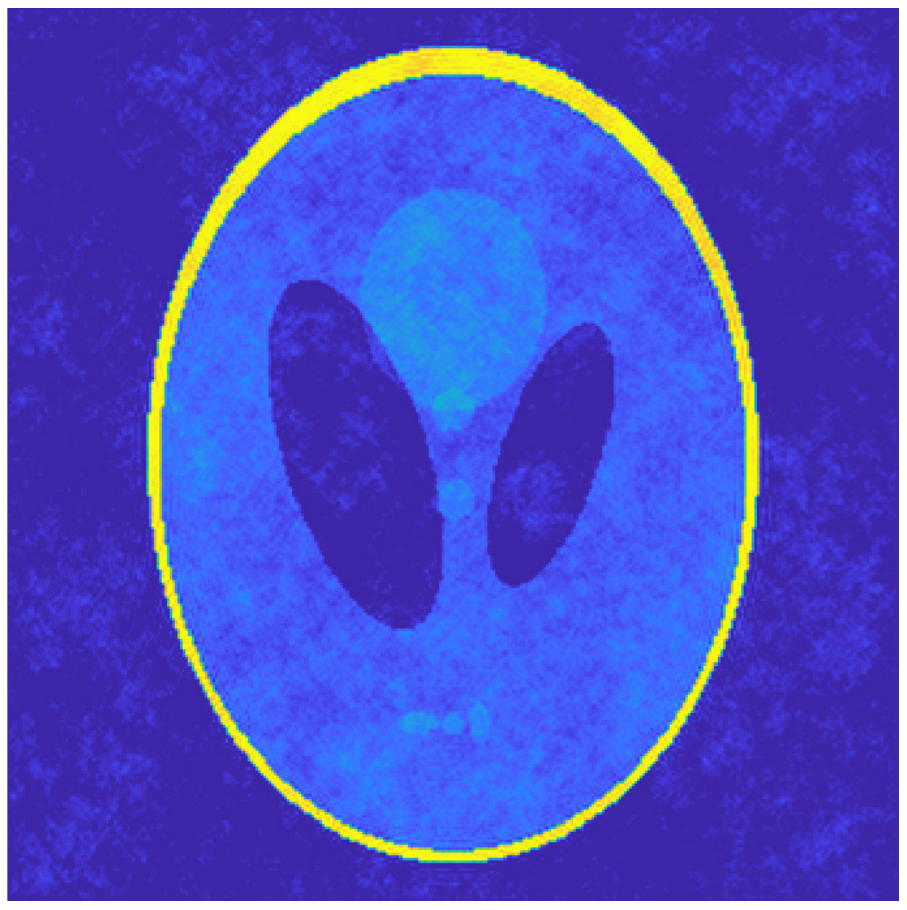}}
\put(20,160){\includegraphics[width=4.5cm]{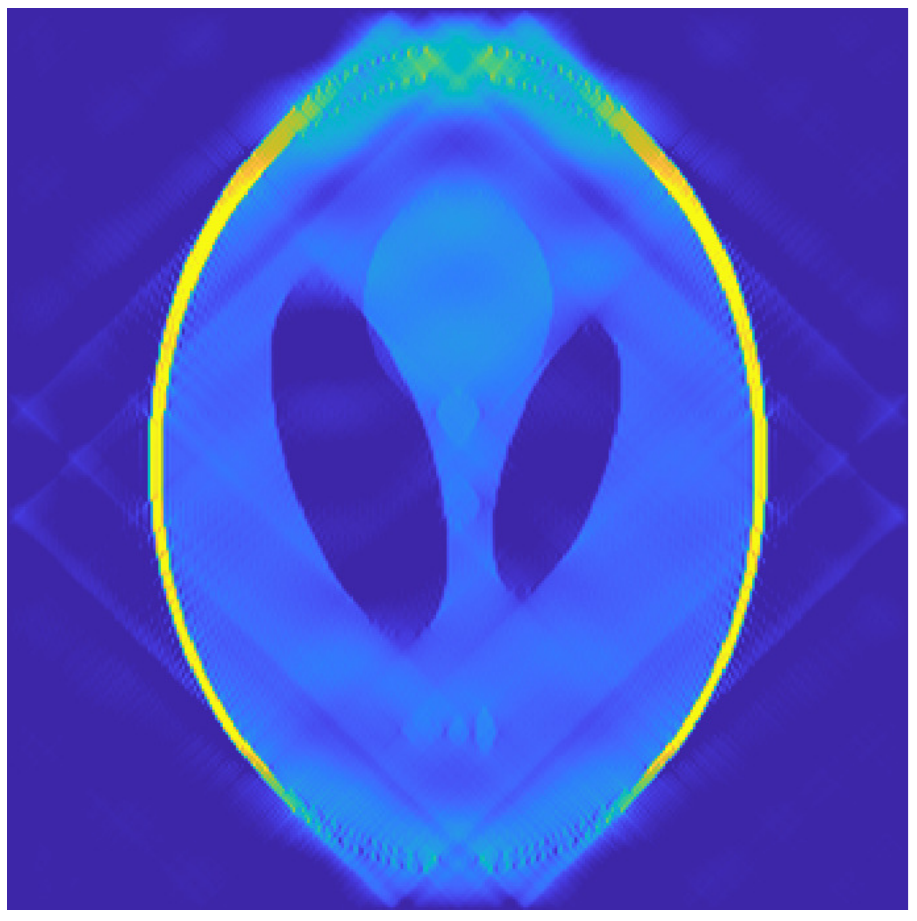}}
\put(160,160){\includegraphics[width=4.5cm]{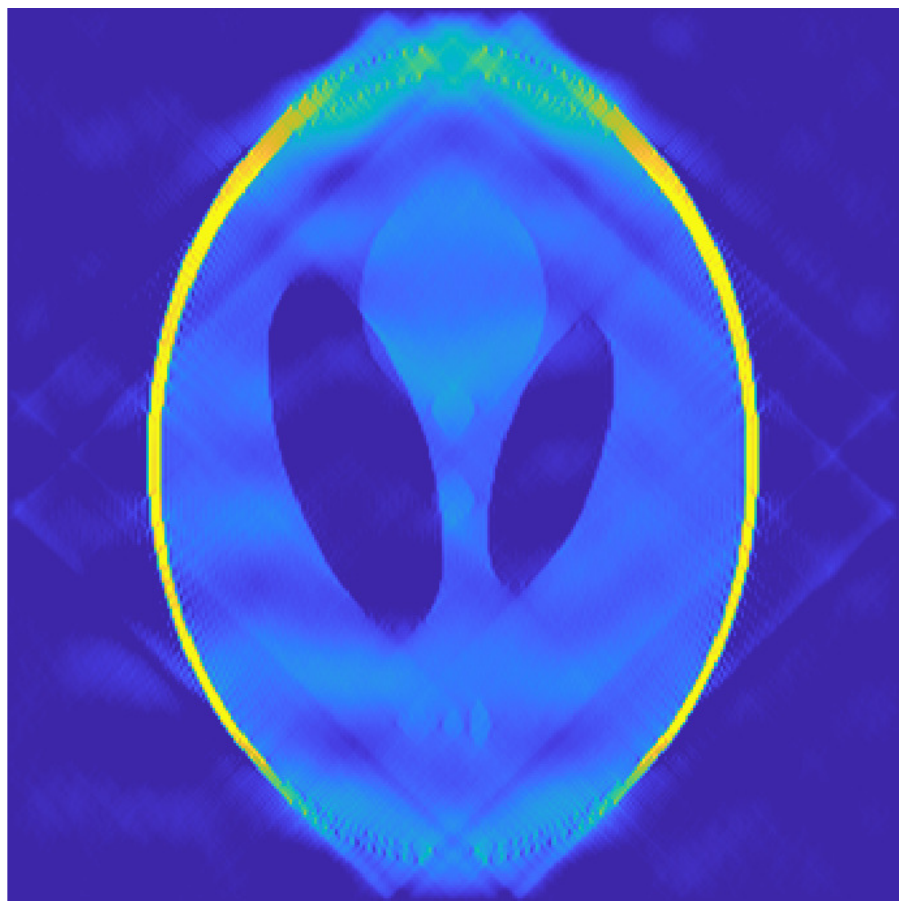}}
\put(300,160){\includegraphics[width=4.5cm]{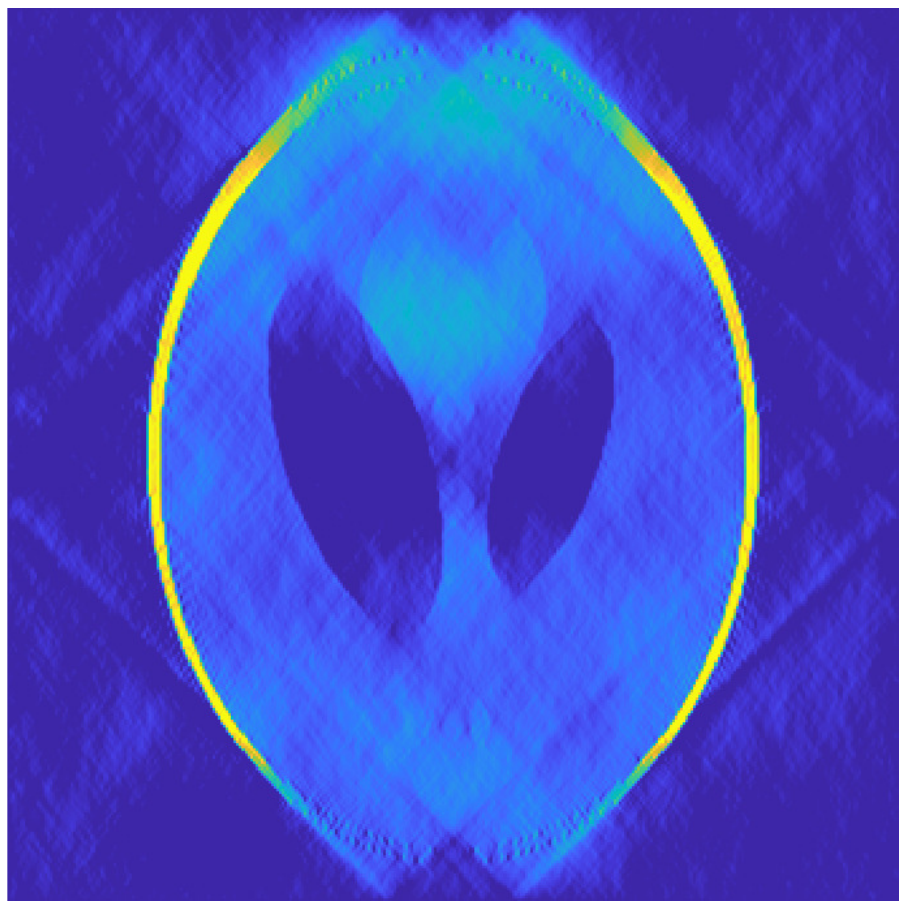}}
\put(20,20){\includegraphics[width=4.5cm]{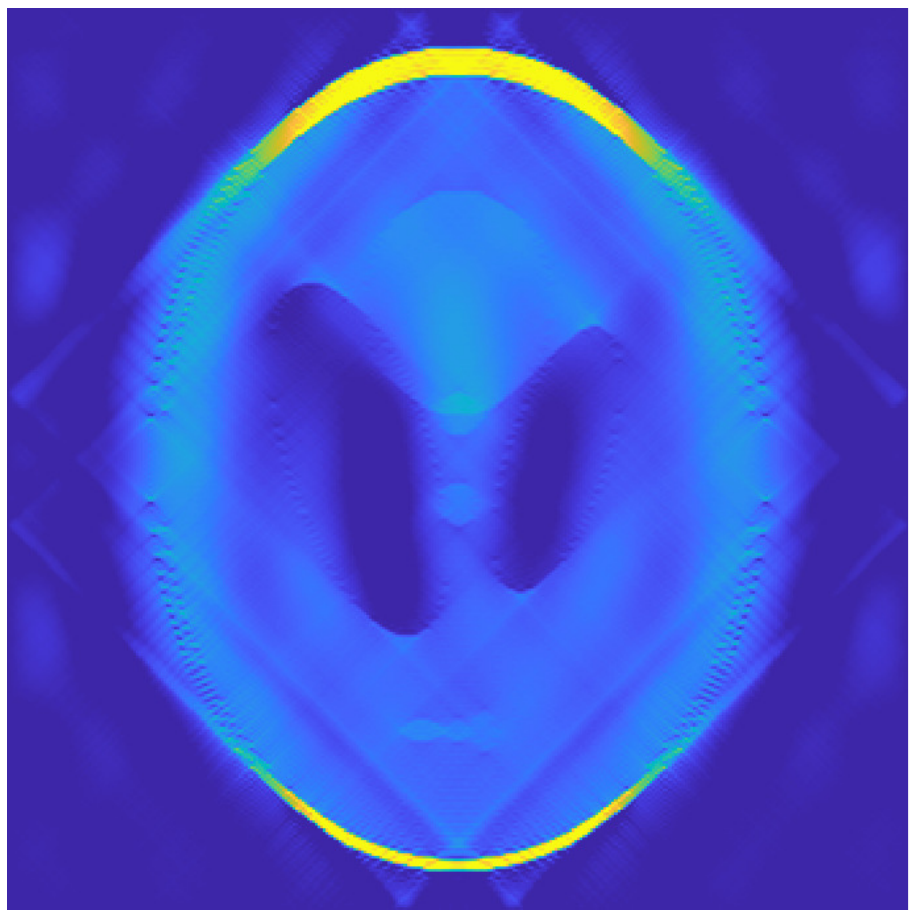}}
\put(160,20){\includegraphics[width=4.5cm]{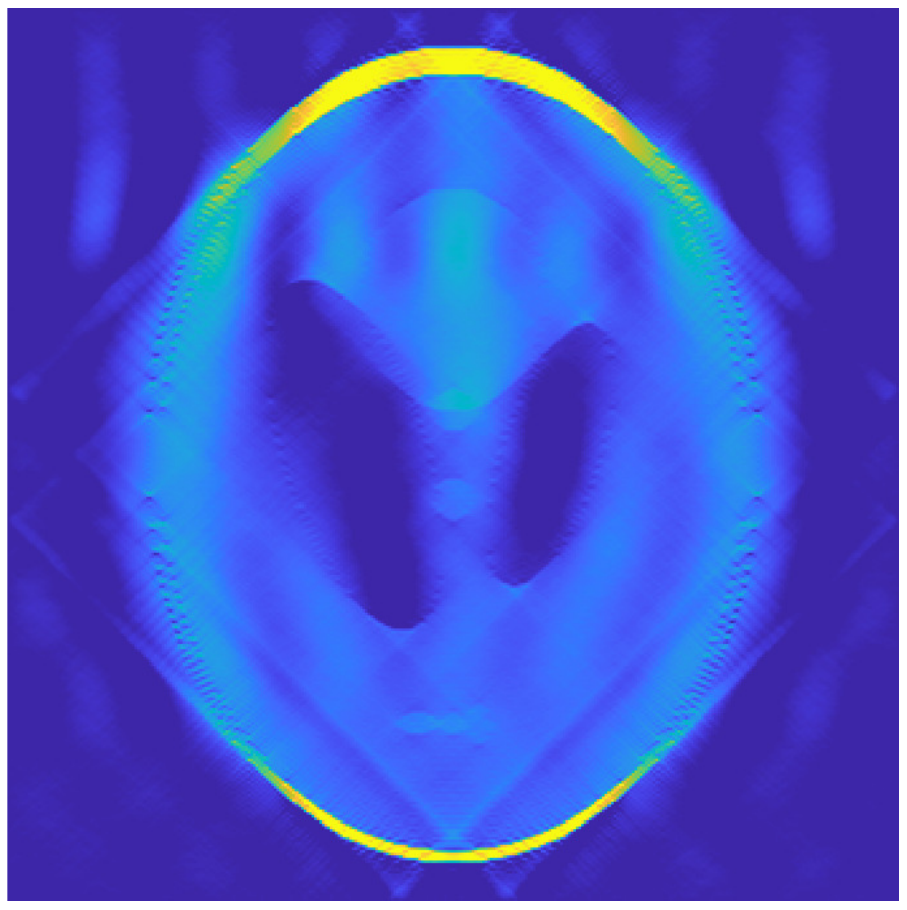}}
\put(300,20){\includegraphics[width=4.5cm]{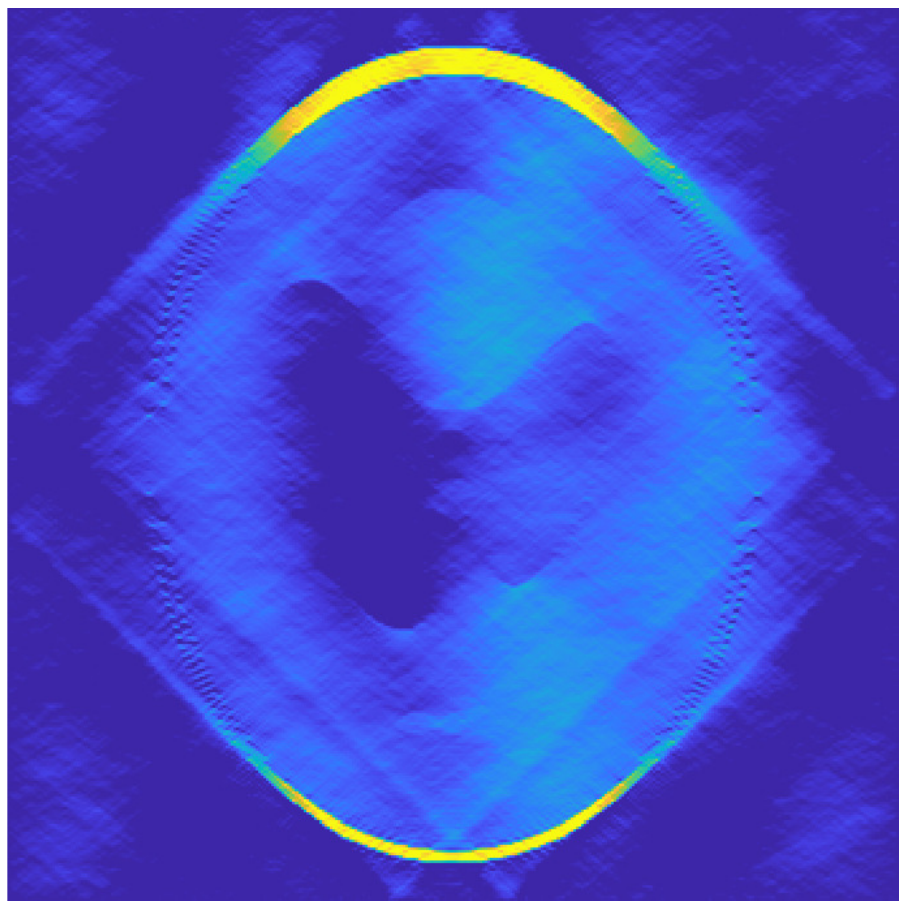}}
\put(388,294){\includegraphics[width=1.64cm]{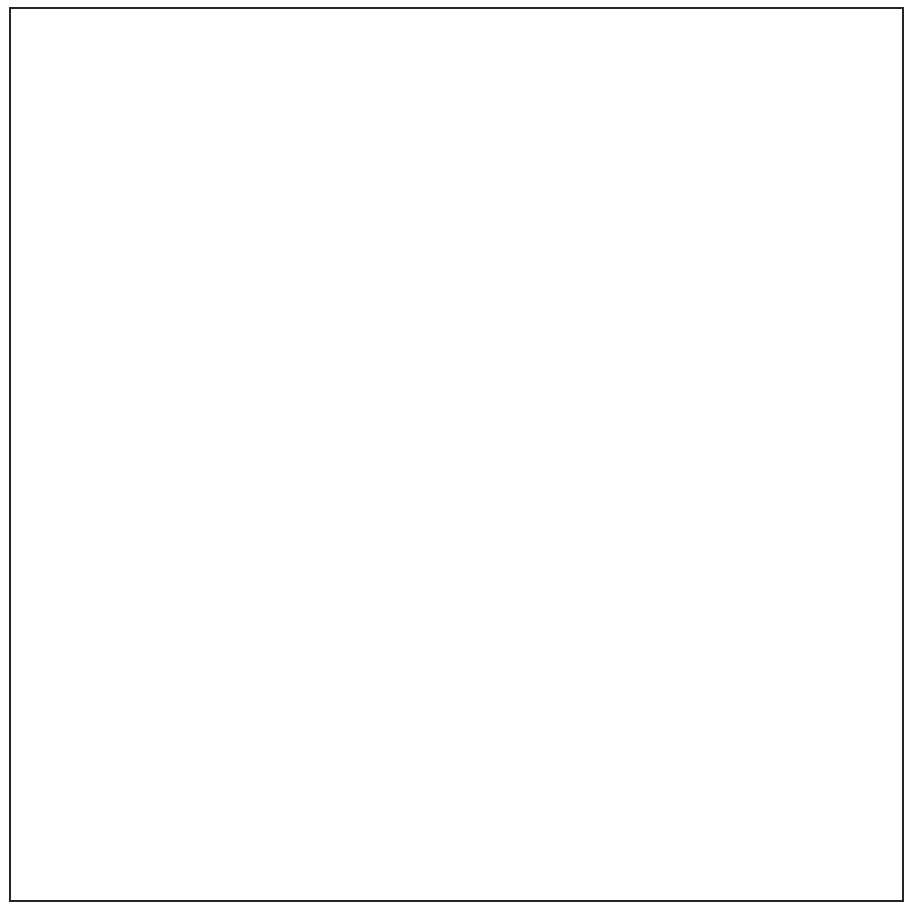}}
\put(248,294){\includegraphics[width=1.64cm]{images/schemeW.pdf}}
\put(108,294){\includegraphics[width=1.64cm]{images/schemeW.pdf}}
\put(388,154){\includegraphics[width=1.64cm]{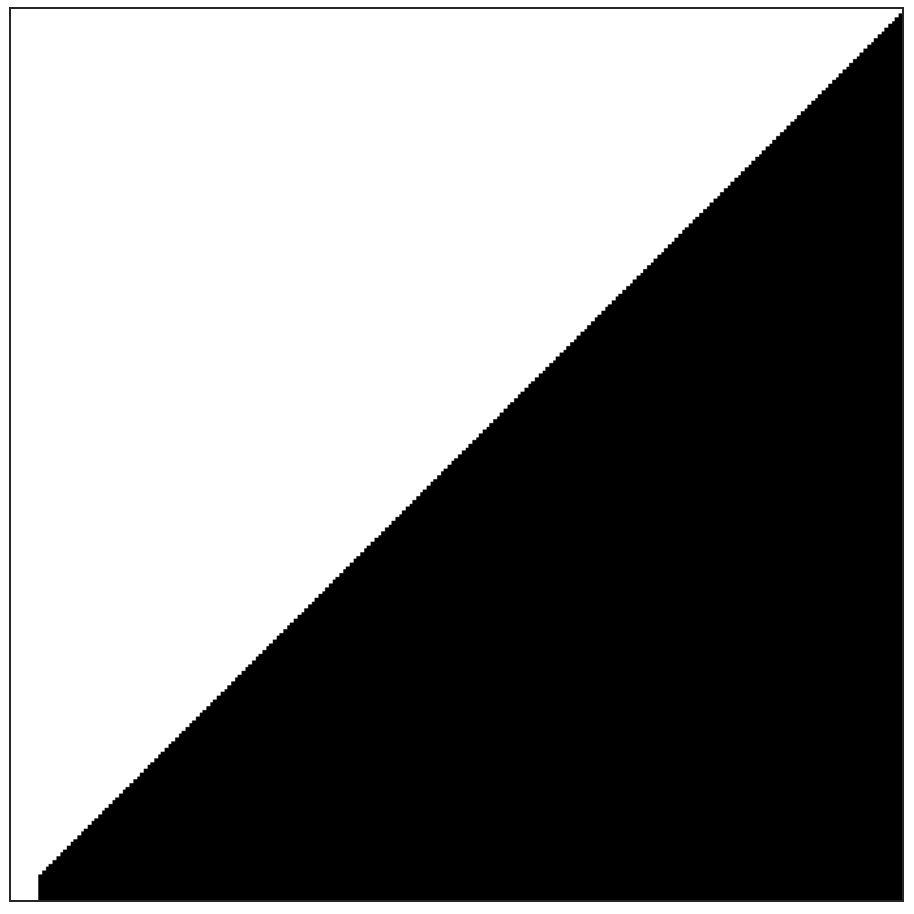}}
\put(248,154){\includegraphics[width=1.64cm]{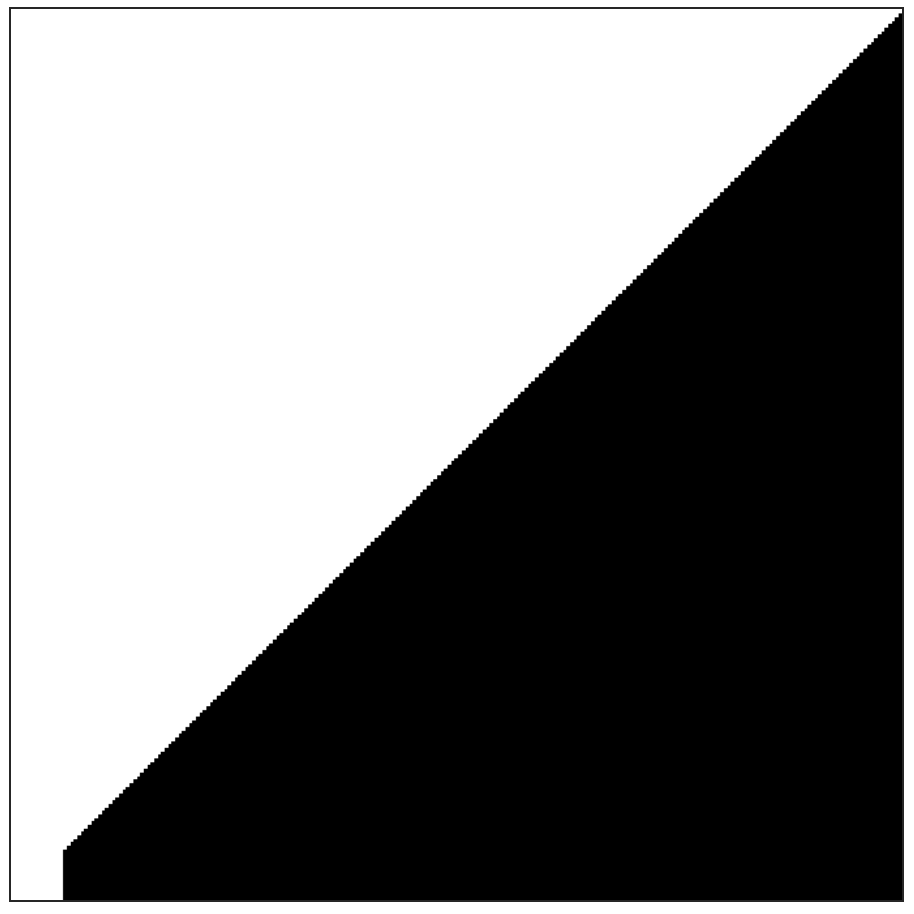}}
\put(108,154){\includegraphics[width=1.64cm]{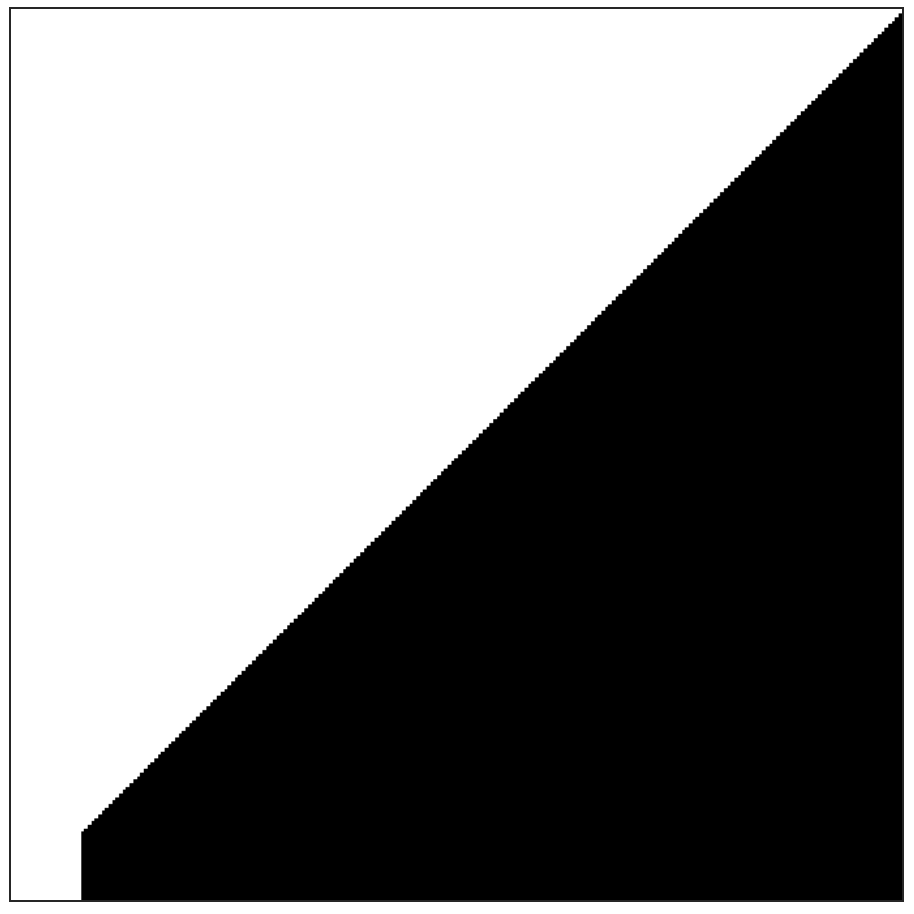}}
\put(388,14){\includegraphics[width=1.64cm]{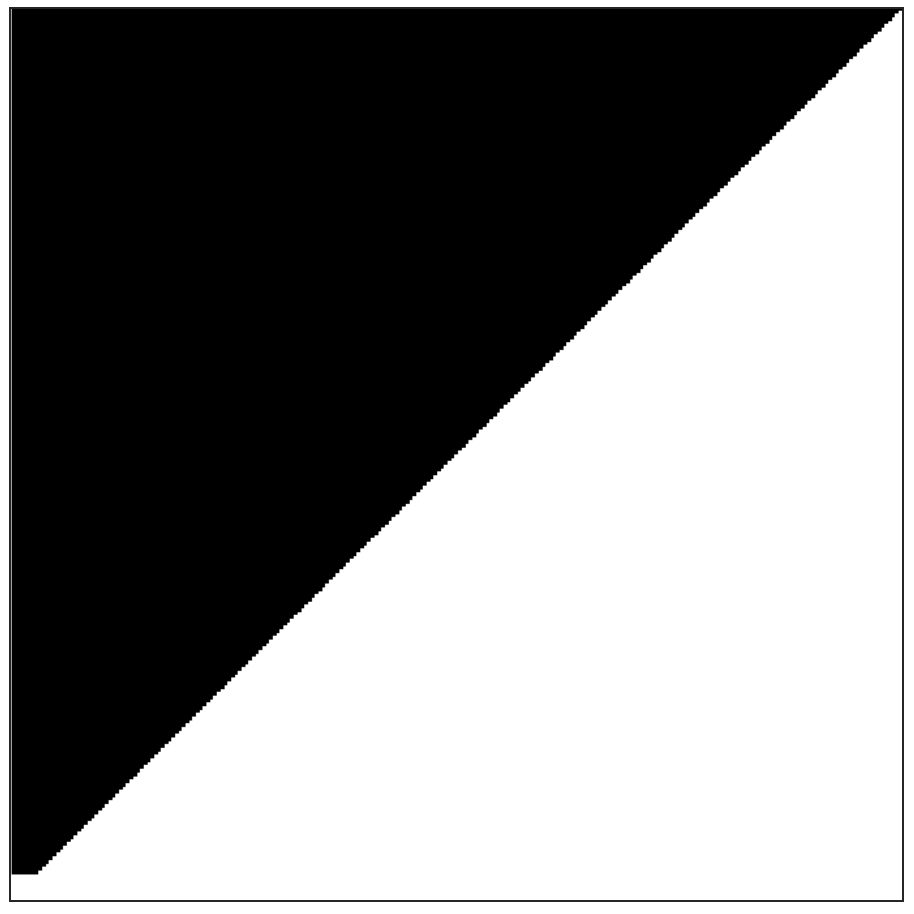}}
\put(248,14){\includegraphics[width=1.64cm]{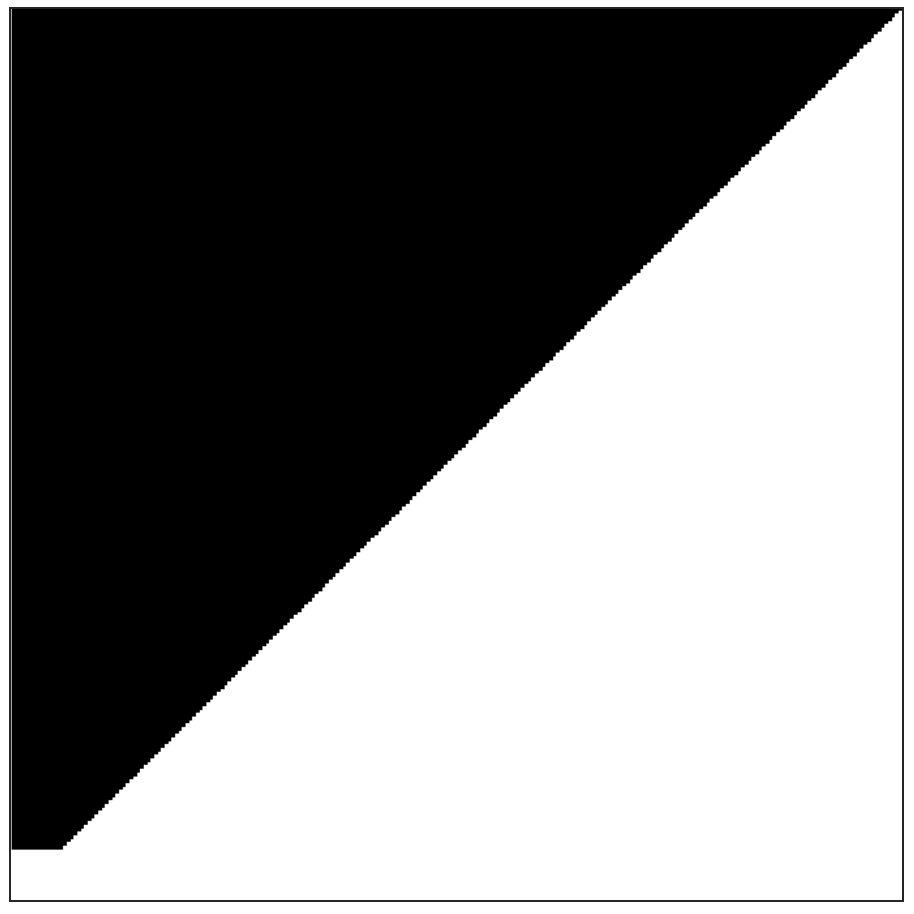}}
\put(108,14){\includegraphics[width=1.64cm]{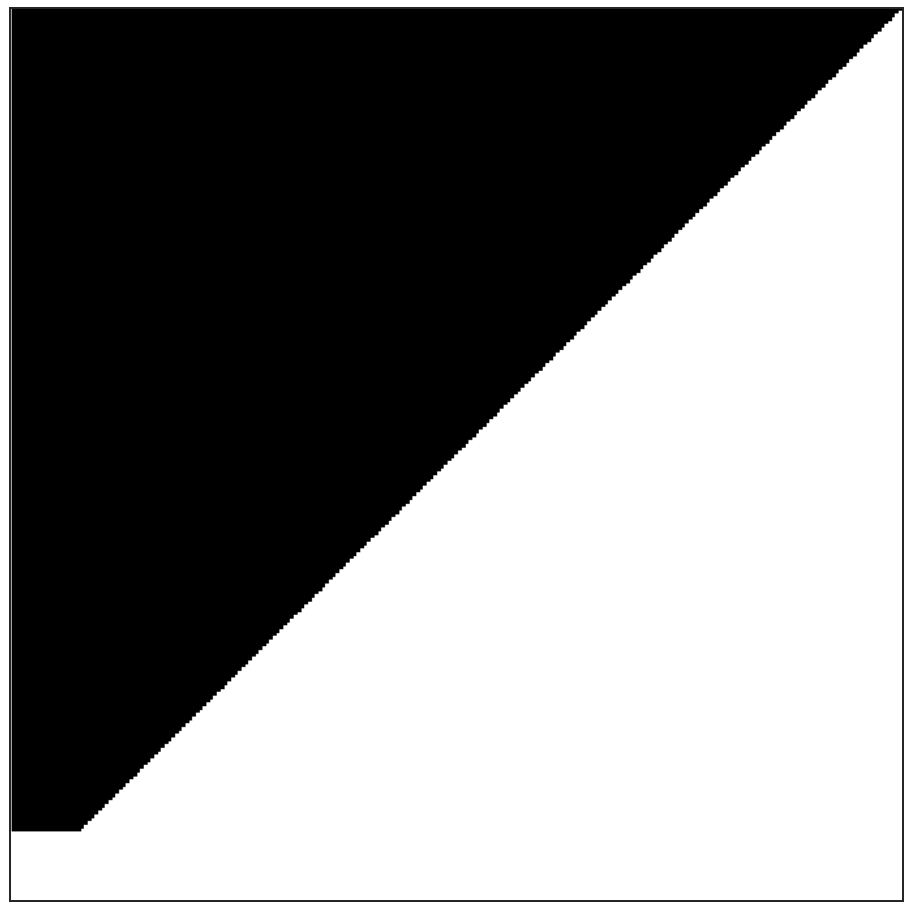}}
\put(-30,360){Two sides}
\put(-30,220){Right side}
\put(-30,80){Top side}
\put(-30,0){Noise level: }
\put(60,0){$0.01\%$}
\put(220,0){$5\%$}
\put(360,0){$25\%$}
\end{picture}
\caption{\label{fig:recfull}Reconstructions from one and two sides for different levels of noise $\delta_h \approx \delta_v$. The measurements are taken until time $T=3$ using $L_{\text{max}} = 256$ trigonometric masks on each side. The bottom-right black-and-white squares represent the generalised Fourier coefficients used for the reconstructions.}
\end{figure}

Here we simulated the wave equation on the unit square until time $T=3$ and took measurements using $L_{\text{max}} = 256$ trigonometric masks $\Phi_l$ on each side. The black-and-white squares at the bottom-right corner of each picture represent the generalised Fourier coefficients used for that reconstruction, as in Figure~\ref{fig:samp}: the white pixels represent the coefficients used. We call it the \textit{sampling square}.

The three reconstructions in the first row are obtained from measurements taken on two sides, the right and the top one, with increasing levels of noise. For these simulations we used all recovered coefficients $(f,\varphi_{n,k})_{L^2(\Omega)}$, where $n,k=1,\dots,256$. This is reflected in the sampling squares, which are completely white.

We note that the reconstructions have a high resolution even with very high levels of noise. One reason for this is the fact that we are extracting only the stable coefficients from each side. Another reason for such a sharp reconstruction is the fact that the main features of the phantom are jump singularities, which are preserved by the measurements and are not particularly affected by Gaussian noise, whose effect is only to blur the background.

In the second and third rows we show reconstructions with measurements taken on a single side. Here we used slightly more than half of the recovered coefficients, as explained in Section~\ref{sub:rec1}. More precisely, in addition to the stable half of the coefficients, we used those  with $(n,k) \in\{1,\dots,20\}^2$ in the case of $0.01\%$ noise, $(n,k) \in\{1,\dots,15\}^2$ for $5\%$ noise and $(n,k) \in\{1,\dots,8\}^2$ for $25\%$ noise. This can be seen by looking at the sampling square of each figure closely.
We clearly see how the wave front set of the phantom influences the reconstruction quality. The edges with normal vector parallel to the side of the measurements are poorly reconstructed, while the others are recovered also with high level of noise. This is a well-known issue in tomography.

\begin{figure}
\begin{picture}(300,120)
\put(60,20){\includegraphics[width=4.5cm]{images/topNoise0dot05214J256K256T3N21214p15.pdf}}
\put(240,20){\includegraphics[width=4.5cm]{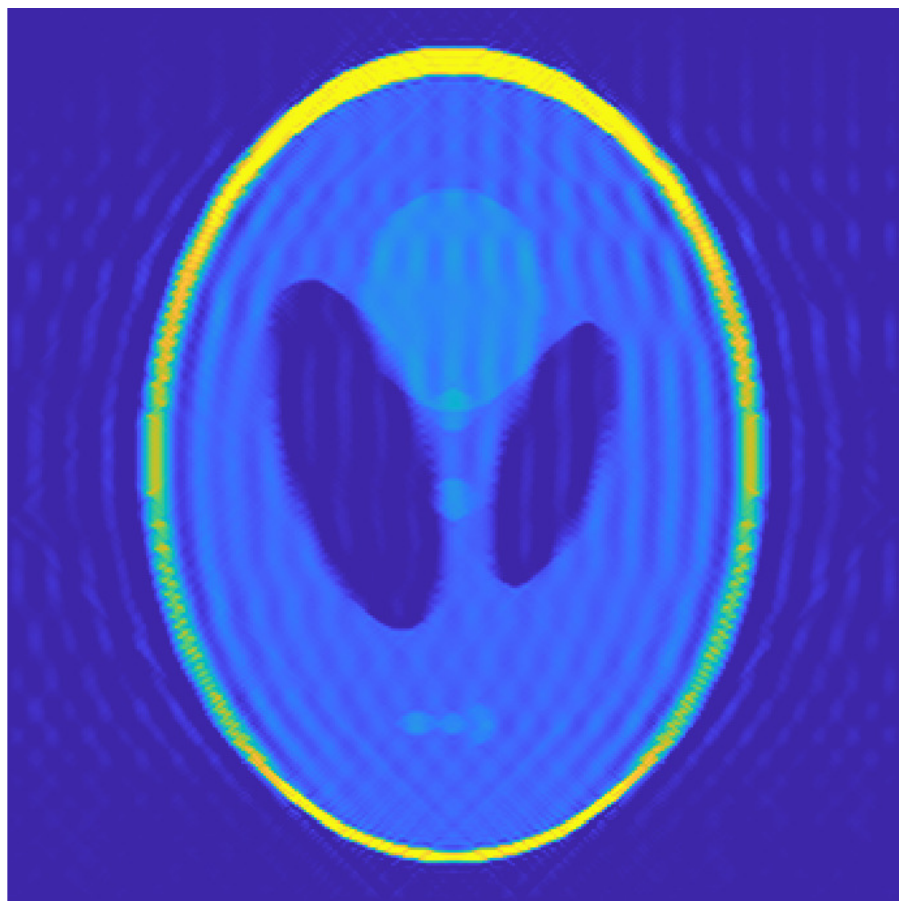}}
\put(148,14){\includegraphics[width=1.64cm]{images/schemeT15.pdf}}
\put(328,14){\includegraphics[width=1.64cm]{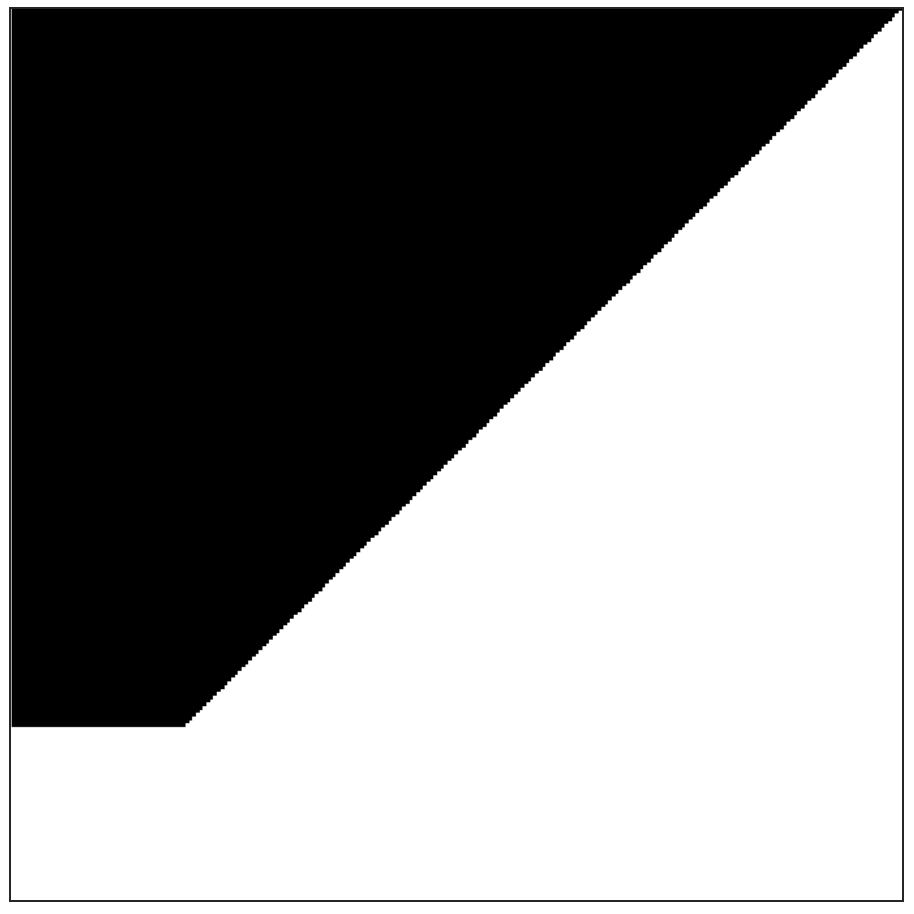}}
\put(0,0){Final time: }
\put(110,0){$ T = 3$}
\put(290,0){$T = 10$}
\end{picture}
\caption{\label{fig:recstab}Reconstructions from measurements with $5\%$ noise $\delta_h$, acquired from the top side and with different measurement times. The measurements are done until time $T=3$ on the left and until $T=10$ on the right. In both cases we use $256$ trigonometric masks and we are able to stably recover the coefficients $(f,\varphi_{n,k})_{L^2(\Omega)}$ with $k \leq n$. In the left reconstruction the coefficients with $n\leq k \leq 15$ are also stably computed, while in the right one we recover those up to $n\leq k \leq 50$.}
\end{figure}

In Figure~\ref{fig:recstab} we show two reconstructions from one side - the top one - with measurements taken on different time intervals. The measurements are also affected by $5\%$ additive Gaussian noise $\delta_h$. The final time for the left image is $T=3$, while on the right it is $T=10$. In both cases, we measure scalar products against trigonometric masks $\Phi_\jj$ with $\jj = 1,\dots,256$. As mentioned previously in Section~\ref{sub:rec1}, we are unable to stably recover all coefficients $(f,\varphi_{n,k})_{L^2(\Omega)}$ for $n,k = 1,\dots, 256$, but only a subset of them depending on the final time $T$ (through the parameter $\bar\jj$). We clearly see that a longer measurement time allows for a much sharper reconstruction. Namely, in the case $T=3$ we are able to recover coefficients with $(n,k) \in\{1,\dots,15\}^2$, while for $T=10$ we can push up to $(n,k) \in\{1,\dots,50\}^2$. In both cases we stably recover also the coefficients with $k \leq n$, as in Figure~\ref{fig:recfull}. Note that the artifacts appearing in the case $T=10$ are mostly due to the Gibbs phenomenon.  These reconstructions corroborate the theoretical findings obtained in Section~\ref{sub:2Doneside}.

\subsubsection{CS-PAT with TV regularisation}\label{sub:cs-pat-tv}

In Figure~\ref{fig:recCS} we present two different compressed sensing reconstructions. The aim of this comparison is to show that the specific subsampling pattern in the frequency domain for the coefficients of the DST arising from PAT measurements is comparable to other subsampling patterns, when it comes to solve the convex optimisation problem \eqref{eq:TVmin}.

In this figure, the coefficients of the DST are directly computed from the phantom, and not from the PAT measurements. In both reconstructions, we consider two-level subsampling schemes: for the first one, inspired by PAT measurements, we fully sample the frequencies for which at least one of the indices $n$ or $k$ is smaller than $10$ and then use a non-uniform log-sampling scheme (\cite{Alberti2017}) for the higher frequencies. For the second sampling scheme, first we fully sample the frequencies $(n,k) \in \{1,\dots,30\}^2$ and then use a quadratic sampling scheme for the higher frequencies. The first row shows the structured sampling pattern coming from the PAT model, while the second row shows the quadratic sampling pattern. In both cases we obtain a perfect reconstruction of the phantom. In order to achieve exact reconstruction, we had to sample $20\%$ of the frequencies in the first case, while in the second one only $8\%$ were enough. From the point of view of the masks $\{\Phi_l\}_l$, in the first row we only used $13\%$ out of the $L_{\text{max}}=256$ on each side. We called minimum energy solution the one obtained from the partial sum \eqref{eq:psum} by setting to zero the unknown coefficients.

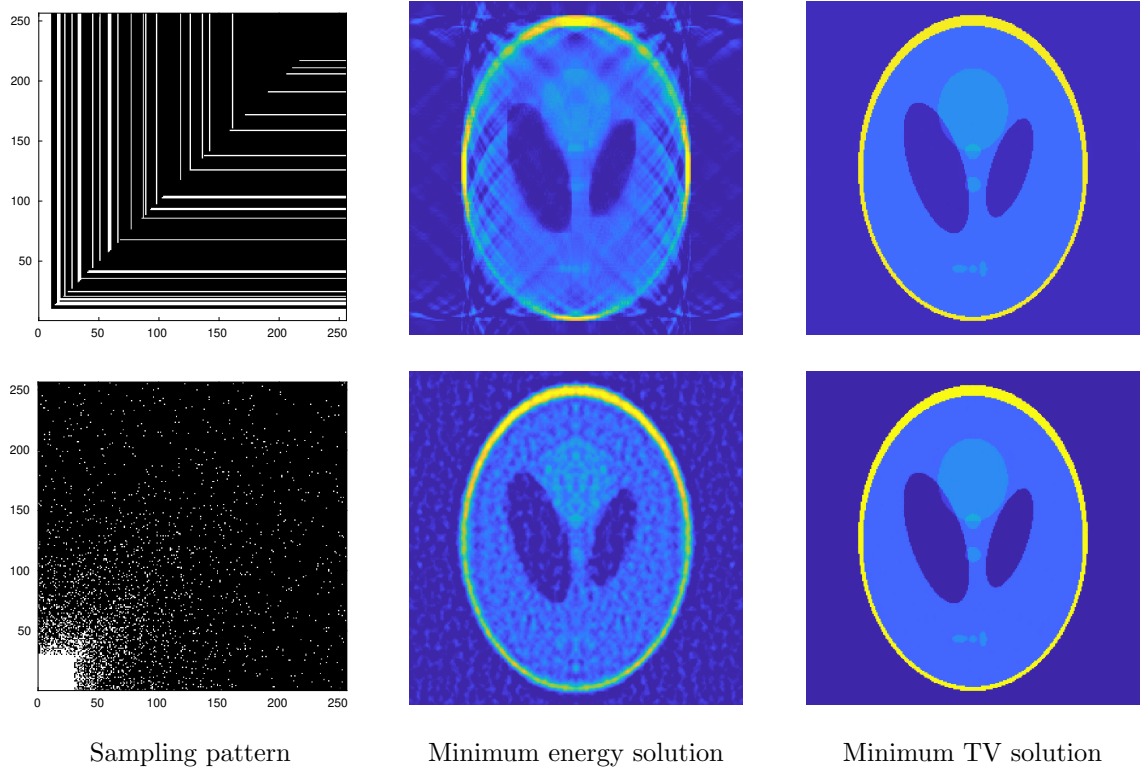
\begin{figure}
\begin{picture}(300,260)
\put(0,160){\includegraphics[width=4.5cm]{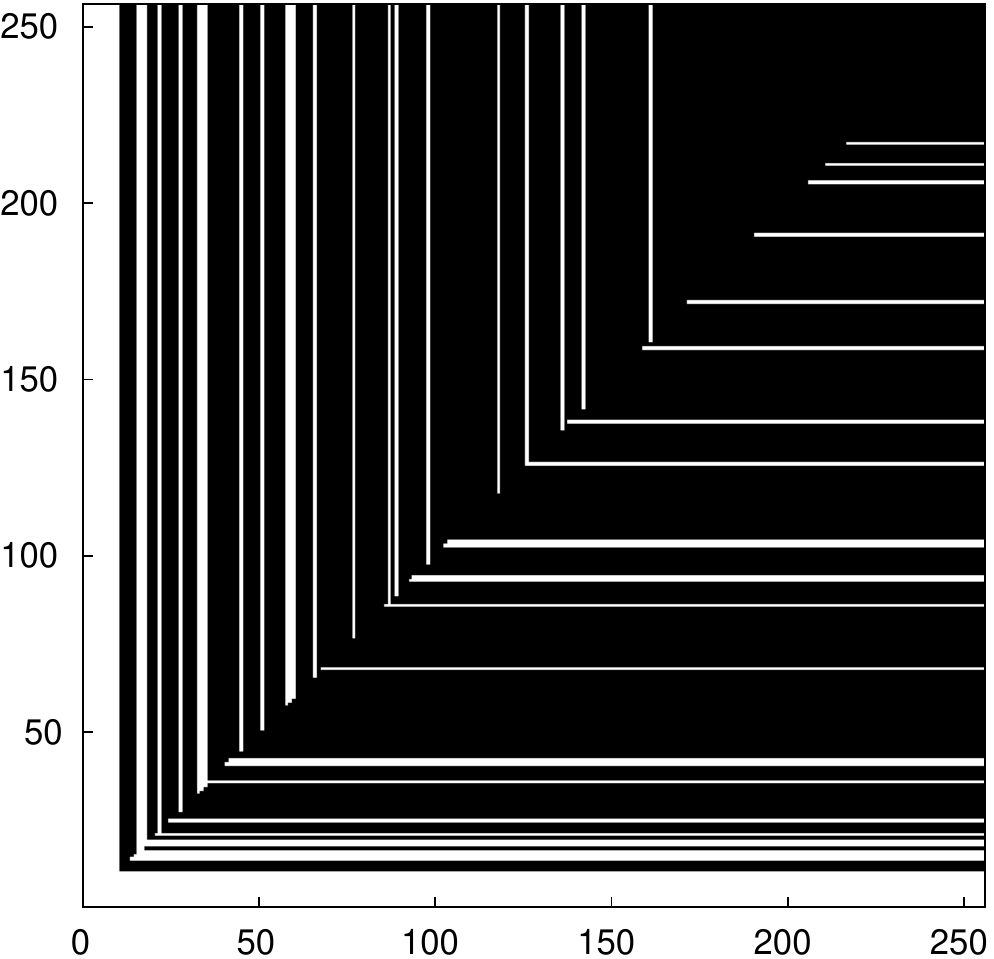}}
\put(150,160){\includegraphics[width=4.5cm]{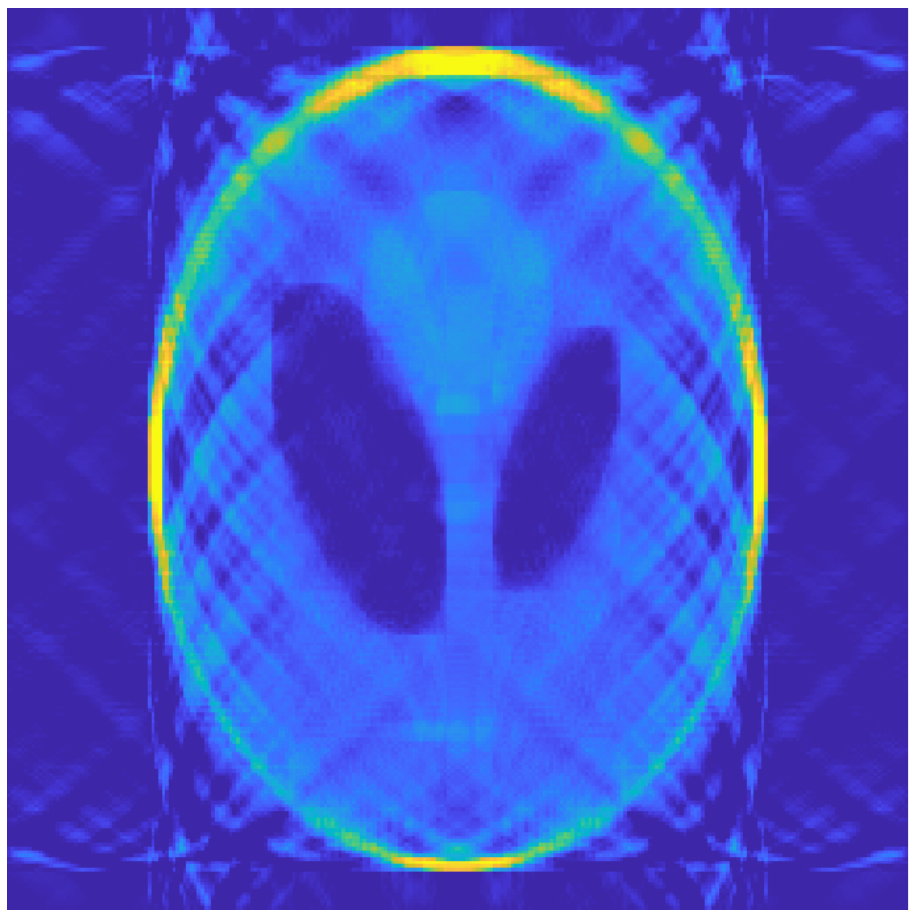}}
\put(300,160){\includegraphics[width=4.5cm]{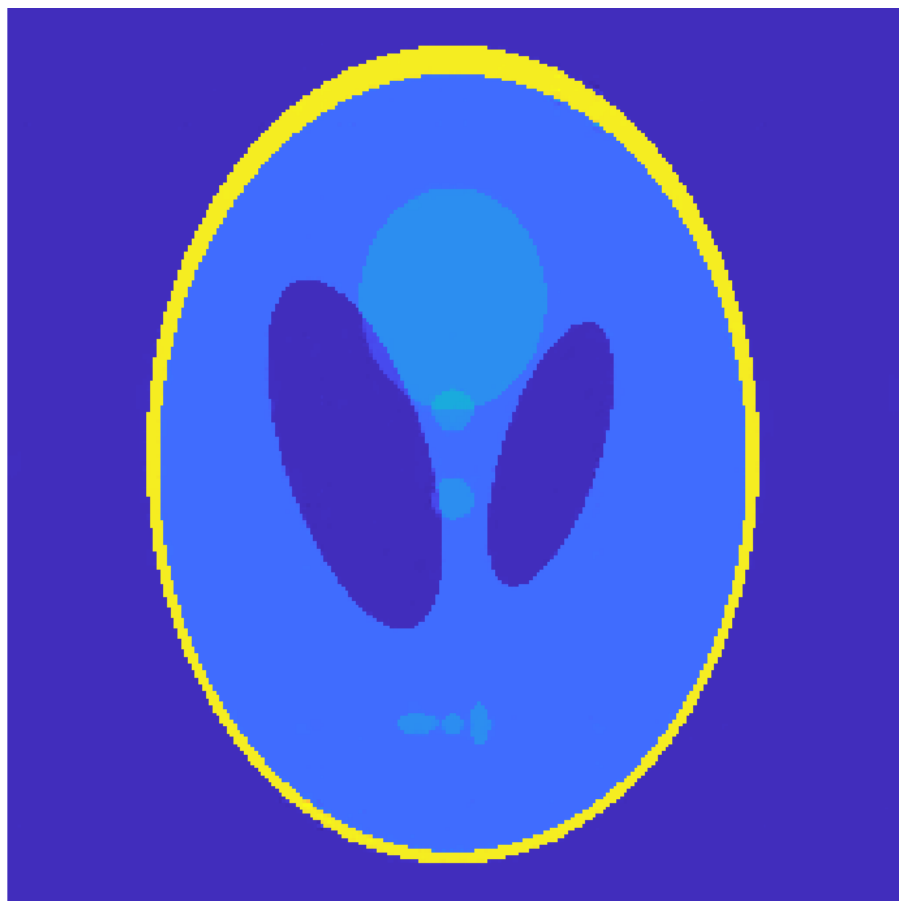}}
\put(0,20){\includegraphics[width=4.5cm]{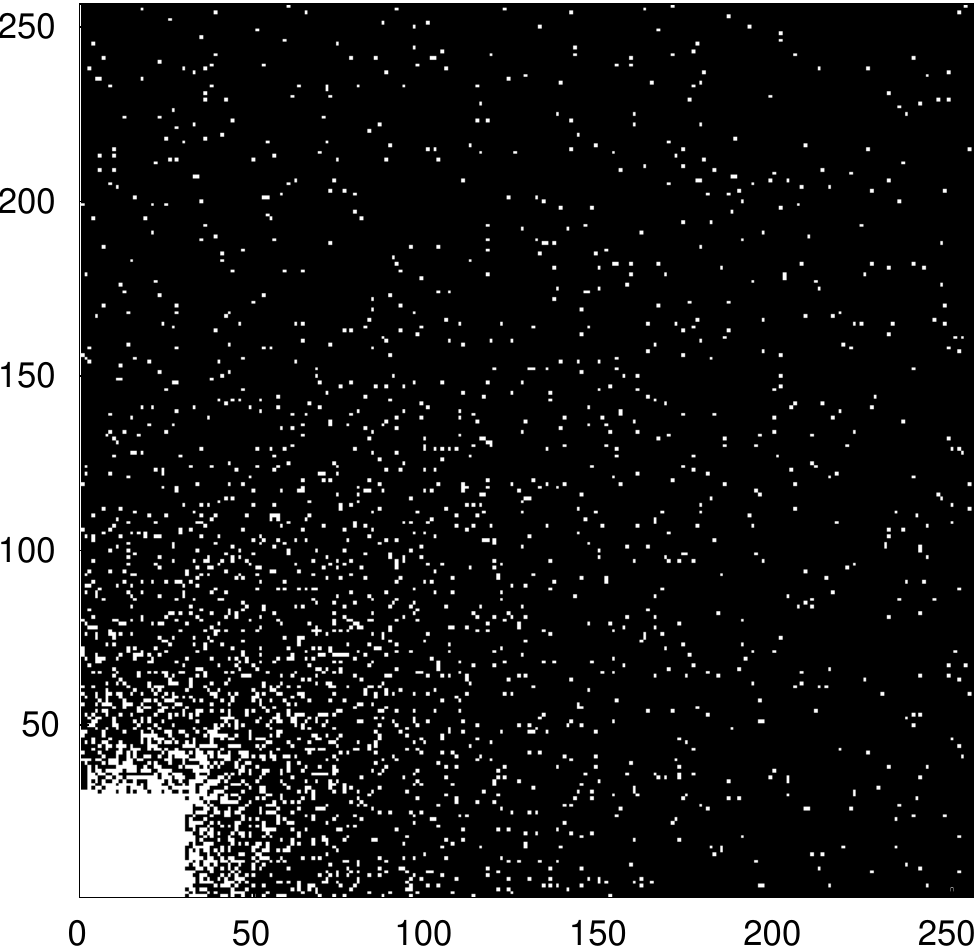}}
\put(150,20){\includegraphics[width=4.5cm]{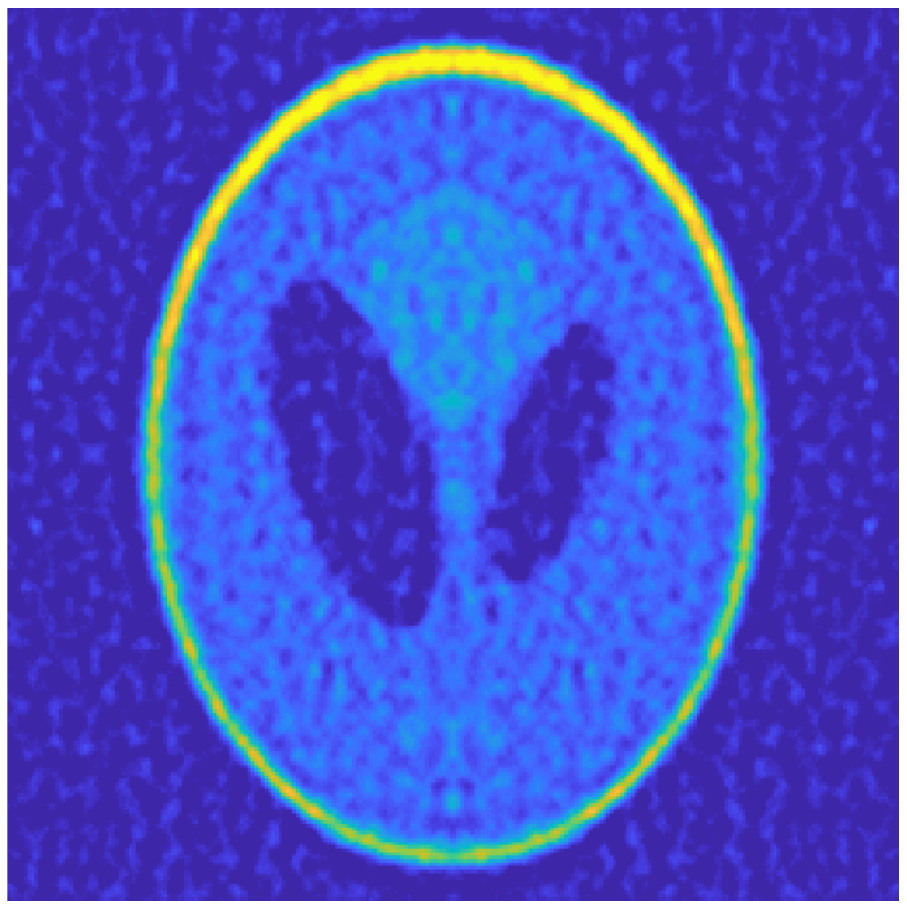}}
\put(300,20){\includegraphics[width=4.5cm]{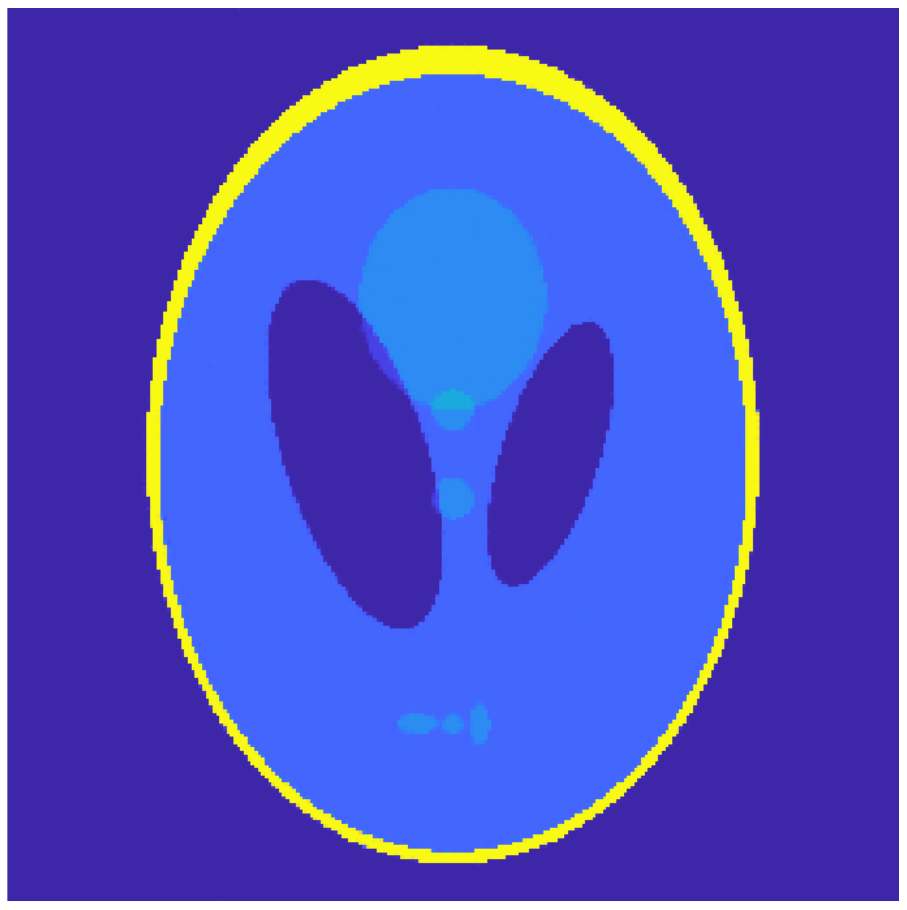}}
\put(30,0){Sampling pattern}
\put(158,0){Minimum energy solution}
\put(315,0){Minimum TV solution}
\end{picture}
\caption{\label{fig:recCS} Comparison of compressed sensing reconstructions. In the first row, the sampling pattern in the frequency domain for the coefficients of the DST comes from the PAT model, while in the second one we use a different sampling pattern. In both cases the coefficients are directly computed from the phantom and not by solving the wave equation. In the second column we show the reconstruction obtained by setting the unknown  coefficients to zero. The last column shows the solution obtained by solving the convex optimisation problem \eqref{eq:TVmin}. In both cases we obtain exact reconstruction. The first row uses $20\%$ of the DST coefficients, corresponding to $13\%$ of the masks $\{\Phi_l\}_l$, while the second row uses $8\%$ of them.}
\end{figure}

In Figure~\ref{fig:recCSPAT} we present compressed sensing reconstructions where the coefficients of the DST have been obtained from PAT measurements as explained in Section~\ref{sub:rec1}. Since the PAT measurements are an approximation of the DST, we had to increase the number of coefficients used in order to have high resolution reconstruction, resulting in a $25\%$ subsampling in the frequency domain - which corresponds to using $18\%$ of the masks $\{\Phi_l\}_l$ for each side. More precisely, we fully sample $(n,k)$ for $n = 1,\dots,16$, $k=1,\dots,256$ and $n = 1,\dots,256$, $k=1,\dots,16$. For higher frequencies we randomly select $30$ horizontal and vertical half-lines. Note that the reconstruction quality, both with and without noise, is comparable to the quality obtained using all coefficients as in Figure~\ref{fig:recfull}.

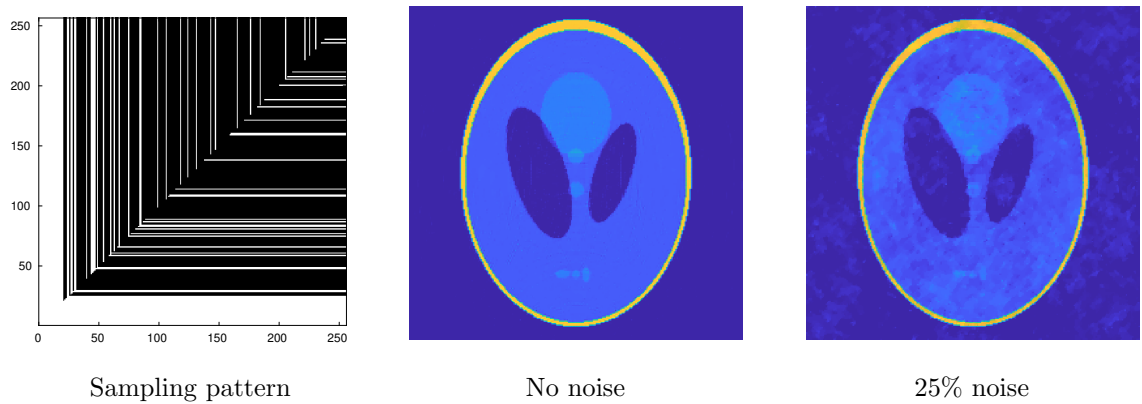
\begin{figure}
\begin{picture}(300,140)
\put(0,20){\includegraphics[width=4.5cm]{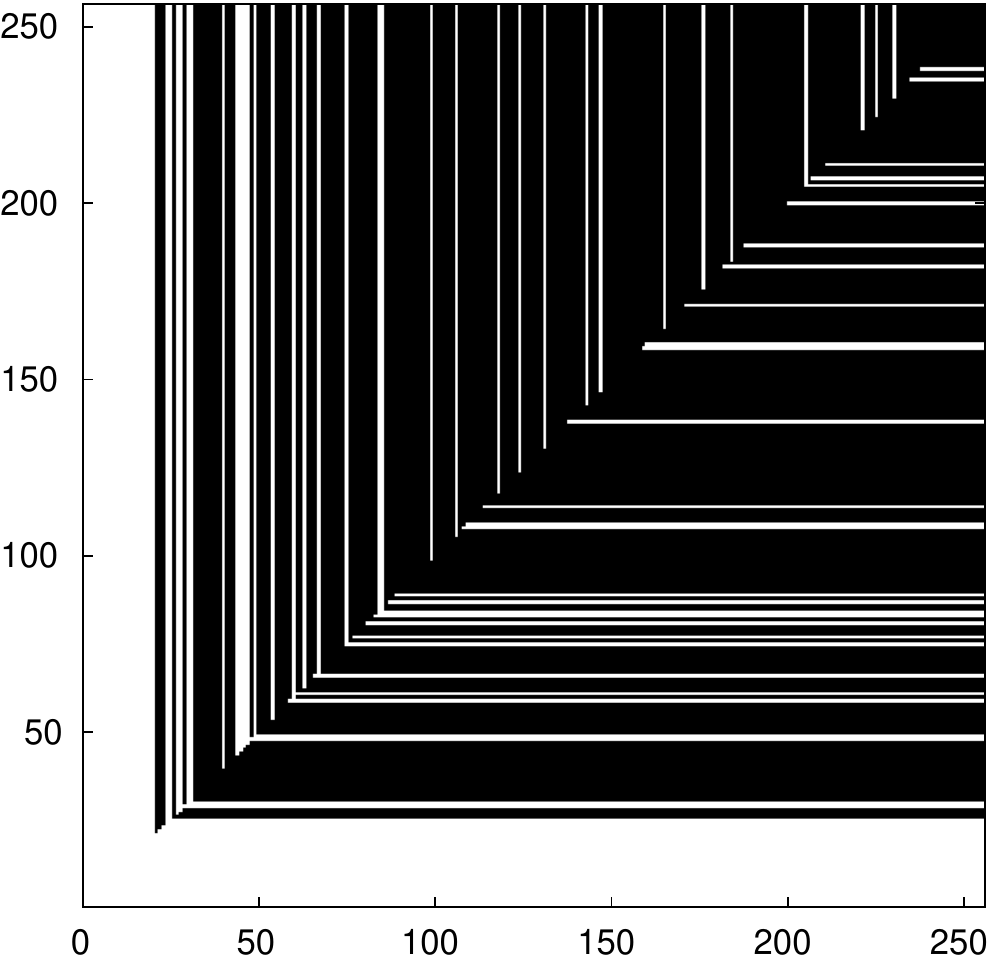}}
\put(150,20){\includegraphics[width=4.5cm]{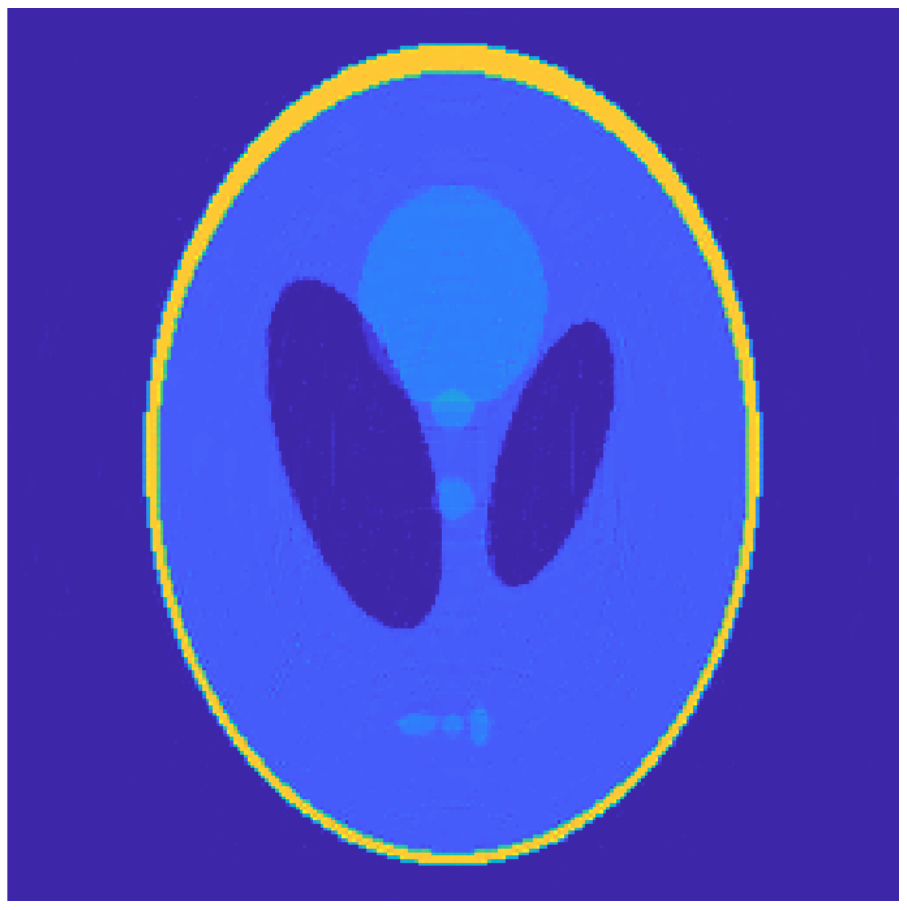}}
\put(300,20){\includegraphics[width=4.5cm]{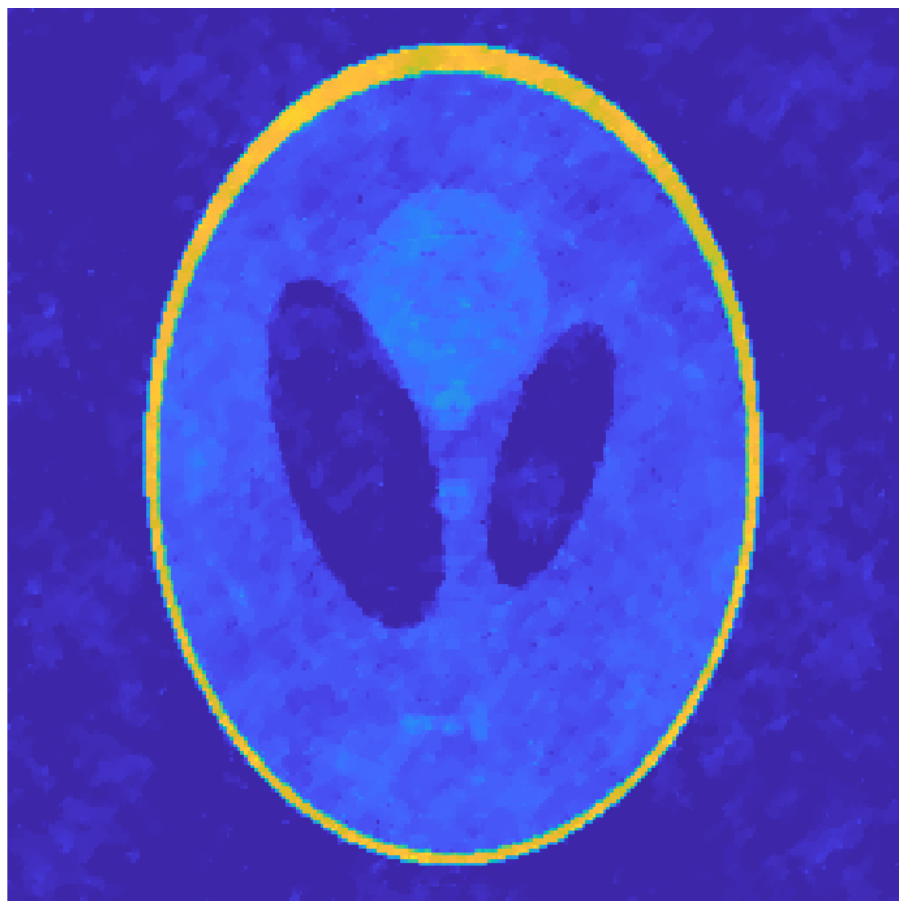}}
\put(30,0){Sampling pattern}
\put(195,0){No noise}
\put(342,0){$25\%$ noise}
\end{picture}
\caption{\label{fig:recCSPAT} Compressed sensing reconstruction from PAT measurements. Left: $25\%$ subsampling pattern in the frequency domain, corresponding to $18\%$ of the masks $\{\Phi_l\}_l$. Center: reconstruction via TV minimisation from noiseless data. Right: reconstruction via TV minimisation with $25\%$ additive Gaussian noise in the PAT data.}
\end{figure}

\section{Conclusions}
In this paper, we have shown for the first time that the approaches based on CS for PAT may be theoretically justified by reducing the inverse problem to a CS problem for undersampled generalised Fourier measurements. This is achieved by a suitable inversion in the time domain applied to the compressed data. The results in this paper are preliminary and open the way for several future directions. We discuss three open questions below.
\begin{itemize}
\item The approach based on this work makes use of the explicit geometries of the domains considered, both in the free-space model and in the bounded-domain model. Even if these domains are of practical relevance and many works on PAT have considered special geometries, it would be interesting to investigate whether it is possible to extend the results to more general domains.
\item The CS problems for undersampled Fourier measurements that arise are peculiar in the sense that the subsampling patterns are not fully random but have a particular structure. In a particular case, these were studied in Section~\ref{sec:CS}. In finite dimension, CS results with structured sampling patterns have been studied in \cite{boyer2019compressed,adcockbrugia2020}. However, we are not aware of any work in our structured infinite-dimensional setting. A thorough theoretical and numerical analysis of these CS problems goes beyond the scopes of this work, and is left for future research.
\item We provided numerical evidence for the uniform stability of the reconstruction problems in the case of the 2D square with measurements on two adjacent sides. However, we were not able to prove this. The rigorous proof of the stability estimate given in \eqref{eq:conjecture} is an open question and we leave it for future work.
\end{itemize}

\section*{Acknowledgments}
This work has been carried out at the Machine Learning Genoa (MaLGa) center, Universit\`a di Genova (IT). This material is based upon work supported by the Air Force Office of Scientific Research under award number FA8655-20-1-7027. GSA and MS are members of the ``Gruppo Nazionale per l'Analisi Matematica, la Probabilit\`a e le loro Applicazioni'' (GNAMPA), of the ``Istituto Nazionale per l'Alta Matematica'' (INdAM). GSA is supported by a UniGe starting grant ``curiosity driven''. PC is member of the ``Cantab Capital Institute for the Mathematics of Information" (CCIMI).

\bibliographystyle{plain}
\bibliography{CS_PAT}

\appendix

\section{Proof of   Theorem~\ref{thm:3Ddisk_free}}\label{sec:3Dball|}

The proof is based on the following lemma.
\begin{Lemma}
Let $\xi=|\xi| e_\xi \in \mathbb{R}^3$. For every $l \in \N$ and  $m \in \{-l,...,l\}$ the following identity holds:
\begin{equation}\label{eq:lemma_3d}
\int_{\mathbb{S}^{2}}Y_l^m(x) e^{2\pi i \rho x \cdot \xi} d\sigma(x) = 4\pi i^{l-2m} \, j_l(2\pi\rho|\xi|) \, Y_l^m(e_\xi) .
\end{equation}
\end{Lemma}
\begin{proof}
 Write $e_\xi = ( \sin\theta_\xi \cos \varphi_\xi, \sin\theta_\xi \sin\varphi_\xi, \cos\theta_\xi) \in \mathbb{S}^2$. By direct calculation:
\begin{equation}\notag
\begin{aligned}
& \int_{\mathbb{S}^{2}} Y_l^m(x) e^{2\pi i \rho x \cdot \xi} d\sigma(x)  \\
& =  \int_0^\pi \int_0^{2\pi} P_l^m(\cos\theta)e^{im \varphi} e^{2\pi i \rho (\xi_1  \sin\theta \cos \varphi + \xi_2 \sin\theta \sin\varphi + \xi_3 \cos\theta)}\sin\theta \, d\varphi \, d\theta  \\
& = \int_0^\pi  P_l^m(\cos\theta) e^{2\pi i \rho \xi_3 \cos\theta} \sin\theta \Big(\int_0^{2\pi} e^{2\pi i \rho \sin\theta (\xi_1, \xi_2) \cdot (\cos\varphi, \sin\varphi) } e^{im \varphi} \, d\varphi \Big) \, d\theta.
\end{aligned}
\end{equation}
Using Lemma \ref{lem:2d}, and noticing that the argument of the vector $(\xi_1,\xi_2)$ is $\phi_\xi$, we can rewrite this integral as
\begin{multline*}
 \int_0^\pi  P_l^m(\cos\theta) e^{2\pi i \rho \xi_3 \cos\theta} \sin\theta \Big(2\pi i^{-m} e^{im \phi_\xi} J_m\big(2\pi \rho \sin\theta |(\xi_1,\xi_2)| \big) \Big) \, d\theta  \\
 = 2 \pi i^{-m}e^{im \varphi_\xi} \int_0^\pi P_l^m(\cos\theta) J_m(2\pi\rho |\xi| \sin\theta \sin\theta_\xi) e^{2\pi i \rho |\xi| \cos\theta_\xi \cos\theta} \sin\theta \, d\theta.
\end{multline*}
Thanks to the main result in \cite{Neves2006}, the previous integral can be expressed in terms of the spherical Bessel function $j_l$:
\begin{equation}\notag
\begin{split}
\int_{\mathbb{S}^{2}} Y_l^m(x) e^{2\pi i \rho x \cdot \xi} d\sigma(x)&=  2\pi i ^{-m} e^{im \varphi_\xi} \Big(2i^{l-m} P^m_l(\cos\theta_\xi)j_l(2\pi\rho|\xi|)\Big) \\
& = 4\pi i^{l-2m} e^{im \varphi_\xi} P^m_l(\cos\theta_\xi)j_l(2\pi\rho|\xi|) \\
& = 4\pi i^{l-2m} \, j_l(2\pi\rho|\xi|) \, Y_l^m(\theta_\xi, \phi_\xi).
\end{split}
\end{equation}
This concludes the proof.
\end{proof}

\begin{proof}[Proof of Theorem~\ref{thm:3Ddisk_free}]
The second part of \eqref{eq:twoorone3} is an immediate consequence of Theorem~\ref{thm:sine}, since $\tilde{\psi}_{l,m,n}$ is a Dirichlet eigenfunction of the Laplacian in the disk with eigenvalue $j_{l+\frac{1}{2},n}^2$ and
\[
\partial_\nu \tilde{\psi}_{l,m,n}=j_{l+\frac{1}{2},n}\, j'_l(j_{l+\frac{1}{2},n})\, \Phi_{l,m}.
\]
It remains to prove the first identity of \eqref{eq:twoorone3}.

By Theorem~\ref{thm:moment} we have
\begin{equation}\label{eq:general_moment23D}
(f, {\psi}_{l,m, \rho})_{L^2(\Omega)} =  \hat g_{l,m}(\rho),
\end{equation}
where, by \eqref{eq:psi1}, we can write for $y\in\Omega$
\begin{equation}\label{eq:2d_disk}
\begin{aligned}
\psi_{l,m,\rho}(y) & =  \frac{\rho^2}{2} \int_{\mathbb{S}^2} \Phi_{l,m}(x)\Bigg( \int_{\mathbb{S}^{2}} e^{2\pi i \rho (x-y) \cdot \omega} d\sigma(\omega) \Bigg) d\sigma(x) \\
& =\frac{\rho^2}{2} \int_{\mathbb{S}^2}e^{-2\pi i \rho y \cdot \omega}\Bigg( \int_{\mathbb{S}^{2}} \Phi_{l,m}(x) e^{2\pi i \rho x \cdot \omega} d\sigma(x) \Bigg) d\sigma(\omega) . \\
\end{aligned}
\end{equation}
We can now explicitly determine $\psi_{l,m,\rho}$. Write $y=|y|e_y \in B_{\R^3}(0,1)$, then by using formula \eqref{eq:lemma_3d} twice we have:
\begin{equation}\notag
\begin{aligned}
\psi_{l,m,\rho}(y) &=   \frac{\rho^2}{2}  \int_{\mathbb{S}^2}e^{-2\pi i \rho y \cdot \omega}\Bigg( \int_{\mathbb{S}^{2}} Y_l^m(x) e^{2\pi i \rho x \cdot \omega} d\sigma(x) \Bigg) d\sigma(\omega)  \\
& = 2\rho^2 \pi i^{l-2m} \ j_l(2\pi\rho) \int_{\mathbb{S}^2} e^{-2\pi i \rho y \cdot \omega} Y_l^m(\omega) d\sigma(\omega)  \\
& = 8\big(\rho\pi i^{l-2m}\big)^2 \ j_l(2\pi\rho) \ j_l(2\pi \rho|y|)\,  \underbrace{Y^m_l(-e_y)}_{= (-1)^l Y^m_l(e_y)}  \\
& = 8 \pi^2\rho^2  j_l(2\pi\rho) \ j_l(2\pi \rho|y|)\,  Y^m_l(e_y).
\end{aligned}
\end{equation}
By \eqref{eq:general_moment23D} and L'H\^opital's rule we have
\[
\begin{split}
\frac{(\hat g_{l,m})'(\rho_{l,m})}{4\pi j_{l+\frac{1}{2},n}^2 \, j_l'(j_{l+\frac{1}{2},n})} &=
 \lim_{\rho \to \rho_{l,n}}\frac{1}{8\pi^2}\frac{\hat g_{l,m}(\rho)}{\rho^2 \, j_l(2\pi\rho)}\\ &
 = \lim_{\rho \to \rho_{l,n}} (f,y\mapsto j_l(2\pi \rho|y|)\,  Y^m_l(e_y))_{L^2(\Omega)} \\ &= (f,\tilde\psi_{l,m,n})_{L^2(\Omega)},
 \end{split}
\]
since $2\pi \rho_{l,n}=j_{l+\frac{1}{2},n}$.
\end{proof}

\section{Riesz Bases}\label{sec:riesz}

In this section, we review some of the basic notions related to Riesz bases of exponentials; for further details, the reader is referred to \cite{Young1981,Christensen2016}. 

\subsection{Riesz Bases in Hilbert spaces}
We start by introducing the concept of Riesz sequence of a Hilbert space, which generalises that of linearly independent set.

\begin{Definition} Let $\H$ be a Hilbert space and $\{f_k\}_{k \in \N} \subseteq \H$.
\begin{enumerate}
\item $\{f_k\}_{k \in \N}$ is a \textit{Riesz sequence} if there exist constants $A,B>0$ such that
\begin{equation}\label{eq:riesz}
A \, \sum_{k \in \N}|c_k|^2 \leq \Big\Vert \sum_{k \in \N}c_k f_k\Big\Vert^2 \leq B \, \sum_{k \in \N}|c_k|^2
\end{equation}
for every finite sequence $\{c_k\}_{k} \subseteq \mathbb{C}^\N$.
\item $\{f_k\}_{k \in \N}$ is a \textit{Riesz basis} if it is a  Riesz sequence and it is \textit{complete}, namely $\overline{\operatorname{span}\{f_k\}}_k = \H$.
\item The optimal bounds such that the Riesz condition \eqref{eq:riesz}  
holds are called \textit{Riesz bounds}.
\end{enumerate}
\end{Definition}
In particular, a Riesz sequence is a Riesz basis for its closed span. If $\{f_k\}_k$ is a Riesz basis, then every element $f \in \H$ can be expressed as a linear combination 
\begin{equation}\label{eq:expansion}
f = \sum_{k\in\N} c_k f_k
\end{equation} 
in a unique way.
If $\{f_k\}_k$ is a Riesz sequence,  its \textit{synthesis operator} is given by
\begin{equation}\notag
T\colon \ell^2(\N) \longrightarrow \H, \qquad \{c_k\} \longmapsto \sum_{k \in \N} c_k f_k.
\end{equation}
The adjoint of the synthesis operator is  the \textit{analysis operator}:
\begin{equation}\notag
T^*\colon \H \longrightarrow \ell^2(\N), \qquad f \longmapsto \{\langle f, f_k\rangle\}_{k \in \N}.
\end{equation}
By using \eqref{eq:riesz}, it can be easily proven that
\begin{equation}\notag
\norm{T}=\norm{T^*} \leq B^{\frac{1}{2}},\qquad \norm{T^{-1}}\le A^{-\frac{1}{2}}.
\end{equation}
The synthesis operator is central in the study of basis-like properties of the set $\{f_k\}_k$. In fact, its surjectivity is related to the possibility of expanding elements in $\H$ as a combination of the $\{f_k\}_k$, while its injectivity is linked to the uniqueness of such expansions.

We can compose the synthesis and analysis operators $T$ and $T^*$  to obtain the \textit{frame operator}  $S:=T T^*$
\begin{equation}\notag
S= T T^*\colon \H \longrightarrow \H, \quad f \longmapsto \sum_{n \in \N} \langle f, f_k\rangle f_k,
\end{equation}
which may be used to recover the unique coefficients $\{c_k\}_k$ in the expansion \eqref{eq:expansion}.

\begin{Proposition}[\cite{Christensen2016}]\label{prop:rec}
Let $\{f_k\}_{k \in \N} \subseteq \H$ be a Riesz sequence with frame operator $S$ and let 
\begin{equation}
f = \sum_{k \in \N} c_k \, f_k \in \overline{\operatorname{span}\{f_k\}_k}
\end{equation}
be an arbitrary element in $\overline{\operatorname{span}\{f_k\}_k}$. Then
\begin{equation}
c_k = \langle f, S^{-1}f_k \rangle, \qquad k \in \N,
\end{equation}
where $\{S^{-1}f_k\}_k$ is called the \textit{bi-orthogonal sequence} of $\{f_k\}_{k \in \N}$ in $\overline{\operatorname{span}\{f_k\}_k}$.
\end{Proposition}
Therefore, in order to recover the coefficients of an element $f \in \overline{\operatorname{span}\{f_k\}_k}$ with respect to a Riesz sequence $\{f_k\}_{k \in \N}$, it is sufficient to find the inverse of the frame operator $S$ associated to the Riesz sequence. This goal can be achieved in practice by approximating the frame operator $S$ with a sequence of finite-dimensional matrices of increasing size.

\subsection{Riesz bases of exponentials}
The theory of Fourier series guarantees that $\{\frac{1}{\sqrt{2\pi}}e^{ik \, \cdot}\}_{k \in \Z}$ is an orthonormal Basis for $L^2([0,2\pi])$. As a consequence, $\{e^{ik \, \cdot}\}_{k \in \Z}$ is a Riesz basis in $L^2([0,2\pi])$ with Riesz bounds $A=B=2\pi$. In this section, we will focus on Riesz bases which are similar to the Fourier basis, but allow for non-uniformly spaced frequencies.

\begin{Definition}
A Riesz basis for $L^2(I)$ of the form $\{e^{i\lambda_k\cdot}\}_{k \in \Z}$, where $I \subseteq \R$ is a bounded interval and $\{\lambda_k\}_{k \in \Z} \subseteq \R$ is a real sequence, is called a \textit{Riesz basis of exponentials}. An expansion of the form
\begin{equation}\label{eq:nonharmonic_series}
f(t) = \sum_{k \in \Z} c_k e^{i \lambda_k t}
\end{equation}
in the $L^2(I)$ sense, for $f \in L^2(I)$, is called \textit{nonharmonic Fourier series}.
\end{Definition}

We will be interested in recovering the coefficients $\{c_k\}_{k \in \Z}$ which appear in the expansion \eqref{eq:nonharmonic_series}. For this problem to make sense, such coefficients must be unique, which happens precisely when $\{e^{i\lambda_k\cdot}\}_{k \in \Z}$ is a Riesz sequence.\par
In general, the upper and lower Riesz conditions in \eqref{eq:riesz} are unrelated. In the context of families of exponentials, however, the situation is different: if the sequence $\{\lambda_k\}_{k\in\Z}$ consists of distinct points, the existence of a lower Riesz bound for $\{e^{i\lambda_k\cdot}\}_{k \in \Z}$ in $L^2(-\pi,\pi)$ implies the existence of the upper bound. In other words, the lower condition is enough to guarantee that $\{e^{i\lambda_k\cdot}\}_{k \in \Z}$ is a Riesz sequence. This is the content of the following theorem:

\begin{Theorem}[{\cite[Theorem~9.8.5]{Christensen2016}}]\label{thm:lower}
Take $\{\lambda_k\}_{k \in \Z} \subseteq \R$. Suppose that   there exists a constant $A > 0$ such that
\begin{equation}\label{eq:lower_riesz}
A \, \sum_{k \in \Z}|c_k|^2 \leq \Big\Vert \sum_{k \in \Z}c_ke^{i\lambda_k\cdot} \Big\Vert^2_{L^2(I)}
\end{equation}
for all finite scalar sequences $\{c_k\}_{k \in \Z}$. Then $\{e^{i\lambda_k\cdot}\}_{k \in \Z}$ is a Riesz sequence in $L^2(I)$. 
\end{Theorem}

The following criterion gives a sufficient condition for inequality \eqref{eq:lower_riesz} to hold.

\begin{Theorem}[\cite{Young1981}, Chapter 4, Section 3, Theorem 3]\label{thm:inf}
 If $\{\lambda_k\}_{k \in \Z}$ is an increasing sequence of real numbers such that
\begin{equation}\notag
\gamma:=\inf_{k \in \mathbb{Z}}(\lambda_{k+1}-\lambda_k) > \frac{\pi}{T},
\end{equation}
then $\{e^{i\lambda_k \cdot}\}_{k\in\Z}$ satisfies the lower Riesz inequality \eqref{eq:lower_riesz} in $L^2([-T, T])$ with lower bound 
\begin{equation}\notag
A = \frac{2}{\pi}\Big(1-\Big(\frac{\pi}{T\gamma}\Big)^2\Big) .
\end{equation}
\end{Theorem}

The following consequence of this result regarding families of cosines is readily derived.
\begin{Corollary}\label{cor:cos}
Take $T>0$.
\begin{enumerate}
\item
 If $\{\lambda_k\}_{k \in \N}$ is an increasing sequence of non-negative numbers such that $\lambda_0=0$ and
\begin{equation}\notag
\gamma:=\inf_{k \in \mathbb{N}}(\lambda_{k+1}-\lambda_k) > \frac{\pi}{T},
\end{equation}
then $\{\cos(\lambda_k \cdot)\}_{k\in\N}$ is a Riesz sequence in $L^2([0, T])$ with lower bound 
\begin{equation}\notag
A = \frac{1}{2\pi}\Big(1-\Big(\frac{\pi}{T\gamma}\Big)^2\Big) .
\end{equation}
\item
 If $\{\lambda_k\}_{k \in \N_+}$ is an increasing sequence of positive numbers such that
\begin{equation}\notag
\gamma:=\min\bigl(\,\inf_{k \in \mathbb{N_+}}(\lambda_{k+1}-\lambda_k),2\lambda_1\bigr) > \frac{\pi}{T},
\end{equation}
then $\{\cos(\lambda_k \cdot)\}_{k\in\N_+}$ is a Riesz sequence in $L^2([0, T])$ with lower bound 
\begin{equation}\notag
A = \frac{1}{2\pi}\Big(1-\Big(\frac{\pi}{T\gamma}\Big)^2\Big) .
\end{equation}
\end{enumerate}
\end{Corollary}
\begin{proof}
1. Let $\{c_k\}_{k\in\N}$ be a finite sequence. 
Setting $\lambda_{-k}=-\lambda_k$ and $c_{-k}=c_k$ for $k\ge 1$ we have
\begin{equation}\label{eq:cos_e}
\sum_{k\in\N}c_k \cos(\lambda_k t)
=\sum_{k\in\N}c_k \,\frac{e^{i\lambda_k t}+e^{-i\lambda_k t}}{2}
=c_0+ \sum_{k\in\Z^*} \frac{c_k}{2} \,e^{i\lambda_k t}.
\end{equation}
By construction we have
\begin{equation}\notag
\inf_{k \in \mathbb{Z}}(\lambda_{k+1}-\lambda_k) =\gamma > \frac{\pi}{T}.
\end{equation}
Therefore, by \eqref{eq:cos_e} and Theorem~\ref{thm:inf} we have
\[
\begin{split}
 \Big\Vert \sum_{k \in \N}c_k \cos(\lambda_k \cdot) \Big\Vert^2_{L^2([0,T])}
 &=\frac12  \Big\Vert c_0+ \sum_{k\in\Z^*} \frac{c_k}{2} \,e^{i\lambda_k \cdot} \Big\Vert^2_{L^2([-T,T])}\\
 &\ge \frac{4A}{2} \left(|c_0|^2+ \sum_{k \in \Z^*}|\frac{c_k}{2}|^2\right)\\
&= 2A \left(|c_0|^2+ \frac12\sum_{k \in \N_+}|c_k|^2\right)\\
&\ge A\sum_{k \in \N}|c_k|^2,
\end{split}
\]
This shows the lower bound. The upper bound is an immediate consequence of \eqref{eq:cos_e} and of Theorem~\ref{thm:lower}.

2. Let $\{c_k\}_{k\in\N_+}$ be a finite sequence.  
Setting $\lambda_{-k}=-\lambda_k$ and $c_{-k}=c_k$ for $k\ge 1$ we have
\begin{equation}\label{eq:cos_e2}
\sum_{k\in\N_+}c_k \cos(\lambda_k t)
=\sum_{k\in\N_+}c_k \,\frac{e^{i\lambda_k t}+e^{-i\lambda_k t}}{2}
= \sum_{k\in\Z^*} \frac{c_k}{2} \,e^{i\lambda_k t}.
\end{equation}
Note that the new sequence of frequencies is given by $\{\lambda_k:k\in\Z^*\}=\{\dots,-\lambda_2,-\lambda_1,\lambda_1,\lambda_2,\dots\}$, and so its minimum distance is given by $\gamma$.
Therefore, arguing as above, by \eqref{eq:cos_e2} and Theorem~\ref{thm:inf} we have
\[
 \Big\Vert \sum_{k \in \N_+}c_k \cos(\lambda_k \cdot) \Big\Vert^2_{L^2([0,T])}
=\frac12  \Big\Vert \sum_{k\in\Z^*} \frac{c_k}{2} \,e^{i\lambda_k \cdot} \Big\Vert^2_{L^2([-T,T])}
 \ge  \frac{4A}{2}   \sum_{k \in \Z^*}|\frac{c_k}{2}|^2
= A  \sum_{k \in \N_+}|c_k|^2.
\]
This shows the lower bound. The upper bound is an immediate consequence of \eqref{eq:cos_e2} and of Theorem~\ref{thm:lower}
\end{proof}

The next result deals again with families of cosines. For these systems to be a Riesz basis, it is sufficient to check the behaviour of the tail of $\{\lambda_k\}_{k \in \N}$.

\begin{Theorem}[{\cite[Lemma~4]{He2001}}]\label{lem:cos}
Let $\{\lambda_n\}_{n \in \N}$, $\{\mu_n\}_{n \in \N} \subseteq [0, + \infty)$ be sequences of nonnegative distinct real numbers ($\lambda_m \neq \lambda_n$ and  $\mu_m \neq \mu_n$ for $m \neq n$) such that
\begin{equation}\notag
\sum_{n \in \N} (\lambda_n - \mu_n)^2 < +\infty.
\end{equation}
Then $\{\cos(\lambda_n \cdot)\}_{n \in \N}$ is a Riesz basis in $L^2([0,1])$ if and only if $\{\cos(\mu_n \cdot)\}_{n \in \N}$ is a Riesz basis in $L^2([0,1])$.
\end{Theorem}

\end{document}